\documentclass[11pt,reqno]{amsart}

\usepackage{geometry}
\usepackage{amssymb}
\usepackage{mathtools}
\usepackage{booktabs}
\usepackage{placeins}
\usepackage[final,protrusion=true,expansion=true]{microtype}
\usepackage[dvipsnames]{xcolor}
\usepackage[colorlinks=true,
            citecolor=red,
            linkcolor=blue,
            urlcolor=red,
            bookmarksnumbered=true,
            bookmarksopen=true]{hyperref}
\hypersetup{
  pdftitle={Disproof of the Yau--Tian--Donaldson conjecture},
  pdfauthor={Jihao Liu},
  pdfkeywords={Yau--Tian--Donaldson conjecture; K-polystability; uniform K-stability; cscK metric; test configuration; projective bundle}
}

\numberwithin{equation}{section}
\DeclareMathOperator{\Aut}{Aut}
\DeclareMathOperator{\DF}{DF}
\DeclareMathOperator{\End}{End}
\DeclareMathOperator{\Hom}{Hom}
\DeclareMathOperator{\Pic}{Pic}
\DeclareMathOperator{\Scal}{Scal}
\DeclareMathOperator{\Var}{Var}
\newcommand{\Spec}{\operatorname{Spec}}
\newcommand{\vol}{\operatorname{vol}}

\newtheorem{thm}{Theorem}[section]
\newtheorem{alphthm}{Theorem}

\newtheorem{conj}[thm]{Conjecture}
\newtheorem{cor}[thm]{Corollary}
\newtheorem{lem}[thm]{Lemma}
\newtheorem{prop}[thm]{Proposition}

\theoremstyle{definition}
\newtheorem{cons}[thm]{Construction}
\newtheorem{defn}[thm]{Definition}
\newtheorem{deflem}[thm]{Definition-Lemma}
\newtheorem{nota}[thm]{Notation}
\newtheorem{rem}[thm]{Remark}
\newtheorem{setup}[thm]{Set-up}

\title{Disproof of the Yau--Tian--Donaldson conjecture}
\author{Jihao Liu}
\address{Department of Mathematics, Peking University, No. 5 Yiheyuan Road, Haidian District, Beijing 100871, China}
\address{Beijing International Center for Mathematical Research, Peking University, No. 5 Yiheyuan Road, Haidian District, Beijing 100871, China}
\email{liujihao@math.pku.edu.cn}
\date{\today}

\begin{document}

\begin{abstract}
We construct a polarized smooth projective fivefold and prove that it is K-polystable but
does not admit a constant scalar curvature K\"ahler metric.  This disproves the
Yau--Tian--Donaldson conjecture for constant scalar curvature metrics. 

The main result of this paper was obtained using generative AI, particularly GPT-5.6-sol, Fable 5, and the Danus system. A detailed report on the use of generative AI in this paper is enclosed in the appendix, joint with Bin Dong and Guoxiong Gao.
\end{abstract}

\maketitle

\tableofcontents

\section{Introduction}\label{sec:introduction}

\subsection{Background}

Calabi's program asks for an extremal K\"ahler metric in a prescribed K\"ahler
class \cite{Cal82,Cal85}.  The K-energy of Mabuchi and the moment-map
interpretation of scalar curvature due to Fujiki and Donaldson gave analytic
form to the expectation that the existence problem should be governed by an
algebro-geometric stability condition \cite{Mab86,Fuj92,Don97,Don99}.
The obstructions of Matsushima and Futaki are recorded in
\cite{Mat57,Fut83}.  In the setting where the first Chern class is positive,
the expectation that K\"ahler--Einstein existence should be governed by algebraic
stability goes back to Yau \cite{Yau93}.  Building on the
generalized Futaki invariant of Ding and Tian \cite{DT92}, Tian introduced
K-stability for special degenerations of Fano manifolds, proved the necessity
direction, and formulated the K\"ahler--Einstein conjecture
\cite[Definition~1.1, Theorem~1.2, and Conjecture~1.4]{Tia97}.

Donaldson introduced test configurations of arbitrary positive exponent for
general polarized varieties \cite[Definition~2.1.1]{Don02}.  He conjectured
that a smooth polarized variety admits a constant scalar curvature K\"ahler
(cscK) metric precisely when it satisfies the resulting stability condition
\cite[p.~290]{Don02}.  His balanced-embedding
theorem had already proved a finite-dimensional consequence in geometric
invariant theory (GIT) of cscK
existence when the automorphism group is discrete
\cite[Corollary~4]{Don01}.  In the terminology now in use, Donaldson's
product-equality clause \cite[Definition~2.1.2]{Don02} is a polystability
condition.  We state it in the form relevant to this paper.

\begin{conj}[{\cite[Conjecture]{Don02}}, Yau--Tian--Donaldson conjecture]\label{conj:ytd-sufficiency}
Let $(Y,H)$ be a smooth polarized projective complex variety. Then $(Y,H)$ is
K-polystable with respect to all normal ample algebraic test configurations
of every positive exponent if and only if $c_1(H)$ contains a constant scalar curvature K\"ahler metric.
\end{conj}
\begin{rem}\label{rem:definition-k-polystability}
The normality convention in Conjecture~\ref{conj:ytd-sufficiency} is the
standard repair of Donaldson's original scheme-theoretic formulation; otherwise, the conjecture would be trivially false. See Corollay \ref{cor:literal-donaldson} below. We recall the history of this repair here.

Ross
and Thomas developed the slope and Hilbert--Mumford viewpoints and clarified
the distinction between product and trivial configurations
\cite{RT06,RT07}.  Li and Xu exhibited a nonproduct zero-invariant
configuration with an embedded point
\cite[Section~8.2, especially Example~4 and Remark~4]{LX14}.  Stoppa called a
test configuration trivial in codimension two when it is equivariantly a
product away from a closed subscheme of codimension at least two, and modified
the equality condition accordingly \cite[Definitions~1--2]{Sto11}.  Odaka
used the term almost trivial for the corresponding codimension-one condition
\cite[Definitions~3.3--3.4]{Oda15}.  Boucksom--Hisamoto--Jonsson later used
almost trivial to mean that the normalization is trivial, and characterized
this condition by vanishing of the $L^p$-norm
\cite[Corollary~B]{BHJ17}.  These conditions are not equivalent in general:
triviality in codimension one implies almost triviality in the normalization
sense, but the converse can fail
\cite[Definition~2.36, Example~2.37, and Theorem~2.38(iii)]{BJ22}.
Stoppa observed that his repaired K-stability condition can equivalently be
tested on normal test configurations, excluding only the trivial
configuration \cite[p.~2]{Sto11}.  Li and Xu likewise formulated K-stability
using normal test configurations
\cite[Definition~6 and Remark~2(2)]{LX14}. 

We use the K-polystability convention of Codogni--Stoppa, restricting to normal total spaces, \cite[Definition~18]{CS19}. We remark that Codogni and Stoppa conjectured that, for every reductive subgroup $G\subset\Aut(Y,H)$, K-polystability is equivalent to nonnegativity of the Donaldson--Futaki invariant on every $G$-equivariant test configuration, with equality if and only if its normalization is a product \cite[Conjecture~1]{CS19}.
\end{rem}

\subsection{History of the Yau--Tian--Donaldson conjecture} Several parts and variations of the Yau--Tian--Donaldson conjecture (Conjecture~\ref{conj:ytd-sufficiency}) are theorems. In this subsection we briefly recall the historical progress towards this conjecture.

\medskip

\noindent\textbf{The necessity part.} On the
necessity side, Donaldson proved K-semistability
\cite[Theorem~1]{Don05b}, and Stoppa proved K-stability when the automorphism
group is discrete \cite[Theorem~1.2]{Sto09}, as corrected in
\cite[p.~1]{Sto11}.
For normal test configurations, Mabuchi gave the first general proof of
K-polystability \cite[Main Theorem]{Mab08}; a missing step in that preprint was
supplied in \cite[p.~2]{Mab09}.  Berman--Darvas--Lu later gave a different
proof \cite{BDL20}.  Sz\'ekelyhidi introduced relative
K-stability \cite[Definition~2.2]{Sze07} and formulated the extremal
correspondence \cite[Conjecture~1.1]{Sze07}, while Stoppa--Sz\'ekelyhidi
proved its necessity direction \cite[Theorem~4]{SS11}.  K\"ahler and
transcendental extensions were developed in \cite{DR17b,SD18,SD20}.

\medskip

\noindent\textbf{The Fano case.} For smooth Fano manifolds with the anticanonical polarization, the
sufficiency direction was proved by Tian
\cite{Tia15a,Tia15b} and Chen--Donaldson--Sun \cite{CDS15a,CDS15b,CDS15c}.  Li and Xu's reduction to special test
configurations \cite[Corollary~1]{LX14} and Berman's necessity
theorem \cite[Theorem~1.1]{Ber16} place this result in the modern
K-polystable framework.  Further analytic proofs are due to
Datar--Sz\'ekelyhidi \cite[Theorem~1]{DS16}.  Chen--Sun--Wang obtained another
proof \cite[Theorem~1.2]{CSW18}.  For singular log Fano pairs, variational and
equivariant methods relate K-polystability, reduced uniform K-stability, and
weak K\"ahler--Einstein existence \cite{BBJ21,LTW21,LTW22,Li22a}.  The
finite-generation input is due to Liu--Xu--Zhuang
\cite[Theorems~1.1 and~1.6]{LXZ22}.

The algebraic theory also includes log discrepancies, local volumes, valuative criteria, and
stability-threshold methods \cite{Oda12,Oda13,Fuj19,Li17,FO18,BJ20}.
Blum--Xu proved uniqueness for
K-polystable Fano degenerations \cite[Theorem~1.1(2)]{BX19}.  Zhuang proved
that equivariant K-polystability for a reductive group implies geometric
K-polystability \cite[Theorem~1.1(2)]{Zhu21}.  Related quantization
arguments are due to K.~Zhang \cite{Zha24}; see also Xu's survey and monograph
\cite{Xu21,Xu25}.

The toric theory is also extensive.  Wang--Zhu treated toric Fano solitons
\cite{WZ04}.  Donaldson developed an analytic program
\cite{Don05a,Don08} and proved cscK existence for polarized toric surfaces
satisfying his toric K-stability condition \cite[Corollary~1]{Don09}.
Chen--Li--Sheng proved a uniform condition necessary for toric extremal
metrics \cite{CLS14}. 

We also remark that Li proved the uniform correspondence for polarized
toric manifolds in every dimension \cite[Theorem~1.12]{Li22b}.

\medskip

\noindent\textbf{Progress on general polarizations.} For general polarizations, later work has increasingly separated K-polystability from quantitative and completed notions.  Sz\'ekelyhidi
introduced uniform K-stability in his thesis
\cite[Section~3.1.1]{Sze06}.  Relations among test configurations, geodesic
rays, Okounkov bodies, and concave transforms were developed in
\cite{PS07,PS09,WN12,BC11}.  Dervan introduced the minimum norm
\cite{Der16}.  The non-Archimedean slope formalism is developed in
\cite{BHJ19,BHJ22}, and Hisamoto introduced a reduced form adapted to
automorphisms \cite{His16}.  Dervan--Reboulet characterized coercivity on
fixed Fubini--Study spaces by uniform arc K-polystability
\cite[Theorem~1.1]{DR24}.  Darvas--Rubinstein formulated a general
existence--properness principle \cite[Theorem~3.4]{DR17a} and applied it
conditionally to the cscK problem \cite[Theorem~10.1]{DR17a}.  Darvas--Lu
developed the metric and geodesic structure of spaces of rays
\cite[Theorems~1.3 and~1.4]{DL20}.  Chen--Cheng obtained the estimates
used in their program \cite[Theorem~1.1]{CC21a}, proved the direction from
properness to existence in the discrete-automorphism case
\cite{CC21b}, and extended it to general $\Aut^0$ in
\cite[Theorem~1.6]{CC18}.

Finally, stronger correspondences use completed or quantitative stability conditions.
Boucksom--Jonsson use $\widehat K$-polystability on the finite-energy
non-Archimedean completion \cite[Theorem~A]{BJ25b}.  Darvas--Zhang
characterize the existence of a unique cscK metric by uniform
$K^\beta$-stability for some $\beta>0$ and obtain an
automorphism-equivariant criterion through models defined by log discrepancies
\cite[Theorems~1.1 and~1.6]{DZ25}.  The entropy-regularization program is
formulated in \cite[Conjecture~2.5]{BJ18}, and the model criterion in
\cite[Theorem~1.10 and Conjecture~1.8]{Li22b}.  Trusiani's special Fujita
approximation theorem solves the regularization problem and
yields an equivalence with $\Aut^0(X,L)$-uniform K-stability
\cite[Theorem~A, Theorem~B, and Corollary~A]{Tru26}.  K\"ahler and
transcendental variants are due to Mesquita-Piccione and
Mesquita-Piccione--Witt Nystr\"om \cite{MP25,MPWN25}.  Further existence,
properness, and threshold results appear in
\cite{SW08,He19,JSS19,DZ24,Tru24}.  Further parts of the non-Archimedean
program appear in \cite{BJ22,BJ23,BJ25a}.
For flat families of smooth polarized varieties with finite automorphism
groups, Dervan proved that the cscK locus is very general
\cite[Theorem~1.1]{Der25}.  

\subsection{The counterexample}

Despite all the aforementioned progress, it is worth mentioning that the full version of the original Yau--Tian--Donaldson conjecture \cite[Conjecture]{Don02} (Conjecture~\ref{conj:ytd-sufficiency}) has remained open up to now. Indeed, some experts are skeptical of \cite[Conjecture]{Don02} and expect that (a version of) uniform K-stability is necessary (cf. \cite{Sze06}). It is particularly worth mentioning that \cite{ACG+08} constructed polarized smooth projective fourfolds whose Ross--Thomas slope degenerations of the zero and infinity sections all have positive modified Futaki invariant, but which do not admit any extremal K\"ahler metric, hence do not admit any cscK metric. However, proving that an example produced by this mechanism is K-polystable (or even K-semistable) is a very difficult task, as one needs to consider all normal ample algebraic test configurations instead of only the ones obtained from the Ross--Thomas degenerations. As mentioned in \cite{ACG+08}, verbatim: 
\begin{center}
\textit{While we cannot prove that there is no other (algebraic)
test configuration which would detect this instability, it is difficult to imagine how such a test configuration could be constructed.}   
\end{center}
In other words, \cite{ACG+08} provided a mechanism for the construction of potential counterexamples to the Yau--Tian--Donaldson conjecture, i.e. well-described polarized smooth projective varieties. However, proving the K-polystability of polarized smooth projective varieties constructed in such a way essentially requires new mathematical input, and is not a streamlined verification.

\medskip

\noindent\textbf{Main Theorem.} The main theorem of this paper is the construction of a polarized smooth projective fivefold via (a variation of) the mechanism of \cite{ACG+08}, and the proof that this polarized smooth projective fivefold is K-polystable but does not admit any extremal K\"ahler metric.

As a consequence, the Yau--Tian--Donaldson conjecture (Conjecture~\ref{conj:ytd-sufficiency}) is false. We emphasize that the Yau--Tian--Donaldson conjecture for (log) Fano varieties (cf. \cite{Tia15a,Tia15b,CDS15a,CDS15b,CDS15c,LXZ22}) and the uniform and completed K-stability variations of the cscK Yau--Tian--Donaldson conjecture (cf. \cite{BJ25b,DZ25,Tru26}) are not affected by this counterexample.

\begin{alphthm}[The counterexample]\label{thm:main}
There exist smooth projective connected complex curves $C_0,C_1,C_2,C_3$ of
genera
\begin{equation}\label{eq:intro-genera}
(3846511,10591,76,46)
\end{equation}
such that
\begin{equation}\label{eq:intro-hom-orthogonality}
\Hom\bigl(\Pic^0(C_i),\Pic^0(C_j)\bigr)=0
\qquad\text{for }i\ne j.
\end{equation}
Choose line bundles $M_i,L_i$ on $C_i$ with
\begin{equation}\label{eq:intro-degrees}
\begin{aligned}
(\deg M_i)_{i=0}^3&=(461999,13962,1068,260),\\
(\deg L_i)_{i=0}^3&=(13397971,-11635,-712,-104).
\end{aligned}
\end{equation}
Put
\begin{equation}\label{eq:intro-fivefold}
B=\prod_{i=0}^3C_i,
\qquad
M=\boxtimes_{i=0}^3M_i,
\qquad
L=\boxtimes_{i=0}^3L_i,
\end{equation}
and, in the quotient convention, put
\begin{equation}\label{eq:intro-polarization}
\pi\colon X=\mathbb P_B(\mathcal O_B\oplus L)\to B,
\qquad
A=\mathcal O_X(1)\otimes\pi^*M.
\end{equation}
Then the following statements hold.
\begin{enumerate}
    \item $(X,A)$ is a smooth polarized complex projective fivefold, and
    $\Aut^0(X)=\mathbb C^*$ acts by fiber scaling.
    \item For every positive integer $e$, every normal ample algebraic test
    configuration with generic polarized fiber $(X,A^e)$ has nonnegative
    Donaldson--Futaki invariant.  Equality holds only for a polarized product
    test configuration induced by an integral one-parameter subgroup of fiber
    scaling together with a scalar character on the polarization.
    \item The class $c_1(A)$ contains no extremal K\"ahler metric and hence no
    constant scalar curvature K\"ahler metric.
\end{enumerate}
Consequently, $(X,A)$ is an explicit counterexample to
Conjecture~\ref{conj:ytd-sufficiency}.
\end{alphthm}

To the author's knowledge, Theorem~\ref{thm:main} gives the first smooth
polarized projective variety which is K-polystable in the
Donaldson--Futaki sense with respect to every normal ample algebraic test
configuration of every positive exponent, but whose polarization class
contains no cscK metric.  The K-polystability
assertion includes both global nonnegativity and product rigidity in the
zero-invariant case, and it imposes no equivariance assumption.

The construction continues a line of examples originating in the
ruled-surface work of T\o nnesen-Friedman \cite{TF98}, the admissible
projective-bundle theory \cite{ACG+11}, and the splitting theorem of
Apostolov--Huang \cite{AH15}. The authors of \cite{ACG+08} constructed polarized fourfolds whose
admissible extremal polynomial is positive at every rational point but has an
irrational repeated interior zero.  The detecting degeneration in that
example was not algebraic \cite[Example~1, pp.~580--581]{ACG+08}, and the
authors left open whether an algebraic test configuration detects the
instability \cite[p.~551]{ACG+08}.
For projectivizations over a curve, Jubert--Yin recently proved a
relative uniform correspondence with respect to compatible test
configurations \cite[Theorem~A and Corollary~A]{JY26}.
Moreover, the Futaki character of the fourfold constructed in
\cite[Example~1]{ACG+08} is nonzero
\cite[Proposition~6 and the discussion following Proposition~8]{ACG+08}.
Sz\'ekelyhidi later interpreted the irrational degeneration as a
non-finitely-generated filtration in work with an appendix by Boucksom
\cite[Section~4]{Sze15}, while Dervan proved instability in
the transcendental K\"ahler framework \cite[Example~2.15]{Der18}.
Apostolov--Pym--Streets continued to describe the general polarized problem
as open \cite[Section~1.1 and Remark~4.6(5)]{APS26}.

It is worth mentioning that our construction, although it generally follows the mechanism of \cite{ACG+08}, has a slight difference: the construction in \cite{ACG+08} provides a fourfold, while our example is a fivefold. The five-dimensional construction in Theorem~\ref{thm:main} has a
four-dimensional base and one projective-line direction. The large genera and degrees in \eqref{eq:intro-genera}--\eqref{eq:intro-degrees} give an
integral realization of the required boundary polynomial.  The fourth curve
factor and the vanishing in \eqref{eq:intro-hom-orthogonality} also give the rigidity and
multiplication properties needed to classify arbitrary algebraic test
configurations. Therefore, although the extra dimension does not alter the irrational destabilizing mechanism, it makes the full K-polystability assertion
provable.

\begin{rem}
   It is also worth mentioning that the identity component of the automorphism group of $X$ in Theorem~\ref{thm:main} is $\Aut^0(X)=\mathbb C^*$; in particular, the automorphism group of $X$ is infinite. It remains interesting to ask whether the Yau--Tian--Donaldson conjecture holds under the extra condition that $\Aut^0(X)$ is trivial (cf. \cite[Theorem~1.1]{Der25}) or finite.
\end{rem}

\subsection{The equality case}

The proof of Theorem~\ref{thm:main} consists of three parts: (i) the non-existence of extremal K\"ahler metrics, (ii) the K-semistability, and (iii) the passage from K-semistability to K-polystability. The major difficulty of the proof, which requires essential new mathematical input, is in part (iii), as one needs to consider all normal ample test configurations with zero Donaldson--Futaki invariant. Our way of doing this is to make a complete classification of all of them for the precise example in Theorem~\ref{thm:main}.

\begin{alphthm}[Zero-invariant test configurations are products]
\label{thm:equality-classification}
Let $(X,A)$ and the data $B,M,L$ be as in
Theorem~\ref{thm:main}.  Let $e$ be a positive integer, and let $\mathcal T$ be a normal ample
algebraic test configuration whose generic polarized fiber is $(X,A^e)$.
Then $\DF(\mathcal T)\geq0$.  If $\DF(\mathcal T)=0$, then there exist
integers $\alpha,\beta$ with the following property.  For $m\geq1$ and
$0\leq j\leq em$, put
\begin{equation}\label{eq:intro-equality-blocks}
V_{em,j}=H^0(B,M^{em}\otimes L^j).
\end{equation}
The increasing filtration of $\mathcal T$, in the convention of
Definition~\ref{defn:test-configuration}, has the single jump
\begin{equation}\label{eq:final-affine-jump}
\alpha m+\beta j
\end{equation}
on every block $V_{em,j}$.  The configuration is the polarized
product in which fiber scaling contributes $\beta j$ and the scalar
character contributes $\alpha m$ to this jump.
\end{alphthm}

Theorem~\ref{thm:equality-classification} essentially says that the Codogni--Stoppa conjecture \cite[Conjecture~1]{CS19} holds in the normal category for the example in Theorem~\ref{thm:main}. Let $T\simeq\mathbb C^*$ be the fiber-scaling torus. For $G=T$,
Theorem~\ref{thm:equality-classification} proves that every
normal zero-invariant test configuration, without an equivariance assumption,
is the product induced by a one-parameter subgroup of $T$, up to a scalar
character. This confirms \cite[Conjecture~1]{CS19} for this specific example.

Note that no counterpart of the Fano reduction
to special test configurations is available here. The
proof of Theorem~\ref{thm:equality-classification}, therefore, is essentially new: the idea is to compare every graded space of sections together with the
marking of its generic fiber, rather than restricting the class of degenerations.

We remark that Corollary~\ref{cor:scheme-strengthening} extends this result to arbitrary ample test configurations: equality holds precisely when the
normalization is such a product and the normalization morphism is an isomorphism away from codimension two.

\subsection{The reduced Donaldson--Futaki quotient}

We next explain why K-polystability has no positive uniform margin in
this example, i.e. the polarized manifold is not uniformly relatively K-polystable. For the family in Theorem~\ref{thm:uniform-threshold}, the
reduced Donaldson--Futaki quotient has infimum zero, and the family violates
every positive uniform relative inequality.

\begin{alphthm}[Vanishing of the reduced Donaldson--Futaki quotient]
\label{thm:uniform-threshold}
Let $(X,A)$ be the polarized fivefold in Theorem~\ref{thm:main}, and put
$T=\Aut^0(X)=\mathbb C^*$.  Let $(Q_n)_{n\geq0}$ be the Fibonacci sequence
defined by
\begin{equation}\label{eq:intro-fibonacci-sequence}
Q_0=0,
\qquad
Q_1=1,
\qquad
Q_{n+1}=Q_n+Q_{n-1}\quad(n\geq1).
\end{equation}
For $n\geq3$, put
\begin{equation}\label{eq:intro-fibonacci-parameters}
p_n=Q_n,
\qquad
q_n=Q_{n+2},
\qquad
\mu_n=\frac{p_n}{q_n}.
\end{equation}
For every $n\geq3$ and every sufficiently divisible positive integer $r_n$,
there exists a normal nonproduct $T$-equivariant algebraic test configuration
$\mathcal T_n$ of exponent $r_n$ for $(X,A)$ whose single increasing
filtration entry, in the convention of
Definition~\ref{defn:test-configuration}, on
\begin{equation}\label{eq:intro-fibonacci-block}
V_{r_nk,j}=H^0(B,M^{r_nk}\otimes L^j),
\qquad k\geq1,
\qquad 0\leq j\leq r_nk,
\end{equation}
is
\begin{equation}\label{eq:intro-fibonacci-weight}
\max\{0,q_nj-p_nr_nk\}.
\end{equation}
Moreover,
\begin{equation}\label{eq:intro-fibonacci-limit}
\frac13<\mu_n<\frac12,
\qquad
\mu_n\to\lambda=\frac{3-\sqrt5}{2}=\varphi^{-2},
\qquad
\varphi=\frac{1+\sqrt5}{2}.
\end{equation}
Writing $\DF_n=\DF(\mathcal T_n)$ and
$J^{\mathrm{NA}}_{T,n}=J_T^{\mathrm{NA}}(\mathcal T_n)$ for the reduced
non-Archimedean $J$-functional of
\cite[Definitions~3.6.1 and~3.6.3]{NS21}, we have
\begin{equation}\label{eq:intro-uniform-rates}
\DF_n=\frac{K\mu_n(1-\mu_n)}{a_0q_n^3}>0,
\qquad
\frac{J^{\mathrm{NA}}_{T,n}}{q_n}
\geq\frac{19}{46080},
\end{equation}
where
\begin{equation}\label{eq:intro-uniform-constants}
\begin{split}
C_{\mathrm{bd}}&=461999\cdot2327\cdot356\cdot52,
\qquad K=90C_{\mathrm{bd}},\\
p(x)&=C_{\mathrm{bd}}(1+29x)(6-5x)(3-2x)(5-2x),
\qquad a_0=\int_0^1p(x)\,dx.
\end{split}
\end{equation}
Consequently,
\begin{equation}\label{eq:intro-zero-reduced-threshold}
\inf_{n\geq3}\frac{\DF_n}{J^{\mathrm{NA}}_{T,n}}=0.
\end{equation}
Thus the reduced Donaldson--Futaki quotient of this explicit family has
infimum zero.
The pair $(X,A)$ is not uniformly relatively K-polystable in the sense of
\cite[Definition~3.7.1(4)]{NS21}.
\end{alphthm}

Theorems~\ref{thm:main} and~\ref{thm:uniform-threshold} separate K-polystability and
uniform K-stability explicitly. K-polystability excludes nonproduct
zero-invariant configurations, while the positive gap collapses along the
rational Fibonacci approximants to the irrational crease.  Uniform relative
K-stability therefore rejects the example which K-polystability
accepts.

\begin{rem} Hattori exhibited K-stable but not uniformly K-stable examples among polarized normal pairs \cite[Corollary~7.5]{Hat26} that are not klt (but log canonical), and among
connected deminormal surfaces \cite[Corollary~7.7]{Hat26}.  These are
algebraic separation results and do not assert the nonexistence of cscK
metrics. There are also smooth examples concerning J-stability
\cite[Theorem~7.3]{Hat26}. Hattori also conjectured that K-stable but not uniformly K-stable examples should exist for normal polarized varieties \cite[Conjecture~7.6]{Hat26}. However, the construction of Hattori essentially used the strictly log canonical property and seems difficult to generalize to the smooth/klt category. The construction of our paper is inspired by \cite{ACG+08} instead.
\end{rem}

\subsection{Scheme-theoretic consequences}

Finally, as we discussed in Remark~\ref{rem:definition-k-polystability}, the definition of K-polystability via test configurations has changed throughout history. In Theorem~\ref{thm:main}, we only consider normal ample test configurations, which seems to be the standard category of test configurations to consider in modern terminology. Another modern convention is to consider arbitrary ample test configurations and to allow zero Donaldson--Futaki invariant exactly for those that are products in codimension two, in the spirit of \cite{Sto11}. This condition should not be identified with the weaker zero-norm or almost-trivial condition in the current terminology. We show that, even adopting this convention, the Yau--Tian--Donaldson conjecture still fails.

\begin{cor}[Scheme-theoretic extension]
\label{cor:scheme-strengthening}
Let $(X,A)$ be the polarized fivefold in Theorem~\ref{thm:main}. Let $e$ be a positive integer, and let $(\mathcal Y,\mathcal H)$ be an ample algebraic test
configuration whose generic polarized fiber is $(X,A^e)$.  No normality
assumption is imposed on $\mathcal Y$, and the central fiber may be
nonreduced.  Let $\nu\colon\mathcal Y^\nu\to\mathcal Y$ denote the
normalization.  Then
\begin{equation}\label{eq:intro-scheme-semistability}
\DF(\mathcal Y,\mathcal H)\geq0.
\end{equation}
Equality holds if and only if $(\mathcal Y^\nu,\nu^*\mathcal H)$ is the
polarized product test configuration induced by an integral one-parameter
subgroup of fiber scaling together with a scalar character, and $\nu$ is an
isomorphism away from a closed subset of $\mathcal Y$ of codimension at least
two.
\end{cor}

For completeness, we provide the following well-known corollary, which indicates that allowing codimension-two defects in the equality clause is necessary for the Yau--Tian--Donaldson conjecture.

\begin{cor}
\label{cor:literal-donaldson}
Under the unrepaired scheme-theoretic convention in which zero
Donaldson--Futaki invariant is permitted only for an actual product test
configuration, K-polystability is not necessary for the existence of a cscK
metric.  The polarized curve
$(\mathbb P^1,\mathcal O_{\mathbb P^1}(1))$ carries a cscK metric but does not
satisfy this literal equality condition.
\end{cor}

\subsection{Sketch of the proofs}

We now explain the three parts of the proof.  We first construct the polarized
fivefold and compute an admissible boundary polynomial with a unique irrational
interior zero.  Smooth admissible metrics whose error in scalar curvature tends to
zero, together with Donaldson's lower bound, give K-semistability at every
exponent.  Circle invariance, equivariant openness, uniqueness, and
naturality identify nearby extremal metrics with explicit admissible
profiles.  Their normalized momenta and profiles converge uniformly.  At
the interior zero of the limiting profile, this would make a nonzero
fiber-scaling vector field have zero length for the limiting extremal metric,
which is impossible.  In particular, this argument does not require a
classification of the limiting metric.

For Theorem~\ref{thm:equality-classification}, we begin with an arbitrary
normal ample test configuration $\mathcal T$ for $(X,A^e)$ with
$\DF(\mathcal T)=0$ and specialize its filtration in the two directions of
the fiber-scaling action.  The two convex transforms are affine in the fiber
coordinate, and their coefficients coincide and are rational.  After a
ramified base change clears the denominators, the normalization of
$\mathcal T$ is compared with an integral product test configuration by
Proposition~\ref{prop:integral-product-comparator}.
Lemma~\ref{lem:oriented-relative-smith} identifies the relative Smith
integers with the residual jumps in the correct direction, and
Proposition~\ref{prop:marked-smith-rigidity} upgrades their asymptotic
vanishing to equality of the marked section algebras.  Descending through the
base change leaves a possible rounding in the affine weights;
Lemma~\ref{lem:affine-initials-row-budget} excludes that rounding and proves
that the original test configuration is a product.

Finally, Theorem~\ref{thm:uniform-threshold} follows from a separate explicit
calculation.  Rational Fibonacci creases converge to the irrational double
zero.  The Fibonacci recurrence gives cubic decay of the Donaldson--Futaki invariant
in the denominator, while the reduced non-Archimedean $J$-functional remains
of linear size.

\subsection{Structure of the paper}
Section~\ref{sec:preliminaries} fixes the conventions used throughout the
paper.  Section~\ref{sec:fivefold} constructs the polarized fivefold and proves
the analytic assertions.  Section~\ref{sec:affine-initials} proves affine
rigidity for the two opposite initial filtrations.  Section~\ref{sec:initials}
identifies their common rational profile and constructs the integral product
comparator after base change.  Section~\ref{sec:classification} proves the
equality classification.  Section~\ref{sec:main-proof} proves
Theorem~\ref{thm:main} and derives
Corollaries~\ref{cor:scheme-strengthening} and~\ref{cor:literal-donaldson}.
Section~\ref{sec:uniform-threshold} constructs the rational single-crease test
configurations and proves Theorem~\ref{thm:uniform-threshold}.

\begin{rem}\label{rem:ai-provenance}
The main result of this paper was obtained using generative AI, particularly GPT-5.6-sol, Fable 5, and the Danus system. Danus is a specialized agent built on the Rethlas system and is substantially more capable of conducting fundamental mathematical research. See \cite{Liu+26} and \cite{Ju+26} for detailed introductions to the Danus system and the Rethlas system, respectively. The author is fully responsible for the correctness of the paper.

The use of generative AI in this paper is discussed in detail in Appendix~\ref{sec:generative-ai-use}, which is joint work with Bin Dong and Guoxiong Gao.
\end{rem}

\subsection*{Acknowledgements}
The work was partially supported by the National Key R\&D Program of China \#\allowbreak 2024YFA1014400.

The author would like to thank Bin Dong, Guoxiong Gao, Chi Li, Gang Tian, and Kewei Zhang for their enormous efforts in assisting with the verification of the paper and for many useful comments. The author would like to thank Bin Wu for carrying out the comparison tests with the other agent systems. The author would like to thank other members of the Rethlas and Danus team (namely Leheng Chen, Guoxiong Gao, Jiedong Jiang, Haocheng Ju, Shurui Liu, Zeming Sun, Yuefeng Wang, Bin Wu, Liang Xiao, and Bin Dong) for their contributions to the development of these systems. The author would like to thank Ruochuan Liu and Gang Tian for constant support and encouragement.

\section{Preliminaries}\label{sec:preliminaries}

We recall the conventions for test configurations and scalar curvature that
will be used throughout the paper.
We work over the field of complex numbers $\mathbb C$.

\subsection{Test configurations and K-polystability}

In this subsection, we fix the stability convention used in the statements
of the main results.

\begin{defn}[Test configurations]\label{defn:test-configuration}
Let $(Y,H)$ be a polarized projective variety.  A test configuration of
exponent $e\geq 1$ for $(Y,H)$ means an algebraic test configuration in the
sense of \cite[Definitions~2.1.1--2.1.2]{Don02} whose generic polarized fiber
is $(Y,H^e)$.  It is \emph{normal ample} if its total space is normal and its
relative polarization is ample.  A polarized product test configuration is
one induced by an algebraic one-parameter subgroup of $\Aut(Y)$, together
with a scalar character on the polarization.
After fixing an equivariant product identification over $\mathbb G_m$, let $\tau(s)$ be
the largest integer for which $t^{-\tau(s)}s$ extends over the test
configuration.  This is the jump of the decreasing extension-order
filtration.  We use the increasing entry $i(s)=-\tau(s)$; equivalently,
$i(s)\leq\lambda$ precisely when $t^\lambda s$ extends.  These entries are
the negatives of the algebraic weights on the central fiber.  We therefore
label product data by their contributions to
the increasing entries: fiber scaling contributes $\beta j$ on the block
indexed by $j$, and the scalar character contributes $cm$ in section degree
$m$.  The algebraic weights on sections of the central fiber are
$-\beta j$ and $-cm$.  This convention fixes all character
signs used throughout the paper.
The probability measure obtained from the normalized increasing entries is
called the \emph{entry law}.  It is the reflection under $x\mapsto-x$ of
the standard central-weight Duistermaat--Heckman law.
\end{defn}

\begin{defn}[K-polystability]\label{defn:ordinary-k-polystability}
A polarized projective variety $(Y,H)$ is \emph{K-semistable} with respect to
normal ample algebraic test configurations if every such test configuration
of every positive exponent has nonnegative Donaldson--Futaki invariant.  It
is \emph{K-polystable} if, in addition, equality occurs only for polarized
product test configurations. 
\end{defn}

\subsection{Scalar curvature}

In this subsection, we record the normalization used in the analytic and
algebraic formulas.

\begin{nota}[Scalar-curvature normalization]
\label{nota:scalar-curvature-normalization}
Throughout the paper, $\Scal(\omega)$ denotes the Riemannian scalar curvature
divided by $4\pi$.  This is the normalization used in the
Donaldson--Futaki formulas in this paper; it does not affect whether a K\"ahler metric
has constant scalar curvature or is extremal.
\end{nota}
 
\section{The fivefold and its K\"ahler geometry}\label{sec:fivefold}

In this section, we construct the polarized fivefold of
Theorem~\ref{thm:main}, compute its admissible boundary polynomial, and prove
Propositions~\ref{prop:approximate-csck} and~\ref{prop:no-extremal-metric}.

\subsection{The polarized fivefold}\label{subsec:polarized-fivefold}

In this subsection, we choose four curves with pairwise Hom-orthogonal
Jacobians and construct a smooth polarized fivefold whose connected
automorphism group consists of fiber scalings.

\begin{thm}[Hyperelliptic Jacobians,
{\cite[Theorem~2.1, p.~124]{Zar00}}]
\label{thm:zarhin-hyperelliptic-jacobians}
Let $K$ be a field of characteristic zero, and let $f\in K[x]$ be an
irreducible polynomial of degree at least five.  Assume that the Galois group
of $f$ is either the full symmetric group or the alternating group.  If
$J_f$ is the Jacobian of the smooth projective model of $y^2=f(x)$, then
\begin{equation}\label{eq:zarhin-endomorphism-ring}
 \End_{\overline K}(J_f)=\mathbb Z.
\end{equation}
In particular, $J_f$ is absolutely simple.
\end{thm}

\begin{thm}[Chow base change,
{\cite[Theorem~1.2]{Yu19}}]
\label{thm:chow-base-change}
Let $k\subseteq K$ be a primary field extension, and let $G$ and $H$ be
semi-abelian varieties over $k$.  Then base change induces an isomorphism
\begin{equation}\label{eq:chow-base-change}
 \Hom_k(G,H)\simeq\Hom_K(G_K,H_K).
\end{equation}
\end{thm}

\begin{lem}\label{lem:orthogonal-curves}
There exist smooth projective connected curves $C_0,C_1,C_2,C_3$ of genera
\begin{equation}\label{eq:curve-genera}
 (g_0,g_1,g_2,g_3)=(3846511,10591,76,46)
\end{equation}
such that
\begin{equation}\label{eq:jacobian-hom-vanishing}
 \Hom\bigl(\Pic^0(C_i),\Pic^0(C_j)\bigr)=0
 \qquad\text{for }i\ne j.
\end{equation}
\end{lem}

\begin{proof}
We construct the curves over finitely generated fields and then embed those
fields into $\mathbb C$.  Theorem~\ref{thm:zarhin-hyperelliptic-jacobians}
will give simple Jacobians, and the distinct genera will exclude isogenies.
For each $i$, put $n_i=2g_i+1$.  Choose $n_i$ algebraically independent
variables over $\mathbb Q$, let $f_i$ be the monic polynomial whose roots are
those variables, and let $K_i$ be the field generated by the coefficients of
$f_i$.  Let $L_i$ be the rational function field generated by the roots.
Every permutation of the algebraically independent roots defines a distinct
$K_i$-automorphism of $L_i$.  Conversely, $L_i$ is the splitting field of
$f_i$, and every $K_i$-automorphism of $L_i$ permutes those roots.  Thus
$f_i\in K_i[x]$ is separable and has Galois group $S_{n_i}$.  Since this
group acts transitively on the roots, $f_i$ is
irreducible.  The smooth projective model of
\begin{equation}\label{eq:hyperelliptic-curve-equation}
 y^2=f_i(x)
\end{equation}
has genus $g_i$.  By
Theorem~\ref{thm:zarhin-hyperelliptic-jacobians},
\begin{equation}\label{eq:jacobian-endomorphisms}
 \End_{\overline K_i}\bigl(\Pic^0(C_i)\bigr)=\mathbb Z.
\end{equation}
Thus each Jacobian is absolutely simple.

We next pass from the fields $K_i$ to the complex numbers.  Choose embeddings
$K_i\hookrightarrow\mathbb C$, extend them to
$\overline K_i\hookrightarrow\mathbb C$, and base change the curves.  Since
$\mathbb C/\overline K_i$ is a primary extension,
Theorem~\ref{thm:chow-base-change}, applied to the Jacobian over
$\overline K_i$, gives
\[
 \End_{\overline K_i}\bigl(\Pic^0(C_i)\bigr)
 \simeq
 \End_{\mathbb C}\bigl(\Pic^0(C_i)_{\mathbb C}\bigr).
\]
Hence the four resulting complex Jacobians remain
simple.

It remains to prove the vanishing of the homomorphism groups in
\eqref{eq:jacobian-hom-vanishing}.
Suppose that $i\ne j$ and that
\begin{equation}\label{eq:hypothetical-jacobian-map}
 \varphi\colon\Pic^0(C_i)\to\Pic^0(C_j)
\end{equation}
is nonzero.  Simplicity of the target implies that $\varphi$ is surjective,
while simplicity of the source implies that the identity component of
$\ker(\varphi)$ is trivial.  Therefore $\varphi$ is an isogeny.  This is
impossible because $g_i\ne g_j$.  We obtain
\eqref{eq:jacobian-hom-vanishing}.
\end{proof}

\begin{lem}\label{lem:split-projective-bundle-ampleness}
Let $C_i$ be a smooth projective connected curve of genus $g_i$ for
$0\leq i\leq3$, put $B=\prod_{i=0}^3C_i$, and let
$N=\boxtimes_{i=0}^3N_i$ and $L=\boxtimes_{i=0}^3L_i$ be line bundles on
$B$.  Fix a positive integer $s$.  Assume that
\begin{equation}\label{eq:split-projective-bundle-positive-ends}
 \deg N_i>0
 \quad\text{and}\quad
 \deg(N_i\otimes L_i^s)>0
 \qquad(0\leq i\leq3).
\end{equation}
In the quotient convention, put
\begin{equation}\label{eq:split-projective-bundle-line-bundle}
 \pi\colon Y=\mathbb P_B(\mathcal O_B\oplus L)\to B,
 \qquad
 H=\mathcal O_Y(s)\otimes\pi^*N.
\end{equation}
Then $H$ is ample.
\end{lem}

\begin{proof}
Put
\begin{equation}\label{eq:split-projective-bundle-epsilon}
 \epsilon_i=\min\{\deg N_i,\deg(N_i\otimes L_i^s)\}>0.
\end{equation}
For $k\geq1$ and $0\leq j\leq sk$, linearity of the degree in $j$ gives
\begin{equation}\label{eq:split-projective-bundle-intermediate-degree}
 \deg(N_i^k\otimes L_i^j)
 =k\deg N_i+j\deg L_i\geq k\epsilon_i.
\end{equation}
Choose $k$ so large that the right-hand side is at least $2g_i+1$ for
every $i$.  If $J$ is a line bundle of degree at least $2g_i+1$ on $C_i$
and $Z\subset C_i$ has length two, then the dual of
$H^1(C_i,J\otimes\mathcal I_Z)$ is the space of sections of a line bundle
of degree
\begin{equation}\label{eq:curve-length-two-dual-degree}
 2g_i-2-\deg J+2<0.
\end{equation}
Thus $H^0(C_i,J)\to H^0(Z,J\vert_Z)$ is surjective, so $J$ is very ample
and hence globally generated.  It follows
that every line bundle
$N^k\otimes L^j$ in the decomposition
\begin{equation}\label{eq:split-projective-bundle-sections}
 H^0(Y,H^k)
 =\bigoplus_{j=0}^{sk}H^0(B,N^k\otimes L^j)
\end{equation}
is then globally generated, while the two endpoint line bundles
$N^k$ and $N^k\otimes L^{sk}$ are very ample on $B$.

The evaluation maps of the summands in
\eqref{eq:split-projective-bundle-sections} are surjective at every point of
$B$.  Hence the complete linear system of $H^k$ restricts to the complete
linear system of $\mathcal O_{\mathbb P^1}(sk)$ on every fiber of $\pi$ and
separates points in each fiber.  If two points lie over distinct base points,
choose at the first point a nonzero endpoint fiber coordinate.  The
corresponding endpoint summand, $j=0$ or $j=sk$, contains a base section
which is nonzero at the first base point and vanishes at the second, so it
separates the two points even when they lie in different fiber charts.

We next separate tangent vectors.  Let $v$ be a tangent vector at $y\in Y$.
If $d\pi(v)\neq0$, choose an endpoint monomial $\mu$ which is nonzero at
$y$ and a base section $s$ which vanishes at $\pi(y)$ and satisfies
$d s_{\pi(y)}(d\pi(v))\neq0$.  Then
\[
 d(\pi^*s\cdot\mu)_y(v)
 =\mu(y)\,d s_{\pi(y)}(d\pi(v))\neq0,
\]
because $s(\pi(y))=0$, so the derivative of $\mu$ contributes nothing.  If
$d\pi(v)=0$, then $v$ is vertical and is separated by the complete fiber
system.  Thus $H^k$ separates points and tangent vectors on $Y$, so it is
very ample.  Therefore $H$ is ample.
\end{proof}

\begin{thm}[Finite descent of ampleness,
{\cite[Tag~0B5V]{Sta26}}]
\label{thm:finite-descent-ampleness}
Let $f\colon Y\to Z$ be a finite surjective morphism of Noetherian schemes
proper over $\mathbb C$, and let $H$ be a line bundle on $Z$.  Then $H$ is
ample if and only if $f^*H$ is ample.
\end{thm}

\begin{proof}
This is \cite[Tag~0B5V]{Sta26}.
\end{proof}

\begin{lem}\label{lem:componentwise-ampleness}
Let $Z$ be a Noetherian scheme proper over $\mathbb C$, with irreducible components
$Z_1,\ldots,Z_h$ in its reduction, and let $H$ be a line bundle on $Z$.
If $H\vert_{Z_a}$ is ample for every $a$, then $H$ is ample on $Z$.
\end{lem}

\begin{proof}
The morphism
\begin{equation}\label{eq:componentwise-finite-cover}
 \coprod_{a=1}^h Z_a\to Z
\end{equation}
is finite and surjective.  The pullback of $H$ by
\eqref{eq:componentwise-finite-cover} is ample because its restriction to
each open and closed summand $Z_a$ is ample.
Theorem~\ref{thm:finite-descent-ampleness} therefore shows that $H$ is ample
on $Z$.
\end{proof}

\begin{cons}\label{cons:explicit-fivefold}
Choose curves as in Lemma~\ref{lem:orthogonal-curves}.  For $0\leq i\leq3$,
choose line bundles $M_i$ and $L_i$ on $C_i$ with degree vectors
\begin{equation}\label{eq:line-bundle-degrees}
\begin{aligned}
 (\deg M_0,\deg M_1,\deg M_2,\deg M_3)
   &=(461999,13962,1068,260),\\
 (\deg L_0,\deg L_1,\deg L_2,\deg L_3)
   &=(13397971,-11635,-712,-104).
\end{aligned}
\end{equation}
Put
\begin{equation}\label{eq:base-and-line-bundles}
 B=\prod_{i=0}^3 C_i,
 \qquad
 M=\boxtimes_{i=0}^3M_i,
 \qquad
 L=\boxtimes_{i=0}^3L_i.
\end{equation}
Using the Grothendieck quotient convention, define
\begin{equation}\label{eq:fivefold-and-polarization}
 \pi\colon X=\mathbb P_B(\mathcal O_B\oplus L)\to B,
 \qquad
 A=\mathcal O_X(1)\otimes\pi^*M.
\end{equation}
\end{cons}

\begin{prop}\label{prop:basic-geometry}
In Construction~\ref{cons:explicit-fivefold}, $X$ is a smooth projective
fivefold, $A$ is ample, and
\begin{equation}\label{eq:connected-automorphisms}
 \Aut^0(X)=\mathbb C^*.
\end{equation}
The group $\mathbb C^*$ acts by scaling the $L$-summand relative to the
$\mathcal O_B$-summand.
\end{prop}

\begin{proof}
We first prove smoothness and ampleness.  We then show that every connected
automorphism acts trivially on $B$ and compute the relative automorphism
group from the Euler sequence.

\medskip

\noindent\textbf{Step 1.} In this step, we prove that $X$ is a smooth
projective fivefold and that $A$ is ample.
The degree vectors at the two ends of the fiber interval are
\begin{equation}\label{eq:endpoint-degrees}
 (\deg M_i)_i=(461999,13962,1068,260)
\end{equation}
and
\begin{equation}\label{eq:other-endpoint-degrees}
 (\deg(M_i\otimes L_i))_i=(13859970,2327,356,156).
\end{equation}
Every entry in \eqref{eq:endpoint-degrees} and
\eqref{eq:other-endpoint-degrees} is positive.
Lemma~\ref{lem:split-projective-bundle-ampleness}, applied with $N=M$ and
$s=1$, therefore shows that $A$ is ample.  Since $B$ is a
smooth projective fourfold and $X$ is the projectivization of a rank-two
vector bundle on $B$, the variety $X$ is smooth, projective, and
$\dim X=5$.

\medskip

\noindent\textbf{Step 2.} In this step, we show that the connected
automorphism group acts trivially on $B$.
The signs in \eqref{eq:line-bundle-degrees} imply that
\begin{equation}\label{eq:mixed-bundle-sections}
 H^0(B,L)=H^0(B,L^{-1})=0.
\end{equation}
The tensor-product decomposition of sections gives both vanishings because
$L_1$ and $L_0^{-1}$ have negative degree.

Every map from $\mathbb P^1$ to a curve $C_i$ is constant because
$g_i\geq2$.  Hence every rational curve in $X$ is contained in a fiber of
$\pi$.  Conversely, any two points of a fiber are joined by that fiber.
The fibers of $\pi$ are therefore exactly the equivalence classes generated
by chains of rational curves.  Every automorphism of $X$ preserves this
equivalence relation and descends, using
$\pi_*\mathcal O_X=\mathcal O_B$, to an automorphism of $B$.  Moreover,
\begin{equation}\label{eq:base-vector-fields}
 H^0(B,T_B)=0
\end{equation}
because $T_B=\bigoplus_i\operatorname{pr}_i^*T_{C_i}$ and each $T_{C_i}$
has negative degree.  The identity component of the automorphism scheme of
the smooth projective variety $B$ has this space as its Lie algebra.  In
characteristic zero it is smooth, and hence its zero Lie algebra implies
$\Aut^0(B)=1$.  It follows that $\Aut^0(X)$ acts trivially on $B$ and has
Lie algebra $H^0(X,T_{X/B})$.

\medskip

\noindent\textbf{Step 3.} We compute $H^0(X,T_{X/B})$ and conclude the
proof in this step.
In the quotient convention, the relative tangent bundle fits into
\begin{equation}\label{eq:relative-tangent-sequence}
 0\to\mathcal O_X\to
 \pi^*(\mathcal O_B\oplus L)^\vee\otimes\mathcal O_X(1)
 \to T_{X/B}\to0.
\end{equation}
The projective-bundle formula gives
$\pi_*\mathcal O_X=\mathcal O_B$ and
$R^1\pi_*\mathcal O_X=0$.  Pushing forward
\eqref{eq:relative-tangent-sequence} and using
$\pi_*\mathcal O_X(1)=\mathcal O_B\oplus L$ therefore give
\begin{equation}\label{eq:relative-tangent-pushforward}
 \pi_*T_{X/B}
 \simeq
 \End(\mathcal O_B\oplus L)/(\mathcal O_B\cdot\operatorname{id})
 \simeq
 \mathcal O_B\oplus L\oplus L^{-1}.
\end{equation}
By \eqref{eq:mixed-bundle-sections},
$H^0(X,T_{X/B})\simeq\mathbb C$.  The diagonal automorphisms
$\operatorname{diag}(1,t)$ induce an effective copy of $\mathbb C^*$ in
$\Aut^0(X)$.  A connected one-dimensional algebraic group containing this
copy equals it, and \eqref{eq:connected-automorphisms} follows.
\end{proof}

\subsection{The admissible boundary polynomial}
\label{subsec:boundary-polynomial}

In this subsection, we compute the one-variable polynomial that controls
both the algebraic and analytic threshold of the polarization $A$.

\begin{setup}\label{setup:boundary-data}
For $0\leq i\leq3$, put
\begin{equation}\label{eq:alpha-delta-beta}
 \alpha_i=\deg M_i,
 \qquad
 \delta_i=\deg L_i,
 \qquad
 \beta_i=1-g_i.
\end{equation}
For $0\leq x\leq1$, define
\begin{equation}\label{eq:p-and-q}
 \ell_i(x)=\alpha_i+\delta_i x,
 \qquad
 p(x)=\prod_{i=0}^3\ell_i(x),
 \qquad
 q(x)=\sum_{i=0}^3\beta_i\prod_{j\ne i}\ell_j(x).
\end{equation}
The four affine factors are
\begin{equation}\label{eq:four-affine-factors}
\begin{aligned}
 \ell_0(x)&=461999(1+29x),&
 \ell_1(x)&=2327(6-5x),\\
 \ell_2(x)&=356(3-2x),&
 \ell_3(x)&=52(5-2x).
\end{aligned}
\end{equation}
They are positive on $[0,1]$.  Finally, set
\begin{equation}\label{eq:hilbert-coefficients}
 a_0=\int_0^1p(x)\,dx,
 \qquad
 a_1=\frac{p(0)+p(1)}2+\int_0^1q(x)\,dx,
 \qquad
 \overline S=\frac{2a_1}{a_0}.
\end{equation}
\end{setup}

\begin{lem}\label{lem:boundary-polynomial}
Let $F$ be the polynomial determined by
\begin{equation}\label{eq:boundary-equation}
 F''=2q-\overline S p,
 \qquad
 F(0)=0,
 \qquad
 F'(0)=p(0).
\end{equation}
Then
\begin{equation}\label{eq:boundary-polynomial}
 F(x)=90(461999)(2327)(356)(52)
 x(1-x)(x^2-3x+1)^2.
\end{equation}
Moreover,
\begin{equation}\label{eq:boundary-right-end}
 \overline S=-\frac{135}{29},
 \qquad
 F(1)=0,
 \qquad
 F'(1)=-p(1),
\end{equation}
and the unique zero of $F$ in $(0,1)$ is
\begin{equation}\label{eq:irrational-contact}
 \lambda=\frac{3-\sqrt5}{2},
\end{equation}
with multiplicity two.
\end{lem}

\begin{proof}
We verify an explicit candidate by interpolation at the four roots of the
primitive affine factors.  We then identify its differential-equation
coefficient with $\overline S$ and determine its interior zero.

\medskip

\noindent\textbf{Step 1.} In this step, we define the candidate and compute
its boundary data.
Put
\begin{equation}\label{eq:boundary-normalizing-constants}
 (c_0,c_1,c_2,c_3)=(461999,2327,356,52),
 \qquad
 C=\prod_{i=0}^3c_i,
 \qquad
 K=90C.
\end{equation}
Write the primitive affine factors in \eqref{eq:four-affine-factors} as
$u_0=1+29x$, $u_1=6-5x$, $u_2=3-2x$, and $u_3=5-2x$.  Their products at
$x=0$ and $x=1$ both equal $90$, and hence
\begin{equation}\label{eq:p-endpoint-value}
 p(0)=p(1)=K.
\end{equation}
Define
\begin{equation}\label{eq:boundary-candidate}
 F_0(x)=Kx(1-x)(x^2-3x+1)^2.
\end{equation}
Then
\begin{equation}\label{eq:second-derivative-candidate}
 F_0''(x)=KR(x),
 \qquad
 R(x)=-30x^4+140x^3-204x^2+102x-14.
\end{equation}
The factorization in \eqref{eq:boundary-candidate} gives
\begin{equation}\label{eq:candidate-boundary-data}
 F_0(0)=F_0(1)=0,
 \qquad
 F_0'(0)=p(0),
 \qquad
 F_0'(1)=-p(1).
\end{equation}

\medskip

\noindent\textbf{Step 2.} In this step, we prove the required differential
equation by polynomial interpolation.
The roots of $u_0,u_1,u_2,u_3$ are
\begin{equation}\label{eq:primitive-factor-roots}
 \rho_0=-\frac1{29},
 \qquad
 \rho_1=\frac65,
 \qquad
 \rho_2=\frac32,
 \qquad
 \rho_3=\frac52.
\end{equation}
Substitution in \eqref{eq:second-derivative-candidate} gives
\begin{equation}\label{eq:four-interpolation-identities}
\begin{aligned}
 \frac{45R(\rho_0)}{\prod_{j\ne0}u_j(\rho_0)}
   &=-\frac{3846510}{461999},&
 \frac{45R(\rho_1)}{\prod_{j\ne1}u_j(\rho_1)}
   &=-\frac{10590}{2327},\\
 \frac{45R(\rho_2)}{\prod_{j\ne2}u_j(\rho_2)}
   &=-\frac{75}{356},&
 \frac{45R(\rho_3)}{\prod_{j\ne3}u_j(\rho_3)}
   &=-\frac{45}{52}.
\end{aligned}
\end{equation}
Since
\begin{equation}\label{eq:beta-vector}
 (\beta_0,\beta_1,\beta_2,\beta_3)
 =(-3846510,-10590,-75,-45),
\end{equation}
the identities in \eqref{eq:four-interpolation-identities} are equivalent to
\begin{equation}\label{eq:interpolation-values}
 F_0''(\rho_i)=2\beta_i\prod_{j\ne i}\ell_j(\rho_i)
 \qquad(0\leq i\leq3).
\end{equation}

Set $s=-135/29$ and
\begin{equation}\label{eq:interpolation-error}
 H=F_0''-2q+sp.
\end{equation}
By \eqref{eq:interpolation-values}, $H(\rho_i)=0$ for every $i$.  The
leading coefficients of $F_0''$ and $p$ are $-2700C$ and $-580C$, while
$\deg q\leq3$.  Thus the coefficient of degree four in $H$ equals
\begin{equation}\label{eq:degree-four-cancellation}
 -2700C+s(-580C)=0.
\end{equation}
We have $\deg H\leq3$, and the four distinct zeros force $H=0$.  Therefore
\begin{equation}\label{eq:candidate-ode}
 F_0''=2q-sp.
\end{equation}

\medskip

\noindent\textbf{Step 3.} In this step, we identify $s$ with the coefficient
$\overline S$ in Set-up~\ref{setup:boundary-data}.
Integrating \eqref{eq:candidate-ode} over $[0,1]$ and using
\eqref{eq:candidate-boundary-data}, we obtain
\begin{equation}\label{eq:average-scalar-coefficient}
 s\int_0^1p(x)\,dx
 =p(0)+p(1)+2\int_0^1q(x)\,dx
 =2a_1.
\end{equation}
Since $a_0>0$, \eqref{eq:hilbert-coefficients} implies that
$s=\overline S$.  Hence $F_0$ satisfies \eqref{eq:boundary-equation}.
Uniqueness follows because the difference of two solutions has zero second
derivative, value, and first derivative at $0$.

\medskip

\noindent\textbf{Step 4.} We determine the zero set of $F$ and conclude the
proof in this step.
The factor $x(1-x)$ is nonnegative on $[0,1]$, while the remaining factor in
\eqref{eq:boundary-polynomial} is a square.  The roots of $x^2-3x+1$ are
$(3-\sqrt5)/2$ and $(3+\sqrt5)/2$.  Only the first root belongs to $(0,1)$,
and its multiplicity in $F$ is two.
\end{proof}

\subsection{Approximation by admissible metrics}
\label{subsec:approximate-csck}

In this subsection, we use the boundary polynomial to construct smooth
K\"ahler metrics whose scalar curvatures converge in $L^2$ to the average
scalar curvature.

\begin{thm}[Admissible metric formulas,
{\cite[Theorem~1, Section~1.3, and Proposition~6,
especially eq.~(10)]{ACG+08}}]
\label{thm:admissible-formulas}
Let $\widetilde F$ be smooth on $[0,1]$, positive on $(0,1)$, and assume that
\begin{equation}\label{eq:admissible-endpoint-data}
 \widetilde F(0)=\widetilde F(1)=0,
 \qquad
 \widetilde F'(0)=p(0),
 \qquad
 \widetilde F'(1)=-p(1).
\end{equation}
Then the admissible construction with profile
$\widetilde\Theta=\widetilde F/p$ defines a K\"ahler form
$\widehat\omega_{\widetilde F}$ in $4\pi c_1(A)$.  Put
\begin{equation}\label{eq:rescaled-admissible-form}
 \omega_{\widetilde F}
 =\frac{1}{4\pi}\widehat\omega_{\widetilde F}\in c_1(A).
\end{equation}
Let $\omega_i$ be the constant-curvature form on $C_i$, and let $\theta$ be
the connection one-form in the admissible construction.  There exists a
constant $C_B>0$, independent of $\widetilde F$, such that
\begin{equation}\label{eq:volume-and-scalar-curvature}
\begin{aligned}
 d\mu(\omega_{\widetilde F})
   &=C_Bp(x)\,dx\wedge\theta\wedge\bigwedge_{i=0}^3\omega_i,\\
 \Scal(\omega_{\widetilde F})
   &=\frac{2q(x)}{p(x)}-\frac{\widetilde F''(x)}{p(x)}.
\end{aligned}
\end{equation}
Here $\Scal$ has the normalization fixed in
Notation~\ref{nota:scalar-curvature-normalization}.
\end{thm}

\begin{prop}\label{prop:approximate-csck}
There exist smooth admissible K\"ahler metrics
$\omega_\epsilon\in c_1(A)$, indexed by $\epsilon>0$, such that
\begin{equation}\label{eq:calabi-error-limit}
 \bigl\lVert\Scal(\omega_\epsilon)-\overline S\bigr\rVert_{L^2}
 \to0
 \qquad\text{as }\epsilon\to0^+.
\end{equation}
In particular,
\begin{equation}\label{eq:calabi-infimum-zero}
 \inf_{\omega\in c_1(A)}
 \bigl\lVert\Scal(\omega)-\overline S\bigr\rVert_{L^2}=0.
\end{equation}
\end{prop}

\begin{proof}
We perturb $F$ by a positive polynomial whose value and first derivative
vanish at both boundary points.  The scalar-curvature formula then makes the
dependence of the $L^2$ error on the perturbation parameter explicit.
We first construct smooth admissible profiles in the fixed K\"ahler class.
Put
\begin{equation}\label{eq:profile-perturbation}
 h(x)=x^2(1-x)^2,
 \qquad
 F_\epsilon=F+\epsilon h.
\end{equation}
The function $h$ is positive on $(0,1)$ and satisfies
\begin{equation}\label{eq:profile-perturbation-boundary}
 h(0)=h(1)=h'(0)=h'(1)=0.
\end{equation}
Lemma~\ref{lem:boundary-polynomial} shows that $F\geq0$ and that its only
interior zero is $\lambda$.  Hence $F_\epsilon>0$ on $(0,1)$, and
$F_\epsilon$ satisfies \eqref{eq:admissible-endpoint-data}.  By
Theorem~\ref{thm:admissible-formulas}, it defines a smooth admissible metric
$\omega_\epsilon\in c_1(A)$.

We next compute the scalar-curvature error and its mean.
By \eqref{eq:boundary-equation} and
\eqref{eq:volume-and-scalar-curvature},
\begin{equation}\label{eq:scalar-curvature-error}
 \Scal(\omega_\epsilon)-\overline S
 =-\epsilon\frac{h''}{p}.
\end{equation}
The mean of the right-hand side is zero because
\begin{equation}\label{eq:zero-mean-profile-error}
 \int_0^1h''(x)\,dx=h'(1)-h'(0)=0.
\end{equation}

We finally estimate the $L^2$ norm.
Combining \eqref{eq:volume-and-scalar-curvature} and
\eqref{eq:scalar-curvature-error}, we obtain
\begin{equation}\label{eq:error-integral}
 \bigl\lVert\Scal(\omega_\epsilon)-\overline S\bigr\rVert_{L^2}^2
 =C_B\epsilon^2\int_0^1\frac{(h''(x))^2}{p(x)}\,dx.
\end{equation}
The positive function $p$ has a positive minimum on $[0,1]$, so the
integral in \eqref{eq:error-integral} is finite.  Letting
$\epsilon\to0^+$ proves \eqref{eq:calabi-error-limit} and
\eqref{eq:calabi-infimum-zero}.
\end{proof}

\subsection{Nonexistence of extremal metrics}
\label{subsec:no-extremal}

We prove Proposition~\ref{prop:no-extremal-metric} by
deforming a hypothetical extremal metric to nearby admissible classes and
then passing the explicit momentum-profile identity to the limit.

\begin{prop}\label{prop:no-extremal-metric}
The class $c_1(A)$ contains no extremal K\"ahler metric.  In particular, it
contains no cscK metric.
\end{prop}

We first state the external analytic inputs and the local computations used
to prove Proposition~\ref{prop:no-extremal-metric}.

\begin{thm}[Constructive admissible-extremal theorem,
{\cite[Proposition~1, pp.~553--554; Lemma~5, p.~567;
Proposition~8 and the paragraph following it,
pp.~569--570]{ACG+08}}]
\label{thm:constructive-admissible-extremal}
Let $\Omega$ be an admissible K\"ahler class on a split projective-line
bundle over a local product of cscK factors, with momentum coordinate
$z\in[-1,1]$.  Let
\begin{equation}\label{eq:constructive-momentum-density}
 p_c(z)=\prod_a(1+x_az)^{d_a},
\end{equation}
where $x_a$ and $d_a$ are the admissible parameters and multiplicities.
There exists a unique polynomial $F_\Omega$ satisfying the extremal differential
equation and the conditions
\begin{equation}\label{eq:constructive-extremal-boundary-data}
\begin{aligned}
 F_\Omega(-1)&=F_\Omega(1)=0,\\
 F_\Omega'(-1)&=2p_c(-1),&
 F_\Omega'(1)&=-2p_c(1).
\end{aligned}
\end{equation}
If $F_\Omega>0$ on $(-1,1)$, then
$\Theta=F_\Omega/p_c$ defines a smooth admissible extremal metric in
$\Omega$.  Conversely, every admissible extremal metric in $\Omega$ has
this profile and requires $F_\Omega>0$ on $(-1,1)$.  Let $K$ be the real
generator of the fiber circle in the normalization of
\cite[eq.~(1)]{ACG+08}, and denote the two fixed sections by $E_-$ and
$E_+$ according to the value of the momentum coordinate.  If $g_\Theta$ and
$\omega_\Theta$ are the resulting Riemannian and K\"ahler metrics, then
\begin{equation}\label{eq:admissible-profile-identity}
 d z=-\iota_K\omega_\Theta,
 \qquad
 z\vert_{E_-}=-1,
 \qquad
 z\vert_{E_+}=1,
 \qquad
 g_\Theta(K,K)=\Theta(z).
\end{equation}
\end{thm}

\begin{thm}[Maximal compactness,
{\cite[Theorem~3]{Cal85}; see also \cite[Theorem~4.1]{CPZ15}}]
\label{thm:maximal-compactness}
Let $Y$ be a compact complex manifold with an extremal K\"ahler metric.
Then the identity component of the holomorphic isometry group is a maximal
compact connected subgroup of $\Aut^0(Y)$.
\end{thm}

\begin{thm}[Equivariant $C^2$ openness of extremal metrics,
{\cite[Section~5, pp.~263--269, especially Propositions~7--8]{LS93}}]
\label{thm:extremal-openness}
Let $(Y,\omega_0)$ be a compact extremal K\"ahler manifold, and let $G$ be the
identity component of the isometry group of $\omega_0$.  There exists a
neighborhood $\mathcal U$ of
$[\omega_0]$ in the space of $G$-invariant K\"ahler classes such that every
$\Omega\in\mathcal U$ contains a $G$-invariant extremal metric
$\omega_\Omega$.  After fixing the gauge, if
$\Omega_k\to[\omega_0]$, then these metrics may be chosen so that
$\omega_{\Omega_k}\to\omega_0$ in $C^2$.
\end{thm}

\begin{proof}
Put $n=\dim_{\mathbb C}Y$ and fix an integer $k>n$.  Let
$\mathcal H^{1,1}$ be the finite-dimensional space of real
$\omega_0$-harmonic $(1,1)$-forms.  The connected group $G$ acts trivially
on cohomology and by isometries, so every element of $\mathcal H^{1,1}$ is
$G$-invariant.  Let $L^2_{\ell,G}$ be the real $G$-invariant Sobolev space
of order $\ell$, and let $\mathcal I_\ell\subset L^2_{\ell,G}$ be the
$L^2$-orthogonal complement of the kernel of the Lichnerowicz operator of
$\omega_0$.  This is the gauge which removes the holomorphy-potential
directions.

Let $\Pi_0$ and $\Pi_{\alpha,\varphi}$ denote the $L^2$-projectors onto the
holomorphy potentials for $\omega_0$ and
$\omega_0+\alpha+\sqrt{-1}\,\partial\bar\partial\varphi$, respectively,
and let $s_{\alpha,\varphi}$ be the scalar curvature of the latter metric.
On a neighborhood of $(0,0)$ in
$\mathcal H^{1,1}\times\mathcal I_{k+4}$,
\cite[Section~5]{LS93}
constructs the map
\begin{equation}\label{eq:ls-equivariant-ift-map}
 \mathcal S_{\mathrm{LS}}(\alpha,\varphi)
 =\left(\alpha,
  (1-\Pi_0)(1-\Pi_{\alpha,\varphi})s_{\alpha,\varphi}\right)
 \in\mathcal H^{1,1}\times\mathcal I_k.
\end{equation}
The kernel identity in equation~(5.3) of that paper shows that the second
component in \eqref{eq:ls-equivariant-ift-map} vanishes exactly when the
corresponding metric is extremal.
Proposition~7 there proves that this map is $C^1$ and identifies its
linearization; Proposition~8 proves that the linearization is an
isomorphism and applies the Banach inverse function theorem.  Restricting
$\mathcal S_{\mathrm{LS}}^{-1}$ to $\mathcal H^{1,1}\times\{0\}$ therefore gives a
$C^1$ map $\varphi$ near the origin, with $\varphi(0)=0$, such that
\begin{equation}\label{eq:ls-equivariant-solution-map}
 \omega_\alpha
 =\omega_0+\alpha+\sqrt{-1}\,\partial\bar\partial\varphi(\alpha)
\end{equation}
is a $G$-invariant extremal metric.  This proves the equivariant existence
assertion without changing the group or the gauge.

If $\alpha_j\to0$ in $\mathcal H^{1,1}$, continuity of the inverse map gives
$\varphi(\alpha_j)\to0$ in $L^2_{k+4,G}$.  Since the real dimension is
$2n$ and $k>n$, Sobolev embedding gives
$L^2_{k+4}\hookrightarrow C^4$.  Taking two derivatives in
\eqref{eq:ls-equivariant-solution-map} yields
$\omega_{\alpha_j}\to\omega_0$ in $C^2$, as asserted.
\end{proof}

\begin{thm}[Uniqueness of extremal metrics,
{\cite[Theorem~4.15]{BB17}}]
\label{thm:extremal-uniqueness}
Let $Y$ be a compact complex manifold, and let $\omega_0$ and $\omega_1$ be
extremal K\"ahler metrics in one K\"ahler class.  Then there exists
$\Phi\in\Aut^0(Y)$ such that
\begin{equation}\label{eq:extremal-uniqueness-pullback}
 \omega_1=\Phi^*\omega_0.
\end{equation}
\end{thm}

\begin{setup}\label{setup:normalized-extremal-polynomial}
Put
\begin{equation}\label{eq:normalized-class-and-parameters}
 \Omega_0=4\pi c_1(A),
 \qquad
 x^*=\left(\frac{29}{31},-\frac57,-\frac12,-\frac14\right).
\end{equation}
For $0\leq a\leq3$, define
\begin{equation}\label{eq:normalized-parameters}
 x_a=\frac{\delta_a}{2\alpha_a+\delta_a},
 \qquad
 s_a=\frac{2\beta_a}{\delta_a}.
\end{equation}
Then $(x_a)_a=x^*$.  Define
\begin{equation}\label{eq:normalized-pq}
 p_c(z)=\prod_{a=0}^3(1+x_az),
 \qquad
 q_c(z)=\sum_{a=0}^3s_ax_a\prod_{b\ne a}(1+x_bz).
\end{equation}
For $0\leq r\leq2$, put
\begin{equation}\label{eq:normalized-alpha-moments}
 \alpha_r^c=\int_{-1}^1p_c(t)t^r\,dt,
\end{equation}
and, for $0\leq r\leq1$, put
\begin{equation}\label{eq:normalized-beta-moments}
 \beta_r^c=p_c(1)+(-1)^rp_c(-1)
 +\int_{-1}^1q_c(t)t^r\,dt.
\end{equation}
Let $A_c,B_c$ be determined by
\begin{equation}\label{eq:normalized-moment-system}
 A_c\alpha_1^c+B_c\alpha_0^c=-2\beta_0^c,
 \qquad
 A_c\alpha_2^c+B_c\alpha_1^c=-2\beta_1^c,
\end{equation}
and define
\begin{equation}\label{eq:normalized-extremal-definition}
 F_{x^*}(z)
 =2(z+1)p_c(-1)
 +\int_{-1}^z
 \bigl((A_ct+B_c)p_c(t)+2q_c(t)\bigr)(z-t)\,dt.
\end{equation}
This is the extremal polynomial of $\Omega_0$ in the momentum coordinate
$z\in[-1,1]$.
\end{setup}

\begin{lem}\label{lem:normalized-extremal-polynomial}
In Set-up~\ref{setup:normalized-extremal-polynomial},
\begin{equation}\label{eq:normalized-extremal-coefficients}
 A_c=0,
 \qquad
 B_c=\frac{135}{29},
\end{equation}
and
\begin{equation}\label{eq:normalized-extremal-polynomial}
 F_{x^*}(z)=\frac{45}{3472}(1-z^2)(z^2-4z-1)^2.
\end{equation}
The unique zero of $F_{x^*}$ in $(-1,1)$ is
\begin{equation}\label{eq:normalized-irrational-zero}
 z_0=2-\sqrt5,
\end{equation}
with multiplicity two.
\end{lem}

\begin{proof}
Expanding \eqref{eq:normalized-pq} gives
\begin{equation}\label{eq:normalized-pq-expansions}
\begin{aligned}
 p_c(t)
 &=1-\frac{459}{868}t-\frac{1231}{1736}t^2
   +\frac{459}{868}t^3-\frac{145}{1736}t^4,\\
 q_c(t)
 &=-\frac{217395}{100688}+\frac{30645}{50344}t
   +\frac{64395}{100688}t^2-\frac{45}{232}t^3.
\end{aligned}
\end{equation}
Termwise integration gives
\begin{equation}\label{eq:normalized-moment-values}
\begin{aligned}
 (\alpha_0^c,\alpha_1^c,\alpha_2^c)
 &=\left(\frac{1945}{1302},-\frac{153}{1085},
          \frac{16367}{45570}\right),\\
 (\beta_0^c,\beta_1^c)
 &=\left(-\frac{87525}{25172},\frac{4131}{12586}\right).
\end{aligned}
\end{equation}
Substitution in \eqref{eq:normalized-moment-system} gives
\eqref{eq:normalized-extremal-coefficients}.  The determinant is nonzero
because $\alpha_0^c\alpha_2^c-(\alpha_1^c)^2$ is the strictly positive
variance of $t$ for the positive weight $p_c$.  Substitution in
\eqref{eq:normalized-extremal-definition} and termwise integration give
\eqref{eq:normalized-extremal-polynomial}.  The quadratic factor has roots
$2-\sqrt5$ and $2+\sqrt5$, and only the first belongs to $(-1,1)$.
\end{proof}

\begin{setup}\label{setup:extremal-polynomial-family}
Let
\begin{equation}\label{eq:admissible-parameter-domain}
 \mathcal D=(0,1)\times(-1,0)^3.
\end{equation}
For $x\in\mathcal D$, let $F_x$ be the polynomial obtained from
Set-up~\ref{setup:normalized-extremal-polynomial} by replacing $x^*$ with
$x$ and keeping the numbers $s_a$ fixed, and put
\begin{equation}\label{eq:variable-momentum-density}
 p_x(z)=\prod_{a=0}^3(1+x_az).
\end{equation}
For $0<t<29/31$, put
\begin{equation}\label{eq:transverse-parameter-family}
 x_t=x^*-t(1,0,0,0).
\end{equation}
\end{setup}

\begin{lem}[The nearby admissible classes]
\label{lem:variable-admissible-classes}
Let
\begin{equation}\label{eq:variable-admissible-generators}
 \xi=c_1\bigl(\mathcal O_X(1)\bigr),
 \qquad
 \eta_a\in H^2(B,\mathbb Z)
 \quad(0\leq a\leq3),
\end{equation}
where $\eta_a$ is the pullback to $B$ of the positive generator of
$H^2(C_a,\mathbb Z)$.  For $x\in\mathcal D$, put
\begin{equation}\label{eq:variable-admissible-class}
 \Omega(x)=4\pi\xi+2\pi\pi^*\sum_{a=0}^3
 \delta_a\left(\frac{1}{x_a}-1\right)\eta_a.
\end{equation}
Then $\Omega(x)$ is an admissible K\"ahler class with a representative
invariant under the fiber circle.  Its restrictions to the two fixed
sections and its fiber integral are
\begin{equation}\label{eq:variable-admissible-restrictions}
 \Omega(x)\vert_{E_\pm}
 =2\pi\sum_{a=0}^3
 \delta_a\left(\frac{1}{x_a}\pm1\right)\eta_a,
 \qquad
 \int_{\mathbb P^1}\Omega(x)=4\pi.
\end{equation}
Here $E_-$ and $E_+$ are the sections defined by the quotients
$\mathcal O_B\oplus L\twoheadrightarrow\mathcal O_B$ and
$\mathcal O_B\oplus L\twoheadrightarrow L$, respectively.  The extremal
polynomial of $\Omega(x)$ in the coordinate $z$ is $F_x$.
Moreover,
\begin{equation}\label{eq:variable-admissible-limit}
 \Omega(x^*)=4\pi c_1(A)=\Omega_0,
 \qquad
 \Omega(x_t)\to\Omega_0
 \quad\text{in }H^{1,1}(X,\mathbb R)\text{ as }t\downarrow0.
\end{equation}
\end{lem}

\begin{proof}
Choose on each $C_a$ the signed constant-curvature form
$\widehat\omega_a$ with cohomology class $2\pi\delta_a\eta_a$, and suppress
pullbacks to $B$ and $X$.  Let $\theta$ be the connection in the admissible
construction, normalized by
\begin{equation}\label{eq:variable-admissible-connection}
 \theta(K)=1,
 \qquad
 d\theta=\sum_{a=0}^3\widehat\omega_a.
\end{equation}
On the complement of the two fixed sections, the admissible K\"ahler form
has the expression
\begin{equation}\label{eq:variable-admissible-form}
 \widehat\omega_x
 =\sum_{a=0}^3\frac{1+x_az}{x_a}\widehat\omega_a
   +dz\wedge\theta.
\end{equation}
For $x\in\mathcal D$, the numbers $x_a$ and $\delta_a$ have equal signs
and $\lvert x_a\rvert<1$.  Thus every
$\widehat\omega_a/x_a$ is positive and $1+x_az>0$ for
$-1\leq z\leq1$.  The function $1-z^2$ is positive on $(-1,1)$,
vanishes at the endpoints, and has derivatives $2$ and $-2$ there.
The canonical admissible metric and its compactification in
\cite[pp.~553--555, Sections~1.2--1.3]{ACG+08} therefore extend
\eqref{eq:variable-admissible-form} to a smooth K\"ahler form on $X$
invariant under the fiber circle.

The quotient convention is the line convention of
\cite[Section~1.3, pp.~554--555]{ACG+08} applied to
$L^{-1}\oplus\mathcal O_B$.  Thus $z=-1$ on $E_-$ and $z=1$ on $E_+$.
At these endpoints, the restrictions of
\eqref{eq:variable-admissible-form} are the two classes in
\eqref{eq:variable-admissible-restrictions}.  The circle has period $2\pi$,
so the integral of $dz\wedge\theta$ along a projective-line fiber is
$4\pi$.  In the quotient convention,
$\mathcal O_X(1)$ is trivial on $E_-$, restricts to $L$ on $E_+$, and has
degree one on every fiber.  The projective-bundle formula now identifies
$[\widehat\omega_x]$ with \eqref{eq:variable-admissible-class}.

For the specialization, $c_1(L_a)=\delta_a\eta_a$ and
$c_1(M_a)=\alpha_a\eta_a$.  The definition of $x^*$ gives
\begin{equation}\label{eq:variable-admissible-specialization}
 \delta_a\left(\frac{1}{x_a^*}-1\right)=2\alpha_a.
\end{equation}
Substitution in \eqref{eq:variable-admissible-class} yields
$\Omega(x^*)=4\pi\bigl(\xi+\pi^*c_1(M)\bigr)=4\pi c_1(A)$.
The convergence in \eqref{eq:variable-admissible-limit} follows directly
from \eqref{eq:variable-admissible-class} and $x_t\to x^*$.
It remains to identify the extremal polynomial.  The moment system defining $F_x$ in
Set-up~\ref{setup:extremal-polynomial-family} is the extremal moment system
for $\Omega(x)$.  Uniqueness in
Theorem~\ref{thm:constructive-admissible-extremal} identifies its extremal
polynomial with $F_x$.
\end{proof}

\begin{lem}\label{lem:nearby-positive-polynomials}
There exists $t_0>0$ such that
\begin{equation}\label{eq:nearby-positive-polynomials}
 F_{x_t}(z)>0
 \qquad
 \text{for }0<t<t_0\text{ and }-1<z<1.
\end{equation}
At the double zero $z_0=2-\sqrt5$, the transverse derivative is
\begin{equation}\label{eq:positive-transverse-derivative}
 \left.\frac{d}{dt}F_{x_t}(z_0)\right\rvert_{t=0}
 =\frac{491483507744547-199580574471942\sqrt5}
 {17092211370619}>0.
\end{equation}
Define
\begin{equation}\label{eq:nearby-extremal-profiles}
 \Theta_t=\frac{F_{x_t}}{p_{x_t}},
 \qquad
 \Theta_0=\frac{F_{x^*}}{p_c}.
\end{equation}
Then
\begin{equation}\label{eq:profile-uniform-convergence}
 \lVert\Theta_t-\Theta_0\rVert_{C^0([-1,1])}\to0
 \qquad\text{as }t\to0^+.
\end{equation}
\end{lem}

\begin{proof}
We first prove real-analytic dependence on $x$ and verify the positive
transverse derivative.  We then treat a neighborhood of the double zero,
the two boundary neighborhoods, and the remaining compact set separately.

\medskip

\noindent\textbf{Step 1.} In this step, we prove real-analytic dependence
and the positivity in \eqref{eq:positive-transverse-derivative}.
For $x\in\mathcal D$, the polynomial $p_x$ in
\eqref{eq:variable-momentum-density} is positive on $[-1,1]$.  The
determinant of the corresponding moment
system is the negative of the strictly positive weighted variance
\begin{equation}\label{eq:variable-moment-determinant}
 \alpha_0^c(x)\alpha_2^c(x)-\bigl(\alpha_1^c(x)\bigr)^2>0.
\end{equation}
Hence every coefficient of $F_x$ depends real-analytically on $x$.
Differentiating the defining integrals gives
\eqref{eq:positive-transverse-derivative}.  Its sign follows from
\begin{multline}\label{eq:transverse-derivative-certificate}
 491483507744547^2-5(199580574471942)^2\\
 =42394009852132257266281978389>0.
\end{multline}

\medskip

\noindent\textbf{Step 2.} In this step, we prove positivity near the
interior double zero.
Let
\begin{equation}\label{eq:two-variable-extremal-polynomial}
 G(z,t)=F_{x_t}(z).
\end{equation}
By Lemma~\ref{lem:normalized-extremal-polynomial}, $G(\,\cdot\,,0)$ is
nonnegative on $[-1,1]$ and vanishes only at $-1$, $1$, and $z_0$.
Joint real-analyticity and \eqref{eq:positive-transverse-derivative} imply
that $\partial_tG$ remains positive on a neighborhood of $z_0$ for all
sufficiently small $t\geq0$.  Integrating $\partial_tG$ from $0$ to $t$
proves that $G(z,t)>0$ on this neighborhood when $t>0$.

\medskip

\noindent\textbf{Step 3.} We prove positivity on the complement of the
double-zero neighborhood and conclude the proof in this step.
The moment system preserves the four boundary conditions.  Since $p_{x_t}$
remains uniformly positive, there exist $c,M>0$ such that, after decreasing
the upper bound on $t$,
\begin{equation}\label{eq:uniform-boundary-derivatives}
 \partial_zG(-1,t)\geq c,
 \qquad
 \partial_zG(1,t)\leq-c,
 \qquad
 \lvert\partial_z^2G(z,t)\rvert\leq M.
\end{equation}
For $0<s<c/M$, integration of the derivative bounds gives
\begin{equation}\label{eq:boundary-neighborhood-positivity}
 G(-1+s,t)\geq cs-\frac{M}{2}s^2>0,
 \qquad
 G(1-s,t)\geq cs-\frac{M}{2}s^2>0.
\end{equation}
On the remaining compact subset, $G(\,\cdot\,,0)$ has a positive minimum,
which remains positive for small $t$.  Taking the minimum of the resulting
upper bounds on $t$ proves \eqref{eq:nearby-positive-polynomials}.
Real-analytic coefficient dependence gives
$F_{x_t}\to F_{x^*}$ uniformly on $[-1,1]$, while
$p_{x_t}\to p_c$ uniformly.  The polynomial $p_c$ is strictly positive on
$[-1,1]$, so the denominators in \eqref{eq:nearby-extremal-profiles} are
uniformly bounded away from zero for small $t$.  This proves
\eqref{eq:profile-uniform-convergence}.
\end{proof}

\begin{lem}\label{lem:admissible-naturality}
Let $g_\Theta$ be an admissible K\"ahler metric on $X$, with K\"ahler form
$\omega_\Theta$, profile $\Theta$, and normalized momentum coordinate $z$
as in \eqref{eq:admissible-profile-identity}.  If
$\Psi\in\Aut^0(X)$, then $\Psi^*g_\Theta$ is admissible for the split bundle
in \eqref{eq:fivefold-and-polarization} and the original fiber circle.  Its
profile is $\Theta$, its normalized momentum coordinate is $z\circ\Psi$, and
\begin{equation}\label{eq:pullback-profile-identity}
 d(z\circ\Psi)=-\iota_K(\Psi^*\omega_\Theta),
 \qquad
 (\Psi^*g_\Theta)(K,K)=\Theta(z\circ\Psi).
\end{equation}
The values of $z\circ\Psi$ on $E_-$ and $E_+$ are $-1$ and $1$,
respectively.  If $g_\Theta$ is extremal, then $\Psi^*g_\Theta$ is extremal.
\end{lem}

\begin{proof}
By Proposition~\ref{prop:basic-geometry}, $\Psi$ is a fiber scaling.  Hence
$\Psi$ covers the identity on $B$, fixes $E_-$ and $E_+$, and commutes with
the fiber circle.  In particular, $\Psi_*K=K$.  Pulling back
\eqref{eq:admissible-profile-identity} gives
\[
 d(z\circ\Psi)=\Psi^*(dz)
 =-\Psi^*(\iota_K\omega_\Theta)
 =-\iota_K(\Psi^*\omega_\Theta)
\]
and
\[
 (\Psi^*g_\Theta)(K,K)
 =\bigl(g_\Theta(K,K)\bigr)\circ\Psi
 =\Theta(z\circ\Psi).
\]
The fixed-section values and the compactification data are unchanged, so
$\Psi^*g_\Theta$ is admissible with the asserted profile.  Scalar curvature
and holomorphicity of its gradient are natural under biholomorphic pullback,
which proves the last assertion.
\end{proof}

\begin{proof}[Proof of Proposition~\ref{prop:no-extremal-metric}]
Positive rescaling preserves extremality, so it is enough to consider
$\Omega_0=4\pi c_1(A)$.  Suppose that $\Omega_0$ contains an extremal metric
$g_0$.  We construct admissible extremal metrics in nearby classes and pass
their momentum-profile identity to the limit.

\medskip

\noindent\textbf{Step 1.} In this step, we construct admissible extremal
metrics converging to $g_0$.
By Proposition~\ref{prop:basic-geometry}, the fiber circle is the unique
maximal compact connected subgroup of $\Aut^0(X)=\mathbb C^*$.
Theorem~\ref{thm:maximal-compactness} implies that this circle acts
isometrically on $g_0$.  Let $K$ be its period-$2\pi$ real holomorphic
generator, oriented as in
\eqref{eq:admissible-profile-identity}.

Choose a sequence $t_k\downarrow0$ as in
Lemma~\ref{lem:nearby-positive-polynomials}, and put
$\Omega_k=\Omega(x_{t_k})$.  Lemma~\ref{lem:variable-admissible-classes}
shows that these are circle-invariant K\"ahler classes with fiber integral
$4\pi$ and that $\Omega_k\to\Omega_0$.  By
Theorem~\ref{thm:constructive-admissible-extremal}, there exists an
admissible extremal metric $\widetilde g_k\in\Omega_k$ with profile
$\Theta_{t_k}$.  By
Theorem~\ref{thm:extremal-openness}, there also exist circle-invariant
extremal metrics $g_k\in\Omega_k$ such that
\begin{equation}\label{eq:nearby-extremal-convergence}
 g_k\to g_0
 \qquad\text{in }C^2.
\end{equation}
Theorem~\ref{thm:extremal-uniqueness} gives
$\Psi_k\in\Aut^0(X)$ such that
\begin{equation}\label{eq:nearby-extremal-pullback}
 g_k=\Psi_k^*\widetilde g_k.
\end{equation}
Let $\omega_k=g_k(J\mathord\cdot,\mathord\cdot)$ and put
$z_k=z_{\widetilde g_k}\circ\Psi_k$.  Lemma~\ref{lem:admissible-naturality}
gives
\begin{equation}\label{eq:nearby-profile-identity}
 dz_k=-\iota_K\omega_k,
 \qquad
 z_k\vert_{E_-}=-1,
 \qquad
 z_k\vert_{E_+}=1,
 \qquad
 g_k(K,K)=\Theta_{t_k}(z_k).
\end{equation}

\medskip

\noindent\textbf{Step 2.} In this step, we prove uniform convergence of the
normalized momentum coordinates.  Fix a background Riemannian metric on
$X$, let $D$ be its diameter, and choose $q_-\in E_-$.  The convergence in
\eqref{eq:nearby-extremal-convergence} implies that
$\omega_k\to\omega_0=g_0(J\mathord\cdot,\mathord\cdot)$ uniformly.  Hence
$dz_k=-\iota_K\omega_k$ is uniformly Cauchy.  Since $z_k(q_-)=-1$, for all
$k,l$ and all $p\in X$, integration along a minimizing background geodesic
from $q_-$ to $p$ gives
\begin{equation}\label{eq:momentum-cauchy-estimate}
 \lvert z_k(p)-z_l(p)\rvert
 \leq D\lVert dz_k-dz_l\rVert_{C^0}.
\end{equation}
Thus $z_k$ converges uniformly to a continuous function
$\zeta\colon X\to[-1,1]$.  Moreover,
\begin{equation}\label{eq:momentum-uniform-convergence}
 z_k\to\zeta\quad\text{uniformly},
 \qquad
 \zeta\vert_{E_-}=-1,
 \qquad
 \zeta\vert_{E_+}=1.
\end{equation}

\medskip

\noindent\textbf{Step 3.} We use the interior zero of the limiting profile
to obtain a contradiction.  Fix $b\in B$.  The fiber
$X_b\simeq\mathbb P^1$ meets $E_-$ and $E_+$, where $\zeta$ has values
$-1$ and $1$.  The intermediate value theorem therefore gives a point
$q\in X_b$ such that
\begin{equation}\label{eq:limiting-momentum-zero-point}
 \zeta(q)=z_0=2-\sqrt5.
\end{equation}
This value lies in $(-1,1)$, so $q$ is not on either fixed section.  On each
fiber the circle acts by
$[u:v]\mapsto[u:e^{\sqrt{-1}\theta}v]$ and its generator vanishes only at
$[1:0]$ and $[0:1]$.  Hence
\begin{equation}\label{eq:nonzero-generator-at-limit-point}
 K(q)\ne0.
\end{equation}

By \eqref{eq:profile-uniform-convergence} and
\eqref{eq:momentum-uniform-convergence},
\begin{equation}\label{eq:composed-profile-convergence}
 \Theta_{t_k}(z_k(q))\to\Theta_0(\zeta(q)).
\end{equation}
The absolute value of the difference in
\eqref{eq:composed-profile-convergence} is at most
$\lVert\Theta_{t_k}-\Theta_0\rVert_{C^0}$ plus
$\lvert\Theta_0(z_k(q))-\Theta_0(\zeta(q))\rvert$.
Passing to the limit in \eqref{eq:nearby-profile-identity} and using
Lemma~\ref{lem:normalized-extremal-polynomial}, we obtain
\begin{equation}\label{eq:threshold-profile-contradiction}
 g_0(K,K)(q)
 =\Theta_0(z_0)
 =\frac{F_{x^*}(z_0)}{p_c(z_0)}
 =0.
\end{equation}
This contradicts the positive definiteness of $g_0$ and
\eqref{eq:nonzero-generator-at-limit-point}.  Thus $\Omega_0$ contains no
extremal metric.  Rescaling by $1/(4\pi)$ proves that $c_1(A)$ contains no
extremal metric, and a cscK metric is extremal.
\end{proof}
 
\section{Affine transforms of the two fiber-scaling initial filtrations}
\label{sec:affine-initials}

This section defines the two initial filtrations obtained from the opposite
directions of fiber scaling and proves that either convex transform is affine
in the fiber coordinate.

\subsection{Fixed Veronese gradings and rational rays}
\label{subsec:fixed-veronese-rays}

We first fix the degree convention and identify the limiting measures on the
five-dimensional Newton--Okounkov body and on its rational slices.

\begin{setup}\label{setup:initial-filtrations}
Notation and conditions are as in Construction~\ref{cons:explicit-fivefold}.
Thus
\begin{equation}\label{eq:initial-base-data}
 B=\prod_{i=0}^3C_i,
 \qquad
 M=\boxtimes_{i=0}^3M_i,
 \qquad
 L=\boxtimes_{i=0}^3L_i.
\end{equation}
The genera and the degrees are
\begin{align}
 (g_0,g_1,g_2,g_3)&=(3846511,10591,76,46),
 \label{eq:initial-genera}\\
 (\deg M_i)_{i=0}^3&=(461999,13962,1068,260),
 \label{eq:initial-M-degrees}\\
 (\deg L_i)_{i=0}^3&=(13397971,-11635,-712,-104).
 \label{eq:initial-L-degrees}
\end{align}
The Jacobians satisfy
\begin{equation}\label{eq:initial-hom-orthogonality}
 \Hom\bigl(\Pic^0(C_i),\Pic^0(C_j)\bigr)=0
 \qquad\text{for }i\ne j.
\end{equation}
We use the quotient convention in
\begin{equation}\label{eq:initial-fivefold}
 X=\mathbb P_B(\mathcal O_B\oplus L),
 \qquad
 A=\mathcal O_X(1)\otimes\pi^*M.
\end{equation}
For an actual power $d$ of $A$, fiber scaling gives
\begin{equation}\label{eq:affine-initials-blocks}
 H^0(X,A^d)=\bigoplus_{j=0}^dV_{d,j},
 \qquad
 V_{d,j}=H^0(B,M^d\otimes L^j).
\end{equation}

Fix a positive integer $e$, and let $\mathcal T$ be a normal ample
algebraic test configuration whose generic polarized fiber is $(X,A^e)$ and
such that $\DF(\mathcal T)=0$.  Exponent-one degree $m$ for $(X,A^e)$ is
called the \emph{section degree}; it corresponds to the actual $A$-degree
\begin{equation}\label{eq:affine-initials-actual-degree}
 d=em.
\end{equation}
The supported section ring is
\begin{equation}\label{eq:initial-veronese-ring}
 R^{[e]}=\bigoplus_{m\geq0}H^0(X,A^{em}).
\end{equation}
Let $\mathcal F_{\leq a}H^0(X,A^d)$ be the subspace whose increasing entries
are at most $a$.  Define two one-parameter subgroups on
\eqref{eq:affine-initials-blocks} by
\begin{equation}\label{eq:opposite-fiber-scaling-subgroups}
 \lambda_+(z)\vert_{V_{d,j}}=z^j\operatorname{id},
 \qquad
 \lambda_-(z)\vert_{V_{d,j}}=z^{-j}\operatorname{id}.
\end{equation}
For $\epsilon\in\{+,-\}$ and every $a,d$, put
\begin{equation}\label{eq:opposite-initial-filtered-pieces}
 \mathcal F^\epsilon_{\leq a}H^0(X,A^d)
 =\lim_{z\to0}\lambda_\epsilon(z)
   \mathcal F_{\leq a}H^0(X,A^d),
\end{equation}
where the limit is taken in the Grassmannian of subspaces of the fixed
dimension.  These filtered pieces define the two fiber-scaling initial
filtrations $\chi^\epsilon$.  Neither initial filtration is recentered or
shifted by a character.  Denote the convex transform of
$\chi^\epsilon$ in actual degree by $G^\epsilon$.  If $\tau(s)$ is a
decreasing-filtration jump, we use
the increasing-entry convention
\begin{equation}\label{eq:affine-initials-sign}
 i(s)=-\tau(s),
 \qquad
 i(st)\leq i(s)+i(t).
\end{equation}

For $0\leq x\leq1$, put
\begin{align}
 h_i(x)&=\deg M_i+x\deg L_i,
 \label{eq:affine-initials-side-lengths}\\
 \Omega_x&=\prod_{i=0}^3[0,h_i(x)],
 \qquad
 \Omega=\bigcup_{0\leq x\leq1}\{x\}\times\Omega_x.
 \label{eq:affine-initials-body}
\end{align}
All four functions $h_i$ are positive on $[0,1]$.  Let $\mu_\Omega$ be
normalized five-dimensional Lebesgue measure on $\Omega$, and define
\begin{equation}\label{eq:rho-definition}
 p(x)=\prod_{i=0}^3h_i(x),
 \qquad
 \rho=\frac{p(x)\,dx}{\int_0^1p(t)\,dt}.
\end{equation}
The $x$-marginal of $\mu_\Omega$ is $\rho$.

When one sign $\epsilon$ is fixed, we suppress it from the notation.
Write $i_{d,j,\gamma}$ for the entries of $\chi^\epsilon$ on
$V_{d,j}$, counted with multiplicity, and put
\begin{equation}\label{eq:affine-initials-block-means}
 n_{d,j}=\dim V_{d,j},
 \qquad
 \overline i_{d,j}=\frac1{n_{d,j}}\sum_\gamma i_{d,j,\gamma}.
\end{equation}
\end{setup}

We record five external results in the precise forms used in this section.

\begin{thm}[Filtered Okounkov equidistribution,
{\cite[Theorem~1.11 and Remark~1.12(i)]{BC11}}]
\label{thm:filtered-okounkov-equidistribution}
Let a graded linear series containing an ample series carry a multiplicative,
linearly bounded filtration, and choose a valuation with one-dimensional
leaves.  The normalized filtration measures converge to the pushforward of
normalized Lebesgue measure on the Newton--Okounkov body by the associated
concave transform.  After reversing the sign of the filtration jumps, this
convergence statement holds for the convex transform in the increasing-entry
convention.
\end{thm}

\begin{thm}[Normal generation on a smooth curve,
{\cite[Theorem~4.2]{FT14}}]
\label{thm:curve-normal-generation}
Let $C$ be a smooth projective curve of genus $g$, and let $N$ be a line
bundle on $C$ with $\deg N\geq2g+1$.  Then the complete section ring of
$N$ is generated in degree one.
\end{thm}

Mumford's curve theorem builds on Gieseker's method; see
\cite[Theorem~4.15]{Mum77} and Gieseker's later account \cite{Gie82}.

\begin{thm}[Chow stability of embedded curves,
{\cite[Theorem~4.15]{Mum77}}]
\label{thm:high-degree-curve-chow}
Let $C$ be a smooth projective curve of genus $g\geq 1$, and let $N$ be a
line bundle of degree at least $2g+1$.  The embedding defined by the
complete linear system of $N$ is Chow stable.
\end{thm}

The balanced-metric criterion is due independently to Luo and Zhang
\cite{Luo98,Zha96}.  We use the modern formulation in
\cite[Theorem~2]{AH15}.

\begin{thm}[Luo--Zhang balanced-metric criterion,
{\cite[Theorem~2]{AH15}}]
\label{thm:balanced-chow-criterion}
Let $(Y,H)$ be a compact polarized projective manifold.  Then $(Y,H)$
is Chow polystable if and only if $H$ admits a balanced Hermitian metric.
\end{thm}

\begin{thm}[Asymptotic Chow stability,
{\cite[Corollary~4]{Don01}}]
\label{thm:donaldson-asymptotic-chow}
Let $(Y,H)$ be a polarized manifold admitting a cscK metric in $c_1(H)$,
and assume that its polarized automorphism group is discrete.  Then
$(Y,H^r)$ is Chow stable for every sufficiently large $r$.
\end{thm}

\begin{lem}[Uniform Chow polystability of the base blocks]
\label{lem:uniform-block-chow}
Notation and conditions are as in Set-up~\ref{setup:initial-filtrations}.
There exists an integer $d_0$ such that, whenever $d\geq d_0$ and
$0\leq j\leq d$, the complete linear-system embedding of $B$ defined by
$M^d\otimes L^j$ is Chow polystable.
\end{lem}

\begin{proof}
Put
\begin{equation}\label{eq:uniform-block-epsilon}
 \epsilon_i=\min\{\deg M_i,\deg M_i+\deg L_i\}>0
 \qquad(0\leq i\leq3).
\end{equation}
The degree of the restriction of $M^d\otimes L^j$ to $C_i$ is
\begin{equation}\label{eq:uniform-block-factor-degree}
 d\deg M_i+j\deg L_i\geq d\epsilon_i.
\end{equation}
Choose $d_0$ so that $d_0\epsilon_i\geq2g_i+1$ for every $i$.
Theorem~\ref{thm:high-degree-curve-chow} shows that all four factor
embeddings are Chow stable.  By
Theorem~\ref{thm:balanced-chow-criterion}, each factor polarization admits
a balanced Hermitian metric $h_i$.

Let $h=\boxtimes_i h_i$ on $M^d\otimes L^j$.  The tensor-product isomorphism
identifies an external tensor product of orthonormal factor bases with an
orthonormal basis of the product section space.  The Bergman function of
this basis is the product of the four constant factor Bergman functions,
and is therefore constant.  Thus $h$ is balanced.  A second application
of Theorem~\ref{thm:balanced-chow-criterion} proves the Chow polystability
of the product embedding.
\end{proof}

The following normalization ties the Chow stability statements above to
the filtration invariants used in the rest of this section.  Let $(V,N)$
be a polarized projective variety of dimension $n$, and let $\eta$ be a
multiplicative linearly bounded filtration of the complete section ring
$S=\bigoplus_{k\geq0}H^0(V,N^k)$ with rational entries.  Fix $k\geq1$
such that the Veronese subring $\bigoplus_{\ell\geq0}S_{k\ell}$ is
generated in degree one, and generate a filtration of this subring from
the weighted space $S_k$.  Regrading $S_{k\ell}$ into degree $\ell$, the
Hilbert and increasing-entry total functions of the generated filtration
admit two-term expansions
\begin{equation}\label{eq:chow-weight-expansions}
 h(\ell)=\alpha_0(k)\ell^{n}+O(\ell^{n-1}),
 \qquad
 W(\ell)=\beta_0(k)\ell^{n+1}+O(\ell^{n}),
\end{equation}
because the total weight function of an ample test configuration is a
polynomial of degree at most $n+1$ in all sufficiently large degrees
\cite{Don02}, \cite[Theorem~3.1]{BHJ17}.  Writing $\overline i_k$ for
the mean entry of $\eta$ on $S_k$, the \emph{degree-$k$ Chow weight} and
the \emph{intrinsic asymptotic Chow invariant} of $\eta$ are
\begin{equation}\label{eq:chow-weight-normalization}
 \operatorname{Chow}_k(\eta)
 =\overline i_k-\frac{\beta_0(k)}{\alpha_0(k)},
 \qquad
 \operatorname{Chow}_\infty(\eta)
 =\liminf_{k\to\infty}\operatorname{Chow}_k(\eta).
\end{equation}
Adding a constant to all entries in degree $k$ shifts $\overline i_k$
and $\beta_0(k)/\alpha_0(k)$ by the same amount, and a positive
rescaling of the entries rescales both terms linearly, so
$\operatorname{Chow}_k(\eta)$ is computed by the centered integral flag
obtained from the degree-$k$ entries by subtracting the mean and
clearing denominators.  The one-parameter subgroup associated with this
centered flag acts on $S_k$ with weights the negatives of the centered
entries, in the convention of Definition~\ref{defn:test-configuration};
its induced degeneration of the section ring is the generated
filtration, with total central weight
$-W(\ell)+\ell\,\overline i_k\,h(\ell)$ in degree $\ell$, and by
Mumford's weight formula the Hilbert--Mumford weight of the Chow point
of the embedded cycle $V\subset\mathbb P(S_k^{\vee})$, computed on this
degeneration, is a positive multiple of its leading normalized
coefficient, that is, of $\operatorname{Chow}_k(\eta)$
\cite[Theorem~2.9]{Mum77}.  Chow
semistability of the embedding defined by $N^k$, and in particular the
Chow stability and polystability statements above, therefore implies
\begin{equation}\label{eq:chow-weight-bridge}
 \operatorname{Chow}_k(\eta)\geq0
\end{equation}
for every such flag supported in degree $k$.  This is the normalization
in which the strictness theorem of \cite{Sze15} is quoted in
Section~\ref{subsec:rational-slice-constancy}.

\medskip

\begin{lem}[Fixed-Veronese normalization]
\label{lem:affine-initials-veronese}
Notation and conditions are as in Set-up~\ref{setup:initial-filtrations}.
Fix $\epsilon\in\{+,-\}$ and write $\chi=\chi^\epsilon$.  Then
$\chi$ is block preserving with respect to
\eqref{eq:affine-initials-blocks}, multiplicative, integer valued, and
two-sided linearly bounded on $R^{[e]}$.  In every supported actual degree
$d=em$, its entry multiset agrees with that of $\mathcal T$.

Normalize valuation coordinates and filtration entries by $d$.  The
Newton--Okounkov body is $\Omega$, and $\chi$ has a finite convex transform
$G\colon\Omega\to\mathbb R$.  The entry probability laws in actual degree
of $\mathcal T$ and $\chi$ are both
\begin{equation}\label{eq:affine-initials-actual-law}
 G_*\mu_\Omega.
\end{equation}
If $\widetilde G$ is the transform in the exponent-one grading of
$(X,A^e)$, then its body is $e\Omega$ and
\begin{equation}\label{eq:affine-initials-transform-regrading}
 \widetilde G(eu)=eG(u).
\end{equation}
Consequently, the exponent-one entry law is
\begin{equation}\label{eq:affine-initials-exponent-one-law}
 (eG)_*\mu_\Omega.
\end{equation}

Let $x=P/Q\in(0,1)\cap\mathbb Q$, where $P,Q$ may be replaced by a common
positive multiple so that
\begin{equation}\label{eq:affine-initials-ray-divisibility}
 e\mid Q,
 \qquad
 Q\deg M_i+P\deg L_i\geq2g_i+1
 \quad(0\leq i\leq3).
\end{equation}
Then the exact-ray ring
\begin{equation}\label{eq:affine-initials-ray-ring}
 \bigoplus_{k\geq0}V_{Qk,Pk}
 =\bigoplus_{k\geq0}H^0(B,(M^Q\otimes L^P)^k)
\end{equation}
is graded by the integer $k$ in this display; we call $k$ its \emph{ray
degree}.  The ring is generated in ray degree one.  The restricted filtration has transform
$y'\mapsto QG(x,y'/Q)$ on $Q\Omega_x$, nonnegative asymptotic Chow
invariant, and squared norm
\begin{equation}\label{eq:affine-initials-ray-norm}
 \lVert\chi_{P,Q}\rVert_2^2
 =Q^6\vol(\Omega_x)
 \int_{\Omega_x}
 \left(G(x,y)-\frac1{\vol(\Omega_x)}
 \int_{\Omega_x}G(x,z)\,dz\right)^2d\nu_x(y),
\end{equation}
where $\nu_x$ is normalized four-dimensional Lebesgue measure on
$\Omega_x$.
\end{lem}

\begin{proof}
We prove the statement in three steps.  The first concerns the degreewise
initial filtration, the second compares the two gradings, and the third
restricts to a rational ray.

\medskip
\noindent\textbf{Step 1.} In this step, we identify the body and the
limiting law in actual degree.  In a fixed section degree, torus
specialization is a Grassmannian limit of every filtered subspace.  Such a
limit preserves dimension.  Taking successive differences of
filtered-piece dimensions shows that the complete entry multiset is
unchanged.  The limiting subspaces are invariant under fiber scaling and
hence are direct sums of the blocks $V_{d,j}$.  The multiplication
inclusions defining a multiplicative filtration are closed conditions on
the relevant Grassmannians, so they pass to the limit.  Integrality and the
common linear bounds are also unchanged.

The valuation defined by the fiber index and product points has
one-dimensional leaves.  In
actual degree $d$, its normalized value semigroup has closed convex body
$\Omega$.  Restricting the degree semigroup to $e\mathbb N$ does not change
the normalized cone or its degree-one slice.
Theorem~\ref{thm:filtered-okounkov-equidistribution} identifies the limiting entry
law with $G_*\mu_\Omega$.  Equality of the degreewise entry multisets
identifies this law with that of $\mathcal T$.

\medskip
\noindent\textbf{Step 2.} In this step, we compare actual degree with
exponent-one degree.  Exponent-one degree $m$ for $(X,A^e)$ is actual
degree $d=em$.  Dividing a valuation vector and an entry by $m$, rather
than by $d$, multiplies both by $e$.  This proves
\eqref{eq:affine-initials-transform-regrading} and
\eqref{eq:affine-initials-exponent-one-law}.

\medskip
\noindent\textbf{Step 3.} In this step, we identify the rational-ray
transform and its norm.  Choose $P,Q$ as in
\eqref{eq:affine-initials-ray-divisibility}, and put
$N=M^Q\otimes L^P$.  Every factor of $N$ has degree at least $2g_i+1$.
Theorem~\ref{thm:curve-normal-generation}, the tensor-product decomposition,
and the factorwise multiplication maps show that the complete section ring of $N$
is generated in degree one.

Write $x=p/q$ in lowest terms and $(P,Q)=(rp,rq)$.  For the body and
transform, compare the exact-ray filtered cone with the $x$-slice of the
supported global filtered cone, working temporarily with decreasing jumps.
Represent an arbitrary point of the slice by supported sections
$s_n\in V_{d_n,j_n}$, where $e\mid d_n$, recorded at levels
$\ell_n\leq\tau(s_n)$, and with $j_n/d_n\to p/q$.  Put
$\delta_n=qj_n-pd_n$.  Choose $H\geq1$ so that both endpoint blocks
$V_{erHh,0}$ and $V_{erHh,erHh}$ are nonzero for every $h\geq1$; this is
possible because $M$ and $M\otimes L$ are ample.

If $\delta_n\geq0$, choose
$0\neq u_n\in V_{erH\delta_n,0}$, using $u_n=1$ when $\delta_n=0$, and form
\begin{equation}\label{eq:affine-initials-positive-ray-correction}
 s_n^{\,perH}u_n\in
 V_{erqHj_n,\,perHj_n}=V_{Qk_n,Pk_n},
 \qquad k_n=eHj_n.
\end{equation}
If $\delta_n<0$, put $\epsilon_n=pd_n-qj_n$, choose
$0\neq u_n\in V_{erH\epsilon_n,erH\epsilon_n}$, and form
\begin{equation}\label{eq:affine-initials-negative-ray-correction}
 s_n^{\,(q-p)erH}u_n\in
 V_{erqH(d_n-j_n),\,erpH(d_n-j_n)}=V_{Qk_n,Pk_n},
 \qquad k_n=eH(d_n-j_n).
\end{equation}
These products are nonzero because the section ring is a domain.  Record
them at the levels supplied by multiplicativity, choosing for $u_n$ an
admissible level whose absolute value is bounded linearly by its degree.
Since $\lvert\delta_n\rvert/d_n\to0$, the endpoint correction has sublinear
degree, valuation, and filtration cost relative to the powered main term.
Consequently, the normalized exact-ray points converge to the original
slice point.  The opposite inclusion is tautological, so the closed filtered
cones agree.  Reflecting back to increasing entries and normalizing by ray
degree $k$, rather than actual degree $Qk$, gives the body $Q\Omega_x$ and
the transform $y'\mapsto QG(x,y'/Q)$.  The product of
constant-curvature metrics on the four curves is cscK in $c_1(N)$, and the
polarized automorphism group is finite.
Theorem~\ref{thm:donaldson-asymptotic-chow} gives Chow stability of the
embedding defined by $N^r$ for every sufficiently large $r$, so the
Hilbert--Mumford bridge \eqref{eq:chow-weight-bridge} gives
$\operatorname{Chow}_k(\chi_{P,Q})\geq0$ for every sufficiently large
$k$, and the intrinsic asymptotic Chow invariant
\eqref{eq:chow-weight-normalization} is nonnegative.
Theorem~\ref{thm:filtered-okounkov-equidistribution}, applied to the exact ray,
gives
\eqref{eq:affine-initials-ray-norm}.  All degrees used here are divisible
by $e$.
\end{proof}

\subsection{Multiplication along fixed rational pairs}
\label{subsec:fixed-pair-multiplication}

Lemmas~\ref{lem:one-curve-multiplication-projector} and~
\ref{lem:quotient-filtration-plucker}, together with
Lemma~\ref{lem:affine-initials-fixed-pair}, compare the mean entries on two fixed
rational blocks with the mean entry on their midpoint block, without
requiring the two filtration flags to split across the four curve factors.

\begin{lem}[The one-curve multiplication projector]
\label{lem:one-curve-multiplication-projector}
Let $C$ be a smooth projective curve of genus $g>1$, and let $E,H$ be
positive line bundles on $C$.  Choose constant-curvature Hermitian metrics
on $E,H$, and $E\otimes H$, and use the resulting $L^2$-products.  Put
\begin{equation}\label{eq:curve-multiplication-spaces}
 U_k=H^0(C,E^k),\qquad
 V_k=H^0(C,H^k),\qquad
 W_k=H^0(C,(E\otimes H)^k),
\end{equation}
and let $q_k\colon U_k\otimes V_k\to W_k$ be multiplication.  For all
sufficiently large $k$, the map $q_k$ is surjective.  Put
\begin{equation}\label{eq:one-curve-dimensions}
 u_k=\dim U_k,\qquad v_k=\dim V_k,\qquad w_k=\dim W_k.
\end{equation}
Let $\mathsf P_k$ be the orthogonal projector onto
$(\ker q_k)^\perp$.  Then there exist $C,\epsilon>0$ such that
\begin{equation}\label{eq:one-curve-projector-estimate}
\begin{split}
 \left\lVert\operatorname{Tr}_{V_k}\mathsf P_k
       -\frac{w_k}{u_k}I_{U_k}\right\rVert_{\mathrm{op}}
 &\leq Ce^{-\epsilon k},\\
 \left\lVert\operatorname{Tr}_{U_k}\mathsf P_k
       -\frac{w_k}{v_k}I_{V_k}\right\rVert_{\mathrm{op}}
 &\leq Ce^{-\epsilon k}.
\end{split}
\end{equation}
\end{lem}

\begin{proof}
We prove Lemma~\ref{lem:one-curve-multiplication-projector} in five steps.
The first fixes the metrics, the next two
construct the compact Bergman kernels from the disk kernels, the fourth
compares their products, and the fifth derives the partial-trace estimates.

\medskip
\noindent\textbf{Step 1.} In this step, we choose constant-curvature
metrics and compute the dimensions.  Normalize the hyperbolic area form by
\begin{equation}\label{eq:curve-probability-area}
 \eta=\frac{\omega_{\mathrm{hyp}}}{4\pi(g-1)},
 \qquad
 \int_C\eta=1.
\end{equation}
For a positive line bundle $J$ of degree $d$, start with a Hermitian
metric $h_0$ and put
$\Theta_0=(\sqrt{-1}/2\pi)F_{h_0}$.  Since
$\int_C(\Theta_0-d\eta)=0$, the scalar Poisson equation gives a smooth real
function $\varphi$ such that
\begin{equation}\label{eq:curve-poisson-equation}
 \frac{\sqrt{-1}}{2\pi}\partial\bar\partial\varphi
 =d\eta-\Theta_0.
\end{equation}
Replacing $h_0$ by $h_0e^{-\varphi}$, with the corresponding Chern
curvature convention, gives
\begin{equation}\label{eq:curve-constant-curvature}
 \frac{\sqrt{-1}}{2\pi}F_h=d\eta.
\end{equation}
We carry out this construction for $E$ and $H$, and use the product metric
on $E\otimes H$.  If $a=\deg E$ and $b=\deg H$, then, for all sufficiently
large $k$, the three relevant degrees exceed $2g-2$ and the dimensions are
\begin{equation}\label{eq:curve-space-dimensions}
 u_k=ak+1-g,\qquad
 v_k=bk+1-g,\qquad
 w_k=(a+b)k+1-g,\qquad
 c_k=\frac{u_kv_k}{w_k}.
\end{equation}

\medskip
\noindent\textbf{Step 2.} In this step, we compute the exact disk kernel.
By uniformization \cite[Theorem~4.4.1]{Jos06}, after identifying the upper
half-plane with the disk, write $C=\Gamma\backslash\mathbb D$, where
$\Gamma<\Aut(\mathbb D)$ is torsion-free and cocompact, and normalize
\begin{equation}\label{eq:disk-metric-area}
 ds^2=\frac{4\lvert dz\rvert^2}{(1-\lvert z\rvert^2)^2},
 \qquad
 \eta=\frac{dx\,dy}{\pi(g-1)(1-\lvert z\rvert^2)^2}.
\end{equation}
For a degree-$d$ constant-curvature line bundle, the pullback to $\mathbb D$
is holomorphically trivial: after choosing a smooth frame, we solve its
scalar $\bar\partial$-equation on the disk.  Fix a holomorphic frame $e_0$.
With the convention
$F_h=-\partial\bar\partial\log\lvert e_0\rvert_h^2$, the identity
$-\partial\bar\partial\log(1-\lvert z\rvert^2)
=dz\wedge d\bar z/(1-\lvert z\rvert^2)^2$, together with
\eqref{eq:curve-constant-curvature} and \eqref{eq:disk-metric-area}, shows that
$u=\log\lvert e_0\rvert_h^2
-d\log(1-\lvert z\rvert^2)/(g-1)$ is harmonic.  Since $\mathbb D$ is simply
connected, write $u=2\operatorname{Re}\psi$ with $\psi$ holomorphic.
Replacing $e_0$ by $e^{-\psi}e_0$ gives a holomorphic frame $e$ in which
\begin{equation}\label{eq:disk-lifted-metric}
 \lvert e(z)\rvert_h^2=(1-\lvert z\rvert^2)^{\lambda_d},
 \qquad
 \lambda_d=\frac d{g-1}.
\end{equation}
Because the metric and bundle are pulled back from $C$, every deck
transformation acts holomorphically and unitarily.  In the frame $e$, these
actions are given by holomorphic factors of automorphy satisfying the cocycle
law; \textbf{Step 3} uses these deck identifications invariantly through the maps
$\rho^J_{\gamma,w}$.
The lifted norm of $fe$ is
\begin{equation}\label{eq:disk-lifted-norm}
 \lVert fe\rVert^2=\frac1{\pi(g-1)}
 \int_{\mathbb D}\lvert f(z)\rvert^2
 (1-\lvert z\rvert^2)^{\lambda_d-2}\,dx\,dy.
\end{equation}
The monomials are orthogonal and satisfy
\begin{equation}\label{eq:disk-monomial-norm}
 \lVert z^n\rVert^2
 =\frac{n!\,\Gamma(\lambda_d-1)}
 {(g-1)\Gamma(n+\lambda_d)}.
\end{equation}
Summing the binomial series gives the exact weighted disk kernel
\begin{equation}\label{eq:affine-initials-disk-kernel}
 \widetilde\Pi_d(z,w)
 =(d+1-g)(1-z\overline w)^{-d/(g-1)}.
\end{equation}
Its invariant norm is
\begin{equation}\label{eq:disk-kernel-invariant-norm}
 \lvert\widetilde\Pi_d(z,w)\rvert
 =(d+1-g)
 \cosh\left(\frac{d_{\mathbb D}(z,w)}2\right)^{-d/(g-1)}.
\end{equation}
In compatible product frames,
\begin{equation}\label{eq:affine-initials-disk-product}
 \widetilde\Pi_{d_1}(z,w)\widetilde\Pi_{d_2}(z,w)
 =\frac{(d_1+1-g)(d_2+1-g)}{d_1+d_2+1-g}
  \widetilde\Pi_{d_1+d_2}(z,w).
\end{equation}

\medskip
\noindent\textbf{Step 3.} In this step, we justify reproduction and
periodize the disk kernel.  Put
\begin{equation}\label{eq:disk-weighted-measure}
 d\mu_\lambda(w)=\frac1{\pi(g-1)}
 (1-\lvert w\rvert^2)^{\lambda-2}\,dx\,dy.
\end{equation}
If $f$ is holomorphic and belongs to $A^1(d\mu_\lambda)$, the kernel
reproduces $f$.  For polynomials this follows by expanding the kernel and
using angular integration and \eqref{eq:disk-monomial-norm}.  For a general
$f$, let $f_r(w)=f(rw)$.  Rotation invariance of the weight, strong
continuity of rotations on weighted $L^1$, and the integral representation
\[
 f(rw)=\int_0^{2\pi}P_r(\theta)f(e^{i\theta}w)
 \,\frac{d\theta}{2\pi}
\]
show that $f_r\to f$ in the weighted $A^1$-norm as $r\uparrow1$.
Each $f_r$ is holomorphic on a disk larger than the closed unit disk, so
its Taylor polynomials converge uniformly there.  Since the scalar kernel
is bounded in $w$ for fixed $z$, the polynomial identity passes first
to $f_r$ and then to $f$.  Thus
\begin{equation}\label{eq:disk-reproduction}
 f(z)=\int_{\mathbb D}
 (d+1-g)(1-z\overline w)^{-\lambda}f(w)\,d\mu_\lambda(w).
\end{equation}
If $s=fe$ is the lift of a section on $C$, its invariant pointwise norm
is bounded, and hence
\begin{equation}\label{eq:automorphic-growth}
 \lvert f(z)\rvert\leq C(1-\lvert z\rvert^2)^{-\lambda/2}.
\end{equation}
For $\lambda>2$, the right-hand side belongs to
$A^1(d\mu_\lambda)$.  Therefore \eqref{eq:disk-reproduction} applies to
all lifted compact sections in the sufficiently positive degrees used in
Steps~4 and~5 of this proof.

Let
\[
 \rho^J_{\gamma,w}\colon\widetilde J_w\to
 \widetilde J_{\gamma w}
\]
be the unitary deck identification for a lifted constant-curvature bundle
$J$, and define
\begin{equation}\label{eq:deck-kernel-term}
 K^J_\gamma(z,w)
 =\widetilde\Pi_d(z,\gamma w)\circ\rho^J_{\gamma,w}.
\end{equation}
The compact Bergman kernel is
\begin{equation}\label{eq:compact-kernel-periodization}
 \Pi_J(z,w)=\sum_{\gamma\in\Gamma}K^J_\gamma(z,w).
\end{equation}
Packing disjoint hyperbolic balls gives the orbit estimate
\begin{equation}\label{eq:hyperbolic-orbit-count}
 \#\{\gamma\mid n\leq d(z,\gamma w)<n+1\}\leq Ce^n.
\end{equation}
Together with \eqref{eq:disk-kernel-invariant-norm}, this gives normal
convergence in sufficiently positive degrees.  For a fundamental domain
$\mathcal F$ and a lifted compact section $\widetilde s$, absolute
convergence, automorphy, and \eqref{eq:disk-reproduction} give
\begin{equation}\label{eq:compact-kernel-reproduction}
\begin{aligned}
 \int_{\mathcal F}\sum_\gamma
 K^J_\gamma(z,w)\widetilde s(w)\eta(w)
 &=\sum_\gamma\int_{\gamma\mathcal F}
 \widetilde\Pi_d(z,w')\widetilde s(w')\eta(w')\\
 &=\widetilde s(z).
\end{aligned}
\end{equation}
The cocycle law gives descent, and replacing $\gamma$ by
$\gamma^{-1}$ gives Hermitian symmetry.  This proves
\eqref{eq:compact-kernel-periodization}.

\medskip
\noindent\textbf{Step 4.} In this step, we compare the product of the two
compact kernels with the kernel of the tensor product.  Use
the periodization \eqref{eq:compact-kernel-periodization} for $E^k$, $H^k$, and
$(E\otimes H)^k$.  The terms with equal deck indices satisfy
\eqref{eq:affine-initials-disk-product}, and therefore
\begin{equation}\label{eq:equal-deck-identity}
 \sum_\gamma K^{E^k}_\gamma(z,w)K^{H^k}_\gamma(z,w)
 =c_k\Pi_{(E\otimes H)^k}(z,w).
\end{equation}
Let $\iota=\operatorname{inj}(C)>0$.  If $\gamma\neq\delta$, then
$d(\gamma w,\delta w)\geq2\iota$, so at least one of
$d(z,\gamma w)$ and $d(z,\delta w)$ is at least $\iota$.  In that
factor of \eqref{eq:disk-kernel-invariant-norm}, split the exponent in
half.  To display the double summation, after decreasing two positive
constants $c_E,c_H$ if necessary, put
\[
 a_\gamma=
 \cosh\left(\frac{d(z,\gamma w)}2\right)^{-c_Ek},
 \qquad
 b_\delta=
 \cosh\left(\frac{d(z,\delta w)}2\right)^{-c_Hk}.
\]
The kernel prefactors are $O(k)$, while
\eqref{eq:hyperbolic-orbit-count} gives, uniformly in $z,w$ and large $k$,
\[
 \sum_\gamma a_\gamma^{1/2}
 +\sum_\gamma a_\gamma
 +\sum_\delta b_\delta^{1/2}
 +\sum_\delta b_\delta\leq C.
\]
Partition the unequal pairs according to which of the two distances is at
least $\iota$.  One half of the corresponding far factor is at most
$e^{-\epsilon k}$, and therefore
\begin{equation}\label{eq:unequal-deck-double-sum}
\begin{aligned}
 \sum_{\gamma\ne\delta}
 \lvert K^{E^k}_\gamma(z,w)K^{H^k}_\delta(z,w)\rvert
 &\leq Ck^2e^{-\epsilon k}
 \left(
   \sum_\gamma a_\gamma^{1/2}\sum_\delta b_\delta
  +\sum_\gamma a_\gamma\sum_\delta b_\delta^{1/2}
 \right)\\
 &\leq C'e^{-\epsilon'k}.
\end{aligned}
\end{equation}
The polynomial prefactor has been absorbed by decreasing the exponential
rate.  This estimate is uniform on $C\times C$.  Hence
\begin{equation}\label{eq:affine-initials-compact-kernel-product}
 \sup_{x,y\in C}
 \left\lvert\Pi_{E^k}(x,y)\Pi_{H^k}(x,y)
 -c_k\Pi_{(E\otimes H)^k}(x,y)\right\rvert
 \leq Ce^{-\epsilon k}.
\end{equation}

\medskip
\noindent\textbf{Step 5.} In this step, we pass from kernels to partial
traces.  The Schwartz kernel of $Q_k=q_kq_k^*$ is the product of the
first two Bergman kernels.  Since $(C,\eta)$ has total mass one, the uniform
bound in \eqref{eq:affine-initials-compact-kernel-product} also bounds the
Hilbert--Schmidt norm of the error operator and gives
\begin{equation}\label{eq:affine-initials-multiplication-scalar}
 \lVert Q_k-c_kI_{W_k}\rVert_{\mathrm{op}}\leq Ce^{-\epsilon k}.
\end{equation}
Since $c_k\asymp k$, the operator $Q_k$ is invertible for large $k$,
so $q_k$ is surjective.  The required projector is
\begin{equation}\label{eq:curve-multiplication-projector}
 \mathsf P_k=q_k^*(q_kq_k^*)^{-1}q_k.
\end{equation}

On the diagonal of \eqref{eq:compact-kernel-periodization}, the
term indexed by the identity deck transformation is the dimension, and the
other terms are exponentially
small.  Thus
\begin{equation}\label{eq:curve-bergman-densities}
 \rho_{E^k}=u_k+O(e^{-\epsilon k}),
 \qquad
 \rho_{H^k}=v_k+O(e^{-\epsilon k}).
\end{equation}
For an orthonormal basis $(t_j)$ of $V_k$, direct expansion gives
\begin{equation}\label{eq:curve-partial-trace-integral}
 \left\langle\operatorname{Tr}_{V_k}(q_k^*q_k)s,s'\right\rangle
 =\int_Ch_{E^k}(s,s')\sum_j\lvert t_j\rvert^2\eta.
\end{equation}
Consequently,
\begin{equation}\label{eq:curve-qq-partial-traces}
 \operatorname{Tr}_{V_k}(q_k^*q_k)=v_kI+O(e^{-\epsilon k}),
 \qquad
 \operatorname{Tr}_{U_k}(q_k^*q_k)=u_kI+O(e^{-\epsilon k}).
\end{equation}
Write $Q_k=c_kI+E_k$.  By
\eqref{eq:affine-initials-multiplication-scalar},
$\lVert c_k^{-1}E_k\rVert=O(k^{-1}e^{-\epsilon k})$.  Expanding
$(I+c_k^{-1}E_k)^{-1}$ as its convergent geometric series and using
$\lVert q_k\rVert^2=\lVert Q_k\rVert=O(k)$ give
\begin{equation}\label{eq:curve-projector-resolvent}
 Q_k^{-1}=c_k^{-1}I+O(k^{-2}e^{-\epsilon k}),
 \qquad
 \mathsf P_k=c_k^{-1}q_k^*q_k+O(k^{-1}e^{-\epsilon k}).
\end{equation}
A partial trace costs at most the $O(k)$ dimension of the traced factor.
Combining \eqref{eq:curve-qq-partial-traces} and
\eqref{eq:curve-projector-resolvent} with
$v_k/c_k=w_k/u_k$ and $u_k/c_k=w_k/v_k$ proves
\eqref{eq:one-curve-projector-estimate}.
\end{proof}

\begin{lem}[A Pl\"ucker estimate for quotient filtrations]
\label{lem:quotient-filtration-plucker}
Let $q\colon U\otimes V\twoheadrightarrow W$ be a surjective map of
finite-dimensional Hermitian spaces of dimensions $u,v,w$.  Give $U,V$
arbitrary increasing weighted flags with mean jumps
$\overline a,\overline b$.  If $F^U_{\leq a}U$ and $F^V_{\leq b}V$ are
their filtered pieces, give $W$ the quotient filtration
\begin{equation}\label{eq:plucker-quotient-filtration}
 F^q_{\leq c}W
 =\sum_{a+b\leq c}q\bigl(F^U_{\leq a}U\otimes F^V_{\leq b}V\bigr).
\end{equation}
Let $\overline r$ be its mean jump, and let $A,B$ be the centered
self-adjoint flag generators.  If $P_S$ is the orthogonal projector onto
$S=q^\vee W^*\subset U^*\otimes V^*$, put
\begin{equation}\label{eq:plucker-moment-errors}
 \mu_U=\operatorname{Tr}_{V^*}P_S-\frac wuI_{U^*},
 \qquad
 \mu_V=\operatorname{Tr}_{U^*}P_S-\frac wvI_{V^*}.
\end{equation}
Then
\begin{equation}\label{eq:affine-initials-plucker}
 \overline r-\overline a-\overline b
 \leq
 \frac{u\lVert\mu_U\rVert_{\mathrm{op}}\lVert A\rVert_{\mathrm{op}}
 +v\lVert\mu_V\rVert_{\mathrm{op}}\lVert B\rVert_{\mathrm{op}}}{w}.
\end{equation}
\end{lem}

\begin{proof}
Orthogonally split the two flags, choose orthonormal eigenbases
$(e_i)$ and $(f_j)$, and set
\[
 D=A\otimes I+I\otimes B.
\]
The columns $q(e_i\otimes f_j)$ have centered costs
\[
 c_{ij}=(a_i-\overline a)+(b_j-\overline b).
\]
Scanning these costs in increasing order shows that the quotient dimension
at a threshold is the rank of the columns available at that threshold.
Hence
\begin{equation}\label{eq:plucker-greedy-basis}
 J^\circ:=\sum_{\ell=1}^wr_\ell-w(\overline a+\overline b)
 =\min_I\sum_{(i,j)\in I}c_{ij},
\end{equation}
where $I$ ranges over the column sets giving a basis of $W$.

On $U^*\otimes V^*$, set
\[
 X_0=-A^{\mathsf T},\qquad
 Y_0=-B^{\mathsf T},\qquad
 H_0=X_0\otimes I+I\otimes Y_0.
\]
For a nonzero Pl\"ucker vector $s\in\bigwedge^wS$, its coordinate indexed
by $I$ is nonzero exactly when the corresponding maximal minor of $q$
is nonzero, and its $H_0$-weight is $-\sum_Ic_{ij}$.  Therefore the
largest occurring weight is
\begin{equation}\label{eq:plucker-largest-weight}
 \ell_{\max}(s)=-J^\circ.
\end{equation}
After normalizing $s$, the expectation of the additive exterior-power
generator is a convex combination of its occurring weights.  Hence
\begin{equation}\label{eq:plucker-expectation}
\begin{aligned}
 -J^\circ
 &=\ell_{\max}(s)\\
 &\geq\langle H_0^{[w]}s,s\rangle\\
 &=\operatorname{Tr}(P_SH_0)
 =\operatorname{Tr}(\mu_UX_0)+\operatorname{Tr}(\mu_VY_0).
\end{aligned}
\end{equation}
In particular,
\begin{equation}\label{eq:plucker-cost-sign-consequence}
 J^\circ\leq-\operatorname{Tr}(P_SH_0)
 \leq\left\lvert\operatorname{Tr}(P_SH_0)\right\rvert.
\end{equation}
The scalar parts vanish because $X_0,Y_0$ are trace free.  Estimating the
two remaining traces by operator norms and dividing by $w$ gives
\eqref{eq:affine-initials-plucker}.
\end{proof}

\begin{lem}[The fixed-pair mean inequality]
\label{lem:affine-initials-fixed-pair}
Notation and conditions are as in Set-up~\ref{setup:initial-filtrations}.
Fix $x,y,z\in(0,1)\cap\mathbb Q$ with $y+z=2x$, and choose a positive
integer $D$ divisible by $e$ and by the reduced denominators of
$x,y,z$.  There exist constants $C,c>0$ such that, for all sufficiently
large $k$, with $d=Dk$,
\begin{equation}\label{eq:affine-initials-fixed-pair}
 \overline i_{2d,2dx}
 \leq\overline i_{d,dy}+\overline i_{d,dz}+Cke^{-ck}.
\end{equation}
The constants may depend on the fixed test configuration, initial
filtration, denominator, rational pair, and Hermitian data, but not on
$k$.  No uniformity for a moving pair is asserted.
\end{lem}

\begin{proof}
We prove the inequality in three steps.  The first identifies the three
section spaces, the second tensorizes the one-curve projector estimates,
and the third applies Lemma~\ref{lem:quotient-filtration-plucker} to the
actual filtration flags.

\medskip
\noindent\textbf{Step 1.} In this step, we fix the rational-slope bundles.
Put
\begin{equation}\label{eq:fixed-pair-bundles}
 E=M^D\otimes L^{Dy}=\boxtimes_{i=0}^3E_i,
 \qquad
 H=M^D\otimes L^{Dz}=\boxtimes_{i=0}^3H_i.
\end{equation}
Every $E_i,H_i$ has positive degree.  With $d=Dk$, we have
\begin{equation}\label{eq:fixed-pair-section-spaces}
\begin{aligned}
 V_{d,dy}&=H^0(B,E^k),\\
 V_{d,dz}&=H^0(B,H^k),\\
 V_{2d,2dx}&=H^0(B,(E\otimes H)^k).
\end{aligned}
\end{equation}
The relevant multiplication is the ordinary multiplication map among
these three complete section spaces.

\medskip
\noindent\textbf{Step 2.} In this step, we tensorize the multiplication
projectors.  For each $i$,
Lemma~\ref{lem:one-curve-multiplication-projector}, applied to
$C_i,E_i,H_i$, gives the factorwise estimate.
The tensor-product decomposition and canonical regrouping identify the global
spaces and maps as
\begin{equation}\label{eq:fourfold-multiplication-tensorization}
 U_k=\bigotimes_{i=0}^3H^0(C_i,E_i^k),\quad
 V_k=\bigotimes_{i=0}^3H^0(C_i,H_i^k),\quad
 W_k=\bigotimes_{i=0}^3H^0(C_i,(E_i\otimes H_i)^k),
\end{equation}
\begin{equation}\label{eq:fourfold-projector-tensorization}
 q_k=\bigotimes_{i=0}^3q_{i,k},
 \qquad
 \mathsf P_k=\bigotimes_{i=0}^3\mathsf P_{i,k}.
\end{equation}
Partial trace commutes with tensor products.  The scalar ratios in
\eqref{eq:one-curve-projector-estimate} are bounded, and only four factors
occur.  Expanding the tensor products gives
\begin{equation}\label{eq:fourfold-partial-traces}
\begin{split}
 \left\lVert\operatorname{Tr}_{V_k}\mathsf P_k
       -\frac{w_k}{u_k}I_{U_k}\right\rVert_{\mathrm{op}}
 &\leq Ce^{-\epsilon k},\\
 \left\lVert\operatorname{Tr}_{U_k}\mathsf P_k
       -\frac{w_k}{v_k}I_{V_k}\right\rVert_{\mathrm{op}}
 &\leq Ce^{-\epsilon k}.
\end{split}
\end{equation}

The finite-dimensional flag estimate is expressed in the dual space.  Let
$R_U,R_V,R_W$ be the antiunitary Riesz maps.  Then
\begin{equation}\label{eq:fixed-pair-riesz-bridge}
 q_k^\vee R_W=(R_U\otimes R_V)q_k^*.
\end{equation}
Thus the orthogonal projector onto
$q_k^\vee W_k^*\subset U_k^*\otimes V_k^*$ is the antiunitary conjugate
of $\mathsf P_k$, and its partial traces have the errors in
\eqref{eq:fourfold-partial-traces}.  This conclusion concerns the full
groups $\operatorname{SL}(U_k)\times\operatorname{SL}(V_k)$, so it
applies to arbitrary flags on the complete section spaces.

\medskip
\noindent\textbf{Step 3.} In this step, we use the Pl\"ucker estimate for
the initial-filtration flags.  Give the two source spaces their actual
 increasing weighted flags and the target the quotient filtration defined in
\eqref{eq:plucker-quotient-filtration}.  Let $\overline r_k$ be its mean
entry.  Linear boundedness gives
\[
 \lVert A_k\rVert_{\mathrm{op}}+\lVert B_k\rVert_{\mathrm{op}}=O(k),
\]
where $A_k,B_k$ are the centered splitting generators.  The exact linear
dimension formula on each curve factor and the tensor-product decomposition
give
\begin{equation}\label{eq:fixed-pair-dimension-growth}
 u_k\asymp k^4,\qquad
 v_k\asymp k^4,\qquad
 w_k\asymp k^4.
\end{equation}
The combination of \eqref{eq:fourfold-partial-traces} and
\eqref{eq:fixed-pair-riesz-bridge}, followed by
Lemma~\ref{lem:quotient-filtration-plucker}, yields
\begin{equation}\label{eq:fixed-pair-quotient-mean}
 \overline r_k
 \leq\overline i_{d,dy}+\overline i_{d,dz}
 +O(ke^{-ck}).
\end{equation}
Finally, multiplicativity in the convention
\eqref{eq:affine-initials-sign} gives
\begin{equation}\label{eq:fixed-pair-target-quotient-comparison}
 F^q_{\leq c}W_k\subseteq F^{\mathrm{target}}_{\leq c}W_k.
\end{equation}
Thus every jump of the target filtration is at most the corresponding jump
of the quotient filtration.  Hence
$\overline i_{2d,2dx}\leq\overline r_k$, and
\eqref{eq:affine-initials-fixed-pair} follows.
\end{proof}

\subsection{The scalar profile and a summability estimate}
\label{subsec:scalar-profile}

We next use the zero Donaldson--Futaki equality to locate the curvature of
the slice-average profile and to bound the accumulated blockwise error.

\begin{lem}[Endpoint trapezoidal summation]
\label{lem:endpoint-trapezoidal-summation}
Let $\psi$ be a finite continuous convex function on $[0,1]$, and let
$P\in C^2[0,1]$.  Then
\begin{equation}\label{eq:affine-initials-trapezoid}
 \sum_{j=0}^dP(j/d)\psi(j/d)
 =d\int_0^1P\psi
 +\frac{P(0)\psi(0)+P(1)\psi(1)}2+o(1).
\end{equation}
\end{lem}

\begin{proof}
For $I_j=[j/d,(j+1)/d]$, the exact error kernel after multiplying the
trapezoidal error by $d$ is
\begin{equation}\label{eq:affine-initials-error-kernel}
 K_d(x)=\frac d2(x-j/d)((j+1)/d-x),
 \qquad x\in I_j.
\end{equation}
It satisfies $0\leq K_d\leq1/(8d)$ and
$K_d(x)\leq Cx(1-x)$.  Distributionally,
\begin{equation}\label{eq:affine-initials-product-second-derivative}
 (P\psi)''=P\psi''+2P'\psi'+P''\psi.
\end{equation}
Finite convexity gives $\psi'\in L^1$ and
$\int_0^1x(1-x)\,d\psi''<\infty$.  Dominated convergence applies to the
$\lvert P\rvert\psi''$ term, while the uniform bound $1/(8d)$ applies to the remaining
finite measure.  This proves \eqref{eq:affine-initials-trapezoid}.
\end{proof}

\begin{lem}[One-crease profile and a summability estimate]
\label{lem:affine-initials-row-budget}
Notation and conditions are as in Set-up~\ref{setup:initial-filtrations}.
Fix one of the two filtrations $\chi^\epsilon$, and let $G$ be its
convex transform after division by the actual degree.  For $i\in\{0,1\}$,
restrict $\chi^\epsilon$ to the exact endpoint ring
\begin{equation}\label{eq:affine-initials-endpoint-rings}
 R_i=\bigoplus_{m\geq0}V_{em,iem},
\end{equation}
and let $G_i$ be the transform of this restricted filtration on $\Omega_i$,
with valuation coordinates and entries divided by the actual degree $em$.
For $0<x<1$, define
\begin{equation}\label{eq:affine-initials-slice-average}
 \phi(x)=\frac1{\vol(\Omega_x)}
 \int_{\Omega_x}G(x,y)\,dy.
\end{equation}
At the endpoints, define
\begin{equation}\label{eq:affine-initials-endpoint-averages}
 \phi(i)=\frac1{\vol(\Omega_i)}
 \int_{\Omega_i}G_i(y)\,dy
 \qquad(i=0,1).
\end{equation}
Then $\phi$ is finite and convex on $[0,1]$, satisfies
\begin{equation}\label{eq:affine-initials-endpoint-continuity}
 \phi(0)=\lim_{x\downarrow0}\phi(x),
 \qquad
 \phi(1)=\lim_{x\uparrow1}\phi(x),
\end{equation}
and has the form
\begin{equation}\label{eq:affine-initials-one-crease}
 \phi(x)=u_0+u_1x+\kappa(x-\lambda)_+,
 \qquad
 \kappa\geq0,
 \qquad
 \lambda=\frac{3-\sqrt5}{2}.
\end{equation}

For all sufficiently large supported degrees $d$, the complete section ring
of $M^d\otimes L^j$ is generated by its degree-one piece $V_{d,j}$ for every
$0\leq j\leq d$.  Generate a filtration of this ring from the weighted space
$V_{d,j}$, and write
\begin{align}
 h^{\mathrm{gen}}_{d,j}(k)&=\alpha_{0,d,j}k^4+O(k^3),
 \label{eq:affine-initials-generated-Hilbert}\\
 W^{\mathrm{gen}}_{d,j}(k)&=\beta^{\mathrm{gen}}_{0,d,j}k^5+O(k^4).
 \label{eq:affine-initials-generated-weight}
\end{align}
The two-term expansions hold as in \eqref{eq:chow-weight-expansions}.
Define
\begin{align}
 c_{d,j}&=\overline i_{d,j}
 -\frac{\beta^{\mathrm{gen}}_{0,d,j}}{\alpha_{0,d,j}},
 \label{eq:affine-initials-chow-gap}\\
 a_{d,j}&=\frac{\beta^{\mathrm{gen}}_{0,d,j}}{\alpha_{0,d,j}}
 -d\phi(j/d),
 \label{eq:affine-initials-generation-gap}\\
 g_{d,j}&=\overline i_{d,j}-d\phi(j/d)
 =c_{d,j}+a_{d,j}.
 \label{eq:affine-initials-row-gap}
\end{align}
There exist $d_0$ and $C_{\mathcal T}>0$ such that, for every supported
actual degree $d=em\geq d_0$,
\begin{equation}\label{eq:affine-initials-row-budget}
 c_{d,j}\geq0,
 \qquad
 a_{d,j}\geq0,
 \qquad
 \sum_{j=0}^dg_{d,j}\leq C_{\mathcal T}.
\end{equation}
\end{lem}

\begin{proof}
We prove Lemma~\ref{lem:affine-initials-row-budget} in four steps.  The first
proves the pointwise signs, the
second compares the exact endpoint rays with the global filtered cone, the
third identifies the equality profile, and the fourth proves the uniform
summability in \eqref{eq:affine-initials-row-budget}.

\medskip
\noindent\textbf{Step 1.} In this step, we prove the nonnegativity of
$c_{d,j}$ and $a_{d,j}$.  Center the entries on $V_{d,j}$ by
$\overline i_{d,j}$.  After clearing their common scalar denominator, we
obtain an integral one-parameter subgroup.  In the algebraic weight convention, the
weights are the negatives of the increasing entries.
After increasing $d_0$, \eqref{eq:uniform-block-factor-degree} and
Theorem~\ref{thm:curve-normal-generation}, together with the tensor-product
decomposition, show that every block ring in the statement is generated by
$V_{d,j}$.
Lemma~\ref{lem:uniform-block-chow} and the Hilbert--Mumford bridge
\eqref{eq:chow-weight-bridge} therefore give
\begin{equation}\label{eq:affine-initials-chow-gap-positive}
 -\frac{\beta^{\mathrm{gen}}_{0,d,j}
 -\overline i_{d,j}\alpha_{0,d,j}}{\alpha_{0,d,j}}
 =c_{d,j}\geq0.
\end{equation}
The filtration generated by $V_{d,j}$ uses only products of sections from
that block.  The exact-ray filtration may also use sections from other
supported degrees.  Subadditivity in \eqref{eq:affine-initials-sign}
therefore implies that the leading mean of the generated filtration is at
least the leading mean of the exact-ray filtration.  For $0<j<d$,
Lemma~\ref{lem:affine-initials-veronese} identifies the leading exact-ray
mean with $d\phi(j/d)$.  At $j=0,d$, the ring generated from $V_{d,0}$ or
$V_{d,d}$ is the $d$-th Veronese of the corresponding exact endpoint ray,
so both valuation coordinates and filtration entries are multiplied by $d$.
The leading exact-ray mean is therefore $d\phi(j/d)$ at the endpoints as
well.  Thus
\begin{equation}\label{eq:affine-initials-generation-gap-positive}
 \frac{\beta^{\mathrm{gen}}_{0,d,j}}{\alpha_{0,d,j}}
 \geq d\phi(j/d).
\end{equation}
This proves $a_{d,j}\geq0$ and $g_{d,j}\geq0$.

\medskip
\noindent\textbf{Step 2.} In this step, we compare the endpoint values.
For $0<x<1$, the change of variables
$y_i=h_i(x)t_i$ in \eqref{eq:affine-initials-slice-average} gives
\begin{equation}\label{eq:affine-initials-cube-average}
 \phi(x)=\int_{[0,1]^4}
 G\bigl(x,h_0(x)t_0,\ldots,h_3(x)t_3\bigr)\,dt.
\end{equation}
For fixed $t$, the argument of $G$ is affine in $x$, because every $h_i$ is
affine.  Convexity of $G$, followed by integration over the cube, proves
that $\phi$ is convex on $(0,1)$.
Work temporarily with decreasing jumps.  Let $\widehat{\mathcal C}$ be the
closed downward-saturated cone generated by the filtered tuples
\begin{equation}\label{eq:affine-initials-filtered-tuples}
 (d,j,\nu(s),\ell),
 \qquad
 0\neq s\in V_{d,j},
 \qquad
 \ell\leq\tau(s).
\end{equation}
The cone generated by the exact tuples with $j=0$ lies in
$\widehat{\mathcal C}\cap\{x=0\}$.  Its upper boundary is therefore at
most the upper boundary of the global face.  For $t\in(0,1)^4$, consider
the affine radial path
\begin{equation}\label{eq:affine-initials-radial-path}
 x\mapsto
 \bigl(x,h_0(x)t_0,\ldots,h_3(x)t_3\bigr).
\end{equation}
The upper boundary restricted to this path is finite and concave.
Closedness shows that its face value is at least its interior limit, while
convexity of the hypograph with a fixed interior point gives the reverse
inequality.  Thus the face value is the interior radial limit.  Reversing
the sign from decreasing jumps to increasing entries shows that the exact
endpoint transform at $x=0$ is at least the interior radial limit.  For the
radial path oriented from $x=1$ toward the interior, closedness makes the
face value at least its interior limit, while convexity of the hypograph makes
it at most that limit.  After reversing the filtration sign, the exact endpoint
transform at $x=1$ is at least its interior radial limit.  Averaging
over the fixed unit cube and using
the common linear bound gives
\begin{equation}\label{eq:affine-initials-endpoint-excesses}
 \phi(0)\geq\lim_{x\downarrow0}\phi(x),
 \qquad
 \phi(1)\geq\lim_{x\uparrow1}\phi(x).
\end{equation}
Raising the endpoint values of a convex function preserves convexity.
Hence the endpoint extension of $\phi$ is finite and convex.

\medskip
\noindent\textbf{Step 3.} In this step, we determine the support of
$\phi''$.  Set
\begin{align}
 q(x)&=\sum_{i=0}^3(1-g_i)\prod_{r\neq i}h_r(x),
 \label{eq:affine-initials-q}\\
 a_0&=\int_0^1p(x)\,dx,
 \qquad
 a_1=\frac{p(0)+p(1)}2+\int_0^1q(x)\,dx.
 \label{eq:affine-initials-a0-a1}
\end{align}
Write
\begin{equation}\label{eq:affine-initials-ci}
 c_0=461999,
 \quad c_1=2327,
 \quad c_2=356,
 \quad c_3=52,
 \quad C_{\mathrm{bd}}=\prod_{i=0}^3c_i.
\end{equation}
The boundary polynomial from Lemma~\ref{lem:boundary-polynomial} is
\begin{equation}\label{eq:affine-initials-F}
 F(x)=90C_{\mathrm{bd}}x(1-x)(x^2-3x+1)^2,
\end{equation}
and satisfies
\begin{equation}\label{eq:affine-initials-F-boundary}
 F''=2q-\frac{2a_1}{a_0}p,
 \quad F(0)=F(1)=0,
 \quad F'(0)=p(0),
 \quad F'(1)=-p(1).
\end{equation}
It is nonnegative on $[0,1]$, and its unique interior zero is the double
zero $\lambda=(3-\sqrt5)/2$.

Let $L_0,L_1$ be the interior endpoint limits of $\phi$, and put
\[
 \epsilon_i=\phi(i)-L_i.
\]
Let $\overline\phi$ denote the continuous convex extension of
$\phi\vert_{(0,1)}$ to $[0,1]$, so that
$\overline\phi(i)=L_i$.  Distributional integration by parts gives
\begin{equation}\label{eq:affine-initials-scalar-DF}
 \DF_{\mathrm{sc}}(\phi)
 =\frac1{a_0}\int_{(0,1)}F\,d\overline\phi''
 +\frac2{a_0}\bigl(p(0)\epsilon_0+p(1)\epsilon_1\bigr).
\end{equation}
Both terms on the right are nonnegative.

For a sufficiently large supported actual degree $d=em$, define the
increasing-entry total of the algebraic source by
\[
 W_{\mathcal T}(d)=\sum_{j=0}^d\sum_\gamma i_{d,j,\gamma}.
\]
Put
\begin{equation}\label{eq:affine-initials-Dd}
 D_d=\sum_{j=0}^dn_{d,j}g_{d,j}
 =W_{\mathcal T}(d)-d\sum_{j=0}^dn_{d,j}\phi(j/d).
\end{equation}
The first equality is the definition of $D_d$, and the second follows from
degreewise preservation of the entry multiset.  In particular, $D_d\geq0$.
Uniformly in $j$, every curve-factor line bundle has degree greater than
$2g_i-2$ for sufficiently large $d$, and hence has dimension equal to its
degree plus $1-g_i$.  The tensor-product decomposition gives
\begin{equation}\label{eq:affine-initials-block-RR}
 n_{d,j}
 =\prod_{i=0}^3(d\deg M_i+j\deg L_i+1-g_i)
 =d^4p(j/d)+d^3q(j/d)+O(d^2).
\end{equation}
Define
\begin{align*}
 b_0^{\mathrm{sc}}
 &=\int_0^1p(x)\overline\phi(x)\,dx,\\
 b_1^{\mathrm{sc}}
 &=\int_0^1q(x)\overline\phi(x)\,dx
   +\frac{p(0)L_0+p(1)L_1}{2}
   +p(0)\epsilon_0+p(1)\epsilon_1.
\end{align*}
Lemma~\ref{lem:endpoint-trapezoidal-summation}, first applied to
$p\overline\phi$ and then to $q\overline\phi$, together with
\eqref{eq:affine-initials-block-RR}, gives
\begin{equation}\label{eq:affine-initials-scalar-weight-expansion}
 d\sum_{j=0}^dn_{d,j}\phi(j/d)
 =b_0^{\mathrm{sc}}d^6+b_1^{\mathrm{sc}}d^5+o(d^5).
\end{equation}
The last two summands in $b_1^{\mathrm{sc}}$ are
$p(0)\epsilon_0$ and $p(1)\epsilon_1$.

Write the increasing-entry expansion of the algebraic source as
\begin{equation}\label{eq:affine-initials-source-weight-expansion}
 W_{\mathcal T}(d)
 =\widetilde b_0d^6+\widetilde b_1d^5+O(d^4).
\end{equation}
Equality of the limiting entry laws, equivalently equality of their
Duistermaat--Heckman pushforwards, gives
\begin{equation}\label{eq:affine-initials-leading-coefficient}
 \widetilde b_0=b_0^{\mathrm{sc}}.
\end{equation}
Substitution in \eqref{eq:affine-initials-Dd}, followed by $D_d\geq0$ along
all sufficiently large supported degrees, gives
\begin{equation}\label{eq:affine-initials-subleading-inequality}
 \widetilde b_1\geq b_1^{\mathrm{sc}}.
\end{equation}

In the increasing-entry convention, put
\begin{equation}\label{eq:affine-initials-source-DF-coefficients}
 \widetilde{\DF}(\mathcal T)
 =\frac{2(\widetilde b_1a_0-a_1\widetilde b_0)}{a_0^2}.
\end{equation}
In the exponent-one grading of $(X,A^e)$, the Hilbert coefficients are
$e^5a_0,e^4a_1$ and the increasing-entry coefficients are
$e^6\widetilde b_0,e^5\widetilde b_1$.  Up to the fixed nonzero
normalization and the common weight sign, the Donaldson--Futaki numerator is
\[
 e^{10}(\widetilde b_1a_0-a_1\widetilde b_0).
\]
Therefore the hypothesis $\DF(\mathcal T)=0$ forces
$\widetilde b_1a_0-a_1\widetilde b_0=0$, and
$\widetilde{\DF}(\mathcal T)=0$.  The scalar coefficient formula uses the
identical positive factor:
\begin{equation}\label{eq:affine-initials-scalar-DF-coefficients}
 \DF_{\mathrm{sc}}(\phi)
 =\frac{2(b_1^{\mathrm{sc}}a_0-a_1b_0^{\mathrm{sc}})}{a_0^2}.
\end{equation}
By \eqref{eq:affine-initials-leading-coefficient},
\eqref{eq:affine-initials-subleading-inequality}, and
\eqref{eq:affine-initials-scalar-DF}, we obtain
\begin{equation}\label{eq:affine-initials-scalar-equality}
 0=\widetilde{\DF}(\mathcal T)
 \geq\DF_{\mathrm{sc}}(\phi)\geq0.
\end{equation}
Both inequalities are equalities.  Since the leading coefficients already
agree, equality of the two coefficient expressions gives
\begin{equation}\label{eq:affine-initials-subleading-equality}
 \widetilde b_1=b_1^{\mathrm{sc}}.
\end{equation}
We use \eqref{eq:affine-initials-subleading-equality} in \textbf{Step 4} to prove the
uniform bound in \eqref{eq:affine-initials-row-budget}.
Equality in \eqref{eq:affine-initials-scalar-DF} yields
\begin{equation}\label{eq:affine-initials-profile-support}
 \epsilon_0=\epsilon_1=0,
 \qquad
 \operatorname{supp}(\phi'')\subseteq\{\lambda\}.
\end{equation}
Since $\phi''$ is a nonnegative measure, it equals
$\kappa\delta_\lambda$ for some $\kappa\geq0$.  This proves
\eqref{eq:affine-initials-one-crease}.

\medskip
\noindent\textbf{Step 4.} In this step, we prove the uniform bound for
$\sum_j g_{d,j}$.  The exact expression for $n_{d,j}$ has the form
\begin{equation}\label{eq:affine-initials-full-block-polynomial}
 n_{d,j}=d^4p(j/d)+d^3q(j/d)+d^2r(j/d)+ds(j/d)+u.
\end{equation}
For the affine part of $\phi$, expansion into powers of $j$ and the exact
power-sum formulas give a polynomial in $d$.  Apply the kernel identity
\eqref{eq:affine-initials-error-kernel} with $P=p$ and
$\psi(x)=(x-\lambda)_+$.  Since $(P\psi)''$ is a finite signed measure and
$\lVert K_d\rVert_\infty\leq1/(8d)$, we obtain
\begin{equation}\label{eq:affine-initials-crease-p-sum}
 \sum_{j=0}^dp(j/d)(j/d-\lambda)_+
 =d\int_0^1p(x)(x-\lambda)_+\,dx
 +\frac{p(1)(1-\lambda)}2+O(d^{-1}).
\end{equation}
The function $q(x)(x-\lambda)_+$ is Lipschitz, and its corresponding sum
is the integral term times $d$ with an $O(1)$ remainder.  After including
the factors in \eqref{eq:affine-initials-full-block-polynomial}, all
remaining terms are $O(d^4)$, uniformly in the fractional part of
$d\lambda$.  Hence
\begin{equation}\label{eq:affine-initials-sharp-scalar-weight-expansion}
 d\sum_{j=0}^dn_{d,j}\phi(j/d)
 =b_0^{\mathrm{sc}}d^6+b_1^{\mathrm{sc}}d^5+O(d^4).
\end{equation}
Combining \eqref{eq:affine-initials-source-weight-expansion},
\eqref{eq:affine-initials-leading-coefficient},
\eqref{eq:affine-initials-subleading-equality}, and
\eqref{eq:affine-initials-sharp-scalar-weight-expansion} in
\eqref{eq:affine-initials-Dd} gives
\begin{equation}\label{eq:affine-initials-Dd-order}
 0\leq D_d=O(d^4).
\end{equation}
Since all four $h_i$ are strictly positive on $[0,1]$, put
$m_i=\min_{[0,1]}h_i>0$.  After increasing the lower bound on $d$, every
factor in \eqref{eq:affine-initials-block-RR} is at least $dm_i/2$,
uniformly for $0\leq j\leq d$.  Thus
\begin{equation}\label{eq:affine-initials-explicit-block-lower-bound}
 n_{d,j}\geq c_{\mathrm{blk}}d^4,
 \qquad
 c_{\mathrm{blk}}=2^{-4}\prod_{i=0}^3m_i>0.
\end{equation}
Every $g_{d,j}$ is nonnegative, and hence
\begin{equation}\label{eq:affine-initials-defect-sum}
 \sum_{j=0}^dg_{d,j}
 \leq\frac{D_d}{c_{\mathrm{blk}}d^4}
 \leq C_{\mathcal T}.
\end{equation}
This proves \eqref{eq:affine-initials-row-budget}.
\end{proof}

\subsection{Constancy on rational slices}
\label{subsec:rational-slice-constancy}

The summability estimate \eqref{eq:affine-initials-row-budget} and the
fixed-pair inequality \eqref{eq:affine-initials-fixed-pair} force the
asymptotic Chow invariant to vanish on every rational ray, after which a
strictness theorem removes all variation inside each rational slice.

\begin{thm}[Strictness for filtrations,
{\cite[Proposition~11]{Sze15}}]
\label{thm:szekelyhidi-filtration-strictness}
Let $(Y,H)$ be a polarized manifold admitting a cscK metric in $c_1(H)$,
and assume that its polarized automorphism group is finite and that the
complete section ring of $H$ is generated in degree one.  A multiplicative
linearly bounded filtration of this ring with positive
$L^2$-norm has strictly positive asymptotic Chow invariant.
\end{thm}

A multiplicative linearly bounded filtration is brought to the
normalization of \cite{Sze15}, in which the filtration is increasing and
begins at level zero, by an entrywise shift linear in the degree; the
shift changes neither the centered $L^2$-norm nor any finite Chow weight
\eqref{eq:chow-weight-normalization}.

\begin{lem}[Vanishing on every rational ray]
\label{lem:affine-initials-ray-vanishing}
Notation and conditions are as in Set-up~\ref{setup:initial-filtrations}.
Let $x=P/Q\in(0,1)\cap\mathbb Q$, with $P,Q$ chosen as in
\eqref{eq:affine-initials-ray-divisibility}.  Write
$\operatorname{Chow}_\infty(\chi_{P,Q})$ for the intrinsic asymptotic Chow
invariant.  Then
\begin{equation}\label{eq:affine-initials-ray-Chow-zero}
 \operatorname{Chow}_\infty(\chi_{P,Q})=0.
\end{equation}
\end{lem}

\begin{proof}
We prove the assertion in three steps.  First we identify the ray invariant
with a liminf of finite Chow defects.  We then choose finitely many fixed
symmetric pairs in one affine interval of $\phi$, and finally use the
summability bound to make the total defect of one pair arbitrarily small.

\medskip
\noindent\textbf{Step 1.} In this step, we identify the finite Chow
normalization.  Write
\[
 h_x(s)=\alpha_xs^4+O(s^3)
\]
for the Hilbert function in ray degree $s$.  Let $\chi_{P,Q}^{(k)}$ be the
finitely generated approximation generated by the weighted ray-degree-$k$
space $V_{Qk,Pk}$, let $w_x(k)$ be the total central algebraic weight of
$\chi_{P,Q}$ in ray degree $k$, and let $b_{0,x}^{(k)}$ be the leading central
weight coefficient of $\chi_{P,Q}^{(k)}$ in the original ray grading.  In the ray variables, the finite Chow normalization
\eqref{eq:chow-weight-normalization} reads
\begin{equation}\label{eq:affine-initials-finite-Chow}
 \operatorname{Chow}_k\bigl(\chi_{P,Q}^{(k)}\bigr)
 =\frac{k b_{0,x}^{(k)}}{\alpha_x}-\frac{w_x(k)}{h_x(k)}.
\end{equation}
The exact-ray ring is generated in ray degree one.  Dilation from ray degree
one to ray degree $k$, together with the sign change from increasing entries
to central algebraic weights, gives
\begin{equation}\label{eq:affine-initials-Chow-scaling}
 \alpha_{0,Qk,Pk}=k^4\alpha_x,
 \qquad
 b_{0,x}^{(k)}=-k^{-5}\beta^{\mathrm{gen}}_{0,Qk,Pk},
 \qquad
 \frac{w_x(k)}{h_x(k)}=-\overline i_{Qk,Pk}.
\end{equation}
It follows from \eqref{eq:affine-initials-finite-Chow} and
\eqref{eq:affine-initials-Chow-scaling} that
\[
 \operatorname{Chow}_k\bigl(\chi_{P,Q}^{(k)}\bigr)
 =\overline i_{Qk,Pk}
 -\frac{\beta^{\mathrm{gen}}_{0,Qk,Pk}}{\alpha_{0,Qk,Pk}}
 =c_{Qk,Pk}.
\]
By \eqref{eq:chow-weight-normalization}, the intrinsic asymptotic Chow
invariant is the full-sequence lower limit of these finite invariants.  Consequently,
\begin{equation}\label{eq:affine-initials-Chow-liminf}
 \operatorname{Chow}_\infty(\chi_{P,Q})
 =\liminf_{k\to\infty}c_{Qk,Pk}\geq0.
\end{equation}

\medskip
\noindent\textbf{Step 2.} In this step, we choose the rational symmetric
pairs.  Since $\lambda$ is irrational, the rational number $x$ lies in
one of the two open intervals on which $\phi$ is affine.  Denote that
interval by $I$.  Fix a positive integer $N$, choose pairwise distinct
positive rational numbers
\[
 0<t_1<\cdots<t_N<\operatorname{dist}(x,\partial I),
\]
and put
\begin{equation}\label{eq:rational-symmetric-pairs}
 y_\ell=x-t_\ell,\qquad z_\ell=x+t_\ell.
\end{equation}
For each $\ell$, choose a denominator $D_\ell$ divisible by $e,Q$ and
by the reduced denominators of $x,y_\ell,z_\ell$.
Lemma~\ref{lem:affine-initials-fixed-pair} applies to this fixed pair on the
progression $d=D_\ell k$.  Passing to the least common multiple of the
finitely many $D_\ell$ gives one progression on which all $N$
inequalities hold, each with an error tending to zero.

\medskip
\noindent\textbf{Step 3.} In this step, we use the summable defects.
The $2N$ input block indices are distinct.  By
\eqref{eq:affine-initials-row-budget}, for every sufficiently large $d$
on the common progression, one pair satisfies
\begin{equation}\label{eq:cheap-symmetric-pair}
 g_{d,dy_\ell}+g_{d,dz_\ell}\leq\frac{C_{\mathcal T}}N.
\end{equation}
Pass to a subsequence on which the chosen index $\ell$ is fixed.  Since
$\phi$ is affine on $I$,
\[
 2\phi(x)=\phi(y_\ell)+\phi(z_\ell).
\]
Subtracting $2d\phi(x)$ from
\eqref{eq:affine-initials-fixed-pair} gives
\begin{equation}\label{eq:midpoint-defect-bound}
 0\leq c_{2d,2dx}
 \leq g_{2d,2dx}
 \leq g_{d,dy_\ell}+g_{d,dz_\ell}+o(1)
 \leq\frac{C_{\mathcal T}}N+o(1).
\end{equation}
Since $Q\mid d$, the target block lies on the chosen exact ray.
Together, \eqref{eq:affine-initials-Chow-liminf} and
\eqref{eq:midpoint-defect-bound} give
\[
 0\leq\operatorname{Chow}_\infty(\chi_{P,Q})
 \leq\frac{C_{\mathcal T}}N.
\]
Letting $N\to\infty$ proves
\eqref{eq:affine-initials-ray-Chow-zero}.
\end{proof}

\begin{lem}[Constancy on rational slices]
\label{lem:affine-initials-ray-detection}
Notation and conditions are as in Set-up~\ref{setup:initial-filtrations}.
For every $x\in(0,1)\cap\mathbb Q$, the restriction of $G$ to
$\Omega_x$ is constant almost everywhere.
\end{lem}

\begin{proof}
Choose $P,Q$ as in
\eqref{eq:affine-initials-ray-divisibility}, and put
$N=M^Q\otimes L^P$.  The product of constant-curvature metrics on the
curve factors is cscK in $c_1(N)$.  Every curve has genus greater than
one, so the polarized automorphism group of $(B,N)$ is finite, and the
section ring of $N$ is generated in degree one by
Lemma~\ref{lem:affine-initials-veronese}.  By
Theorem~\ref{thm:szekelyhidi-filtration-strictness}, positive ray norm would
imply strictly positive asymptotic Chow invariant.
Lemma~\ref{lem:affine-initials-ray-vanishing} therefore gives
$\lVert\chi_{P,Q}\rVert_2=0$.  Formula \eqref{eq:affine-initials-ray-norm} shows
that the slice variance of $G$ vanishes.  Thus $G\vert_{\Omega_x}$ is
constant almost everywhere.
\end{proof}

\subsection{Propagation from rational slices}
\label{subsec:slice-propagation}

Lemma~\ref{lem:affine-initials-slice-propagation} extends constancy on
rational slices to every interior slice of $\Omega$.

\begin{lem}[Propagation from rational slices]
\label{lem:affine-initials-slice-propagation}
Let $H$ be a finite convex function on $\Omega$.  If
$H\vert_{\Omega_x}$ is constant almost everywhere for every
$x\in(0,1)\cap\mathbb Q$, then there exists a finite convex function
$\psi\colon(0,1)\to\mathbb R$ such that
\begin{equation}\label{eq:affine-initials-H-scalar}
 H(x,y)=\psi(x)
\end{equation}
at every interior point of $\Omega$.  The value $\psi(x)$ is the
normalized slice average of $H$.
\end{lem}

\begin{proof}
A finite convex function is continuous on the interior of its domain.  On a
rational interior slice, almost-everywhere constancy and continuity imply
constancy on the relative interior of that slice.  Fix $x\in(0,1)$ and
two points $y,y'$ in the relative interior of $\Omega_x$.  For a
rational sequence $x_k\to x$, the affine side lengths ensure that
\[
 y,y'\in\operatorname{relint}(\Omega_{x_k})
\]
for all sufficiently large $k$.  Hence
$H(x_k,y)=H(x_k,y')$, and interior continuity gives
$H(x,y)=H(x,y')$.

Let $\psi(x)$ be this common value.  For fixed $t\in(0,1)^4$, the map
\[
 x\mapsto
 \bigl(x,h_0(x)t_0,\ldots,h_3(x)t_3\bigr)
\]
is affine.  Restricting $H$ to this path shows that $\psi$ is convex.
The boundary of each rectangular slice has four-dimensional measure zero,
so $\psi(x)$ is also the normalized slice average.
\end{proof}

\subsection{Exclusion of the irrational crease}
\label{subsec:irrational-crease}

Algebraicity now excludes the remaining possible crease at
$\lambda=(3-\sqrt5)/2$ through the rational spectral structure of an
algebraic Duistermaat--Heckman measure.

\begin{thm}[Volume from normalized blow-ups,
{\cite[Lemma~5.1]{BHJ17}}]
\label{thm:normalized-blow-up-volume}
Let $Y$ be a normal projective variety of dimension $n$, let $H$ be ample,
and let $S\subseteq\bigoplus_{m\geq0}H^0(Y,mH)$ be a graded subalgebra
containing an ample series.  For every sufficiently large positive integer $m$, let
$\mathfrak a_m$ be the base ideal of $S_m$, let
$\mu_m\colon Y_m\to Y$ be its normalized blow-up, and write
$\mathfrak a_m\mathcal O_{Y_m}=\mathcal O_{Y_m}(-F_m)$.  Then
\begin{equation}\label{eq:normalized-blow-up-volume}
 \vol(S)
 =\lim_{m\to\infty}
 \left(\mu_m^*H-\frac1mF_m\right)^n.
\end{equation}
\end{thm}

\begin{proof}
This is \cite[Lemma~5.1]{BHJ17}.
\end{proof}

\begin{thm}[Piecewise polynomial density of Duistermaat--Heckman measures,
{\cite[Theorem~5.10]{BHJ17}}]
\label{thm:bhj-piecewise-polynomial-density}
Let $Y$ be a normal projective variety of dimension $n$, let $H$ be ample,
and let $F^\bullet R(Y,H)$ be a finitely generated integer filtration.  Then
$F^\bullet R(Y,H)$ has linear growth, and the density of the absolutely
continuous part of its limit measure is piecewise polynomial of degree at
most $n-1$.
\end{thm}

\begin{lem}[The evaluation-ideal system]
\label{lem:filtered-evaluation-ideal-system}
Let $Y$ be an integral projective variety, let $H$ be ample, and let
$F^\bullet R(Y,H)$ be a finitely generated integer filtration.  After an
integral character shift and passage to one fixed Veronese, put
$N=\mathbb Z^2$.  There exist a pointed rational polyhedral cone
$C\subset N_{\mathbb R}$ and a finitely generated saturated semigroup
\begin{equation}\label{eq:evaluation-ideal-saturated-semigroup}
 S=C\cap N
\end{equation}
with the following properties.  For $\tau=(m,\ell)\in S$, let
\begin{equation}\label{eq:evaluation-ideal-definition}
 \mathfrak a_\tau=\operatorname{im}\left(
 F^\ell H^0(Y,mH)\otimes\mathcal O_Y(-mH)\to\mathcal O_Y
 \right),
\end{equation}
where $\mathfrak a_\tau=0$ when the filtered piece is zero and
$\mathfrak a_0=\mathcal O_Y$.  Then
$\mathfrak a_\bullet$ is a finitely generated $S$-graded system of coherent
ideals:
\begin{equation}\label{eq:evaluation-ideal-multiplicativity}
 \mathfrak a_\tau\mathfrak a_{\tau'}
 \subseteq\mathfrak a_{\tau+\tau'}.
\end{equation}
\end{lem}

\begin{proof}
Choose finitely many bihomogeneous generators of the filtered Rees algebra,
including its parameter of section degree zero, and let
$T\subset\mathbb Z^2$ be the semigroup generated by their bidegrees.  We use
the convention in which a section in $F^\ell H^0(Y,mH)$ has bidegree
$(m,\ell)$ and the Rees parameter has bidegree $(0,-1)$.  An integral
character shift makes the second coordinate of every remaining generator
nonnegative.  Those generators have positive first coordinate, so their
degrees together with $(0,-1)$ generate a pointed cone.  A vector in
the intersection of this cone with its negative has first coordinate zero,
so it lies on the nonnegative parameter ray; its negative lies on that ray
only when the vector is zero.  Since $Y$ is
integral, a monomial in nonzero Rees generators is nonzero.  Thus $T$ is
exactly the support of the filtered Rees algebra.  We may choose the
Veronese so that a nonzero degree-one piece exists.  Its bidegree
$(1,\ell_0)$ together with $(0,-1)$ generates $N=\mathbb Z^2$.  Let $C$ be
the rational polyhedral cone generated by $T$.  The saturation
of $T$ in $N$ is $S=C\cap N$.  To see that it is finitely generated, choose
integral generators $v_1,\ldots,v_q$ of $C$.  If
$u=\sum_i a_iv_i$ belongs to $C\cap N$, write
\[
 u=\sum_i\lfloor a_i\rfloor v_i+h,
 \qquad
 h\in N\cap
 \left\{\sum_i b_iv_i\mathrel{\big|}0\leq b_i<1\right\}.
\]
The set on the right is bounded and contains only finitely many lattice
points.  Those points together with the $v_i$ generate $S$.

The convention in \eqref{eq:evaluation-ideal-definition} extends the system
from its supported degrees $T$ to $S$ by zero ideals.  Multiplicativity of
the filtration and of the evaluation maps gives
\eqref{eq:evaluation-ideal-multiplicativity}.  Every filtered section is a
sum of monomials in the chosen Rees generators.  Its evaluation ideal is
therefore a sum of products of the finitely many evaluation ideals of those
generators.  Consequently,
\begin{equation}\label{eq:evaluation-ideal-rees-algebra}
 \bigoplus_{\tau\in S}\mathfrak a_\tau
\end{equation}
is a finitely generated $S$-graded $\mathcal O_Y$-algebra.  The zero slots
add no generators.  This verifies the lattice, saturation, zero-piece,
multiplicativity, and finite-generation hypotheses used below.
\end{proof}

\begin{thm}[Rational chamber decomposition for base ideals,
{\cite[Proposition~4.7]{ELM+06}, with the corrected proof in
\cite[Proposition~1.1]{ELM+23}}]
\label{thm:elmnp-base-ideal-chambers}
Let $F^\bullet R(Y,H)$ be as in
Theorem~\ref{thm:bhj-piecewise-polynomial-density}, and construct its
saturated evaluation-ideal system as in
Lemma~\ref{lem:filtered-evaluation-ideal-system}.  Then there
exist a positive integer $d$ and a finite smooth rational fan with support
$C$ and primitive integral ray generators $e_i=(m_i,\ell_i)$.  Thus every two adjacent ray
generators span a unimodular cone.  Their finite slopes are strictly
increasing; a boundary ray of the full cone may be vertical.  The
integer $d$ may be enlarged so that $de_i$ belongs to the supported
semigroup for every ray: since $e_i$ belongs to the saturation of the support
semigroup, some positive multiple of $e_i$ belongs to that support, and the
conclusion below remains valid after replacing $d$ by a positive multiple.
Every integral vector
$\tau=(m,\ell)$ in a chamber spanned by two adjacent generators has a unique
expression
\begin{equation}\label{eq:elmnp-chamber-decomposition}
 \tau=p_i e_i+p_{i+1}e_{i+1},
 \qquad
 p_i,p_{i+1}\in\mathbb Z_{\geq0},
\end{equation}
and
\begin{equation}\label{eq:elmnp-chamber-integral-closure}
 \overline{\mathfrak a_{d\tau}}
 =\overline{
   \mathfrak a_{de_i}^{\,p_i}
   \mathfrak a_{de_{i+1}}^{\,p_{i+1}}}.
\end{equation}
\end{thm}

\begin{thm}[Rational spectral structure]
\label{thm:rational-dh-spectrum}
Let $Y$ be a normal projective variety, let $H$ be ample, and let a
finitely generated integer filtration of the full section ring $R(Y,H)$ be
given.  The density of the absolutely continuous part of its limit weight
measure is polynomial with rational coefficients on each interval of a
finite partition whose breakpoints are rational.
\end{thm}

\begin{proof}
We prove the rationality assertion in three steps.  First, we place all base
ideals on one birational model over a finite rational chamber decomposition.
We then obtain a rational polynomial formula for the filtered volume on each
chamber.  Finally, we differentiate the tail distribution and identify its
breakpoints.

\medskip
\noindent\textbf{Step 1.} In this step, we construct the chamber
decomposition and a common model for its ray ideals.
Put $n=\dim Y$, and let $\mathfrak a_{m,\ell}$ be the base ideal of the
filtered piece of $H^0(Y,mH)$ at weight $\ell$.  Piecewise polynomiality
follows from Theorem~\ref{thm:bhj-piecewise-polynomial-density}.  We retain
the multigraded base-ideal data in order to prove the rationality assertion.
Lemma~\ref{lem:filtered-evaluation-ideal-system} supplies the lattice $N$,
the pointed rational cone $C$, the saturated semigroup $S=C\cap N$, and the
finitely generated evaluation-ideal system, with zero ideals at unsupported
indices.  The preliminary changes in that lemma preserve the assertion that
we are proving.  An integral character shift translates normalized weights
by an integer.  If the fixed Veronese has index $r$, its degree-$m$ spectrum
is the degree-$rm$ spectrum of the original filtration, divided by $m$
rather than by $rm$.  Its limit measure is therefore the pushforward of the
original one under $x\mapsto rx$; this is also
\cite[Remark~5.6]{BHJ17}.  Integer translation and positive integral
dilation preserve rational breakpoints and rational polynomial
coefficients, in both directions.  We may consequently prove the assertion
after these changes, and retain the notation $H$ and $F^\bullet$ for the
resulting polarization and filtration.

By Theorem~\ref{thm:elmnp-base-ideal-chambers}, there exist a positive integer
$d$ and finitely many integral chamber generators
\[
 e_i=(m_i,\ell_i)
\]
whose finite slopes are strictly increasing, with the following property.
Every integral vector $\tau=(m,\ell)$ in the chamber spanned by adjacent
vectors $e_i,e_{i+1}$ has a unique expression
\[
 \tau=p_i e_i+p_{i+1}e_{i+1},
 \qquad
 p_i,p_{i+1}\in\mathbb Z_{\geq0},
\]
and
\begin{equation}\label{eq:rational-spectrum-integral-closure}
 \overline{\mathfrak a_{d\tau}}
 =\overline{
   \mathfrak a_{de_i}^{\,p_i}
   \mathfrak a_{de_{i+1}}^{\,p_{i+1}}}.
\end{equation}
Choose one normal projective birational model $\mu\colon Y'\to Y$ that
dominates the normalized blow-ups of the finitely many ray ideals.  Write
\begin{equation}\label{eq:rational-spectrum-ray-divisors}
 \mathfrak a_{de_i}\mathcal O_{Y'}=\mathcal O_{Y'}(-E_i),
\end{equation}
where every $E_i$ is an integral Cartier divisor.  On $Y'$, the product of
the two ray ideals is invertible.  The universal property therefore factors
$\mu$ first through the ordinary blow-up.  Since $Y'$ is normal, this
factorization lifts through its normalization.  Integral closure does not change a
normalized blow-up, so
\eqref{eq:rational-spectrum-integral-closure} identifies that blow-up with
the normalized blow-up of $\mathfrak a_{d\tau}$.  Pulling back its exceptional
Cartier divisor gives the explicit chain
\begin{equation}\label{eq:rational-spectrum-pulled-back-ideal}
\begin{aligned}
 \mathfrak a_{d\tau}\mathcal O_{Y'}
 &=\overline{\mathfrak a_{d\tau}}\mathcal O_{Y'}\\
 &=\overline{
   \mathfrak a_{de_i}^{\,p_i}
   \mathfrak a_{de_{i+1}}^{\,p_{i+1}}}\mathcal O_{Y'}\\
 &=\mathfrak a_{de_i}^{\,p_i}
   \mathfrak a_{de_{i+1}}^{\,p_{i+1}}\mathcal O_{Y'}
 =\mathcal O_{Y'}(-p_iE_i-p_{i+1}E_{i+1}).
\end{aligned}
\end{equation}

\medskip
\noindent\textbf{Step 2.} We compute the filtered volume on one chamber
and prove that it is a polynomial with rational coefficients in this step.
Let $\lambda_{\max}$ be the maximal asymptotic weight of the modified
filtration, and use the convention
\[
 R(u)=\bigoplus_{q\geq0}
 F^{\lceil qu\rceil}H^0(Y,qH).
\]
For every $s<\lambda_{\max}$, choose
$s'$ with $s<s'<\lambda_{\max}$.  By
\cite[Theorem~5.3(i)]{BHJ17}, the threshold series $R(s')$ contains an
ample series.  For all sufficiently large $m$,
\[
 d\lceil ms\rceil\leq \lceil dm s'\rceil.
\]
The decreasing property of the filtration shows that, in every sufficiently
large degree $m$, the degree-$m$ piece of the graded series
\[
 \bigoplus_{m\geq0}F^{d\lceil ms\rceil}H^0(Y,dmH)
\]
contains the degree-$m$ piece of the $d$-th Veronese of $R(s')$.  Since
containing an ample series is an asymptotic condition, the rounded series
itself contains an ample series.  Above
$\lambda_{\max}$ its volume is zero; outside the compact support supplied
by \cite[Theorem~5.3(iii)]{BHJ17}, the tail is constant and its density
vanishes.  It therefore suffices to work on a finite-slope chamber with
$s<\lambda_{\max}$.  Put $t=\ell/m$.  Dividing the two
coordinates of the ray decomposition by $m$ shows that
\[
 \frac{p_i}{dm}
 \quad\text{and}\quad
 \frac{p_{i+1}}{dm}
\]
are affine functions of $t$ with rational coefficients: they are obtained by
inverting the integral two-by-two matrix with columns $e_i,e_{i+1}$.
For each integral vector $\tau=(m,\ell)$ in this chamber, the Cartier divisor
\[
 dm\mu^*H-p_iE_i-p_{i+1}E_{i+1}
\]
is the pullback of the moving divisor on the normalized blow-up of
$\mathfrak a_{d\tau}$.  It is globally generated and hence nef.  Thus the
intersection number in \eqref{eq:rational-spectrum-finite-volume} is the
moving self-intersection supplied by the
base-ideal construction, not the volume of an arbitrary non-nef divisor.
For this integral pair, set
\begin{equation}\label{eq:rational-spectrum-finite-volume}
 V_{d\tau}:=
 \left(
   \mu^*H-\frac{p_i}{dm}E_i-\frac{p_{i+1}}{dm}E_{i+1}
 \right)^n.
\end{equation}
For a real number $s$ in the interior of this chamber, put
$\tau_m(s)=(m,\lceil ms\rceil)$ and define
\begin{equation}\label{eq:rational-spectrum-chamber-polynomial}
 P_i(s)=
 \left(\mu^*H-c_i(s)E_i-c_{i+1}(s)E_{i+1}\right)^n,
\end{equation}
where $c_i$ and $c_{i+1}$ are the rational-affine functions obtained from
the ray decomposition.  The spaces
\begin{equation}\label{eq:rational-spectrum-rounded-ray-series}
 S_m^{(s)}=F^{d\lceil ms\rceil}H^0(Y,dmH)
\end{equation}
form a graded linear series: multiplicativity of the filtration and
$\lceil ms\rceil+\lceil m's\rceil\geq\lceil(m+m')s\rceil$ give the
required product inclusion.  Its degree-$m$ base ideal is
$\mathfrak a_{d\tau_m(s)}$.  Therefore
Theorem~\ref{thm:normalized-blow-up-volume}, applied to the $d$-Veronese and
divided by $d^n$, together with
\eqref{eq:rational-spectrum-pulled-back-ideal}, gives
\begin{equation}\label{eq:rational-spectrum-rounded-ray-limit}
 \lim_{m\to\infty}
 \frac{n!}{(dm)^n}\dim S_m^{(s)}=P_i(s).
\end{equation}

We now compare this rounded-ray series with the actual threshold at $s=t$.
Choose $\epsilon>0$ so that $t-\epsilon$ and $t+\epsilon$ remain in the
interior of the chamber.  For all sufficiently large $m$,
\[
 d\lceil m(t-\epsilon)\rceil
 \leq\lceil dmt\rceil
 \leq d\lceil m(t+\epsilon)\rceil.
\]
Since the filtration is decreasing, the corresponding inclusions are
\[
 F^{d\lceil m(t-\epsilon)\rceil}H^0(Y,dmH)
 \supseteq F^{\lceil dmt\rceil}H^0(Y,dmH)
 \supseteq F^{d\lceil m(t+\epsilon)\rceil}H^0(Y,dmH).
\]
Hence
\begin{equation}\label{eq:rational-spectrum-threshold-squeeze}
 \begin{aligned}
 P_i(t+\epsilon)
 &\leq\liminf_{m\to\infty}
 \frac{n!}{(dm)^n}\dim F^{\lceil dmt\rceil}H^0(Y,dmH)\\
 &\leq\limsup_{m\to\infty}
 \frac{n!}{(dm)^n}\dim F^{\lceil dmt\rceil}H^0(Y,dmH)\\
 &\leq P_i(t-\epsilon).
 \end{aligned}
\end{equation}
The function $P_i$ is a polynomial and is therefore continuous.  Letting
$\epsilon\downarrow0$ in
\eqref{eq:rational-spectrum-threshold-squeeze} proves both the existence of
the actual filtered-volume limit and its equality with $P_i(t)$.  Thus
\begin{equation}\label{eq:rational-spectrum-chamber-volume}
 V(t)=P_i(t)
 =\left(\mu^*H-c_i(t)E_i-c_{i+1}(t)E_{i+1}\right)^n,
\end{equation}
which is a polynomial in $t$ with rational coefficients, because the
coefficients $c_i(t)$ and $c_{i+1}(t)$ are rational-affine and the
intersections of integral Cartier divisors are integers.

\medskip
\noindent\textbf{Step 3.} We pass from the chamber volume to the density
and determine the breakpoints in this step.
Let $\nu$ be the limit measure in the present grading.  Restrict the
filtration to the $d$-th Veronese $R(Y,dH)$.  Its limit measure is
$(x\mapsto dx)_*\nu$ by \cite[Remark~5.6]{BHJ17}.  Applying the tail formula
\cite[Theorem~5.3(iii), equation~(5.4)]{BHJ17} at the threshold $dt$, and
using the limit proved in
\eqref{eq:rational-spectrum-threshold-squeeze}, gives
\[
 \vol\!\left(
   \bigoplus_{m\geq0}
   F^{\lceil dmt\rceil}H^0(Y,dmH)
 \right)=d^nV(t).
\]
Integration of equation~(5.4) from $dt$ to $+\infty$ then gives, away from
an endpoint atom,
\begin{equation}\label{eq:rational-spectrum-tail-identity}
 \nu([t,+\infty))
 =\frac{1}{(dH)^n}
   \vol\!\left(
     \bigoplus_{m\geq0}
     F^{\lceil dmt\rceil}H^0(Y,dmH)
   \right)
 =\frac{V(t)}{H^n}.
\end{equation}
The first equality in \eqref{eq:rational-spectrum-tail-identity} is the
probability normalization in \cite[equation~(5.4)]{BHJ17}; the second uses
$(dH)^n=d^nH^n$ and the preceding volume identity.
Thus passage to degrees divisible by $d$ computes the pushforward measure,
not a different subsequential limit.  The absolutely continuous density is
therefore $-V'(t)/H^n$ on the interior of each chamber and has rational
polynomial coefficients there.  Each finite chamber endpoint is a slope
$\ell_i/m_i$ and is rational.  This proves
Theorem~\ref{thm:rational-dh-spectrum}.  The argument does not exclude a
finite atomic part at the chamber boundaries.
\end{proof}

\begin{rem}[Use for a fiber-scaling initial]
\label{rem:rational-spectrum-initial}
In the application in Lemma~\ref{lem:affine-initials-crease-exclusion},
finite generation is asserted only for the algebraic source test
configuration after a sufficiently divisible
Veronese.  Either fiber-scaling initial filtration $\chi^\epsilon$ of
Set-up~\ref{setup:initial-filtrations} has weight multiplicities equal to
those of the source in every degree, so its limit measure equals the source
limit measure.  It inherits
Theorem~\ref{thm:rational-dh-spectrum} without any assertion that the initial
filtration is finitely generated.  Returning from the Veronese coordinate to
actual degree divides the normalized weight coordinate by that positive
integer and preserves rational breakpoints and rational coefficients.
\end{rem}

\begin{lem}[Exclusion of the irrational crease]
\label{lem:affine-initials-crease-exclusion}
Notation and conditions are as in Set-up~\ref{setup:initial-filtrations}.
Assume that $G$ is the convex transform in actual degree of a
fiber-scaling initial filtration $\chi^\epsilon$ of an algebraic test
configuration.  Let $\phi$
be its normalized slice average, and assume
\begin{equation}\label{eq:affine-initials-crease-hypotheses}
 \operatorname{supp}(\phi'')\subseteq\{\lambda\},
 \qquad
 G_*\mu_\Omega=\phi_*\rho,
 \qquad
 \lambda=\frac{3-\sqrt5}{2}.
\end{equation}
Then there exist $a,b\in\mathbb R$ such that
\begin{equation}\label{eq:affine-initials-crease-conclusion}
 G(x,y)=a+bx
\end{equation}
at every interior point of $\Omega$.
\end{lem}

\begin{proof}
We prove Lemma~\ref{lem:affine-initials-crease-exclusion} in three steps.  The
first removes transverse variation,
the second records the rational spectral constraint inherited from the
algebraic source, and the third treats all possible signs of the two slopes.

\medskip
\noindent\textbf{Step 1.} In this step, we prove that $G$ agrees with
$\phi$.  By Fubini,
\[
 \phi(x)=\int_{\Omega_x}G(x,y)\,d\nu_x(y)
\]
for $\rho$-almost every $x$.  Equality of the pushforward measures in
\eqref{eq:affine-initials-crease-hypotheses} gives equality of their second
moments.  Therefore
\begin{equation}\label{eq:conditional-variance-identity}
\begin{aligned}
 \int_\Omega(G-\phi)^2\,d\mu_\Omega
 &=\int_\Omega G^2\,d\mu_\Omega
 -2\int_\Omega\phi G\,d\mu_\Omega
 +\int_\Omega\phi^2\,d\mu_\Omega\\
 &=0.
\end{aligned}
\end{equation}
Thus $G(x,y)=\phi(x)$ almost everywhere.  Convexity and interior
continuity upgrade this equality to every interior point.  Since
$\phi''$ is a nonnegative measure supported at one point,
\begin{equation}\label{eq:affine-initials-spectral-crease}
 \phi(x)=a+bx+\kappa(x-\lambda)_+
\end{equation}
for some $a,b\in\mathbb R$ and $\kappa\geq0$.

\medskip
\noindent\textbf{Step 2.} In this step, we record the spectral restriction
coming from the algebraic source.  After passing to a sufficiently
divisible Veronese, the filtration of the source test configuration is a
finitely generated integer filtration.
Theorem~\ref{thm:rational-dh-spectrum} shows that its absolutely continuous
Duistermaat--Heckman density is polynomial with rational coefficients on a
finite partition with rational breakpoints.  By
Lemma~\ref{lem:affine-initials-veronese}, the initial filtration has an
entry multiset identical to that of the source in every degree, so the two
filtrations have equal measures.
Passing between exponent-one degree and actual degree only rescales the
weight coordinate by the positive integer $e$, and therefore preserves
the rational spectral structure.

Suppose for a contradiction that $\kappa>0$, and put
\[
 b_R=b+\kappa.
\]
The two slopes of $\phi$ are $b<b_R$.  We shall use the expansion
\begin{equation}\label{eq:affine-initials-p-expanded}
\begin{aligned}
 p(x)={}&1791155412951840+48540311690994864x
 -96622883665346480x^2\\
 &+59625573524485696x^3-11543001550134080x^4.
\end{aligned}
\end{equation}
Its two highest coefficients are nonzero, and $p>0$ on $[0,1]$.

\medskip
\noindent\textbf{Step 3.} In this step, we exclude the crease by considering
the possible signs of $b$ and $b_R$.

We first record the breakpoints used in the argument.  A point at which the
two one-sided polynomial density germs differ, or at which one branch appears
with a nonzero one-sided jump, is an endpoint of the rational partition in
Theorem~\ref{thm:rational-dh-spectrum} and is therefore rational.  Since $p$
is positive on $[0,1]$, every outer support endpoint is rational.  A branch
endpoint not shared by a second branch is also a rational breakpoint: its
jump has size $p(0)/(Iu)$ or $p(1)/(Iv)$, where
$I=\int_0^1p(x)\,dx$ and $u,v>0$ are the absolute branch slopes.
If an explicit density formula agrees on a nonempty open interval with
a polynomial from the rational partition, polynomial identity forces all of
its coefficients to be rational.  Redundant partition points therefore do
not affect the argument.

\smallskip
\noindent\textbf{Case 3.1.} Assume that the two slopes share a weak sign.
Then $\phi$ is monotone.  On every nonconstant branch, its
pushforward density is obtained from $p/I$ by an affine change of variable
and division by the absolute value of the branch slope.  If both slopes are
nonzero, the leading coefficients of the two density germs at the image of
$\lambda$ are
\[
 \frac{p_4}{I\lvert b\rvert^5}
 \quad\text{and}\quad
 \frac{p_4}{I\lvert b_R\rvert^5},
\]
where $p_4\neq0$ is the coefficient of $x^4$ in $p$.  Since
$\lvert b\rvert\neq\lvert b_R\rvert$ in the monotone cases, this image is a
genuine rational breakpoint.  In a rational coordinate translated from an
outer support endpoint of the left branch, its density is
$p(z/\lvert b\rvert)/(I\lvert b\rvert)$.  The ratio of its coefficients of
$z^3$ and $z^4$ is
\[
 \frac{p_3}{p_4}\lvert b\rvert,
\]
where $p_3p_4\neq0$.  Rationality of these coefficients forces
$b\in\mathbb Q$.  The distance from the image of $\lambda$ to this endpoint
is $\lvert b\rvert\lambda$, which forces $\lambda\in\mathbb Q$.

If $b=0<b_R$, the left branch gives an atom.  On the right branch, in the
coordinate translated from its outer support endpoint, the constant
coefficient is $p(1)/(Ib_R)$.  If $b<0=b_R$, the corresponding coefficient
on the left branch is $p(0)/(I(-b))$.  In either case the nonzero slope is
rational.  The distance between the two rational support endpoints then
forces the relevant distance to $\lambda$ to be rational.
Every monotone case contradicts the irrationality of $\lambda$.

\smallskip
\noindent\textbf{Case 3.2.} Assume $b<0<b_R$.  Put
\begin{equation}\label{eq:crease-heights}
 u=-b,\qquad v=b_R,\qquad
 H_L=u\lambda,\qquad H_R=v(1-\lambda).
\end{equation}
Suppose first that $H_L\neq H_R$.  At the end of the shorter branch, that
branch disappears with the nonzero jump $p(0)/(Iu)$ or $p(1)/(Iv)$.  The
minimum and the longer outer endpoint are support endpoints, while the end of
the shorter branch is a rational breakpoint.  Hence all three points, and
therefore both heights, are rational.  If $H_L>H_R$, the unmatched left tail has
density $p(z/u)/(Iu)$ in a rational coordinate $z$ translated from the
outer endpoint of the left branch.  Its two highest coefficients force
$u\in\mathbb Q$, and then $H_L=u\lambda$ forces
$\lambda\in\mathbb Q$.  If $H_R>H_L$, the unmatched right tail has
density $p(1-z/v)/(Iv)$.  Its constant coefficient $p(1)/(Iv)$ forces
$v\in\mathbb Q$, after which $H_R=v(1-\lambda)$ also gives a
contradiction.

It remains to consider $H_L=H_R=:H$.  In this case the minimum and the common
outer endpoint are rational support endpoints.  Their difference
$H$ is therefore a nonzero rational number.  The constant coefficient of the two-branch density at
the minimum is
\begin{equation}\label{eq:equal-height-density}
 \frac{p(\lambda)}I\left(\frac1u+\frac1v\right)
 =\frac{p(\lambda)}{IH}.
\end{equation}
Reduction of \eqref{eq:affine-initials-p-expanded} modulo
$\lambda^2-3\lambda+1$ gives
\begin{equation}\label{eq:p-at-lambda}
 p(\lambda)
 =11881330905913872-6726783661974688\lambda.
\end{equation}
This number is irrational, contradicting the rationality of the density
coefficient in \eqref{eq:equal-height-density}.

Both cases are impossible.  Therefore $\kappa=0$, and
\eqref{eq:affine-initials-crease-conclusion} follows.
\end{proof}

\subsection{The affine-transform theorem}
\label{subsec:affine-transform-theorem}

We now combine rational-slice constancy, convex propagation, and the
spectral exclusion to obtain the form used in the equality classification.

\begin{prop}[Affine opposite initials]
\label{prop:affine-initials}
Notation and conditions are as in Set-up~\ref{setup:initial-filtrations}.
For $\epsilon\in\{+,-\}$, let $G^\epsilon$ be the convex
transform of $\chi^\epsilon$ in actual $A$-degree normalization.
There exist $a_\epsilon,b_\epsilon\in\mathbb R$ such that
\begin{equation}\label{eq:affine-initials-final-transform}
 G^\epsilon(x,y)=a_\epsilon+b_\epsilon x
\end{equation}
at every interior point $(x,y)$ of $\Omega$.

The entry probability laws in actual degree of $\mathcal T$ and
$\chi^\epsilon$ are both
\begin{equation}\label{eq:affine-initials-final-actual-law}
 (x\mapsto a_\epsilon+b_\epsilon x)_*\rho.
\end{equation}
In the exponent-one grading of $(X,A^e)$, their common law is
\begin{equation}\label{eq:affine-initials-final-regraded-law}
 (x\mapsto e(a_\epsilon+b_\epsilon x))_*\rho.
\end{equation}
No equality between the two affine functions is asserted.
\end{prop}

\begin{proof}
Fix $\epsilon$, write $\chi=\chi^\epsilon$ and
$G=G^\epsilon$, and let $\phi$ be the function defined by
\eqref{eq:affine-initials-slice-average} and
\eqref{eq:affine-initials-endpoint-averages}.  We first remove the transverse
variables.
Lemma~\ref{lem:affine-initials-row-budget} gives
\[
 \phi(x)=u_0+u_1x+\kappa(x-\lambda)_+.
\]
Lemmas~\ref{lem:affine-initials-ray-vanishing} and~\ref{lem:affine-initials-ray-detection}
show that $G$ is constant almost
everywhere on every rational interior slice.
Lemma~\ref{lem:affine-initials-slice-propagation} then gives
\begin{equation}\label{eq:affine-initials-profile-equality}
 G(x,y)=\phi(x)
\end{equation}
at every interior point of $\Omega$.

We next exclude the crease.
Lemma~\ref{lem:affine-initials-veronese} identifies the laws in actual degree
of the source and the initial with $G_*\mu_\Omega$.  The identity
\eqref{eq:affine-initials-profile-equality} and the fact that the boundary
has Lebesgue measure zero give
\begin{equation}\label{eq:affine-initials-profile-law}
 G_*\mu_\Omega=\phi_*\rho.
\end{equation}
Lemma~\ref{lem:affine-initials-crease-exclusion} gives $\kappa=0$.
Setting
\[
 a_\epsilon=u_0,\qquad b_\epsilon=u_1
\]
proves \eqref{eq:affine-initials-final-transform}.

It remains to identify the two degree normalizations.  Since the
$x$-marginal of $\mu_\Omega$ is $\rho$,
\eqref{eq:affine-initials-final-actual-law} follows from
\eqref{eq:affine-initials-final-transform}.  Finally,
\eqref{eq:affine-initials-transform-regrading} multiplies every normalized
entry by $e$ in exponent-one degree.  This proves
\eqref{eq:affine-initials-final-regraded-law}.  Since
$\epsilon\in\{+,-\}$ was arbitrary, the conclusion holds for both initial
filtrations.
\end{proof}
 
\section{Semistability and the integral affine comparator}\label{sec:initials}

This section proves semistability at every exponent and converts the affine
initial data of Section~\ref{sec:affine-initials} into an integral product
comparator after a ramified base change and normalization.

\subsection{Semistability at every exponent}

We first combine the approximate cscK metrics of
Proposition~\ref{prop:approximate-csck} with two analytic results on the
centered $L^2$ norm of a test configuration.

\begin{thm}[Donaldson's scalar-curvature lower bound,
{\cite[Theorem~2]{Don05b}}]
\label{thm:donaldson-lower-bound}
Let $(Y,H)$ be a polarized smooth projective variety, and let
$\mathcal Y$ be an exponent-one normal ample test configuration with
positive Donaldson centered norm $N_{2,\mathrm{Don}}(\mathcal Y)$.  Every
K\"ahler metric $\omega\in c_1(H)$ satisfies
\begin{equation}\label{eq:donaldson-calabi-bound}
 \left\lVert\Scal(\omega)-\widehat S_H\right\rVert_{L^2}
 \geq -\frac{F_{\mathrm{Don}}(\mathcal Y)}
 {N_{2,\mathrm{Don}}(\mathcal Y)},
\end{equation}
where $\widehat S_H$ is the average scalar curvature in $c_1(H)$.
The inequality \eqref{eq:donaldson-calabi-bound} follows from the cited theorem.  We
now translate the two algebraic quantities in that inequality into the
conventions of this paper.  If
\begin{equation}\label{eq:donaldson-hilbert-expansion}
 h(m)=A_0m^n+A_1m^{n-1}+O(m^{n-2})
\end{equation}
and the increasing-entry total weight is
\begin{equation}\label{eq:donaldson-entry-weight-expansion}
 W_{\mathrm{entry}}(m)=B_0m^{n+1}+B_1m^n+O(m^{n-1}),
\end{equation}
then
\begin{equation}\label{eq:donaldson-bhj-normalization}
 F_{\mathrm{Don}}
 =B_1-\frac{A_1B_0}{A_0}
 =\frac{A_0}{2}\DF,
 \qquad
 N_{2,\mathrm{Don}}^2
 =A_0\Var\bigl(\operatorname{DH}(\mathcal Y)\bigr)
 =A_0\lVert\mathcal Y\rVert_{2,\mathrm{BHJ}}^2.
\end{equation}
Thus \eqref{eq:donaldson-calabi-bound} has the displayed sign, and the
Donaldson and Boucksom--Hisamoto--Jonsson centered norms have identical zero
loci.
\end{thm}

\begin{thm}[Vanishing of the centered norm,
{\cite[Corollary~B and Lemma~2.10]{BHJ17}}]
\label{thm:zero-norm}
Let $(Y,H)$ be a polarized variety.  An ample test configuration for $(Y,H)$
has zero centered $L^2$ norm if and only if it is almost trivial.
Equivalently, after twisting the linearization by a scalar character, its
normalization is the trivial test configuration.  The vertical term
$c\mathcal Y_0$ in \cite[Lemma~2.10]{BHJ17} is precisely this scalar-character
twist after the exponent has been incorporated into $H$.
Thus a normal ample test configuration of zero norm is the trivial product
up to such a scalar twist and has zero Donaldson--Futaki invariant.
\end{thm}

The conclusion of Theorem~\ref{thm:zero-norm} is narrower than the
automorphism-induced polarized products allowed in
Definition~\ref{defn:ordinary-k-polystability}.

\begin{prop}\label{prop:semistability-all-exponents}
Let $e$ be a positive integer.  Every normal ample algebraic test
configuration whose generic polarized fiber is $(X,A^e)$ has nonnegative
Donaldson--Futaki invariant.
\end{prop}

\begin{proof}
We first rescale the approximate metrics from $c_1(A)$ to $c_1(A^e)$.
We then use Theorems~\ref{thm:donaldson-lower-bound} and~\ref{thm:zero-norm}
according to whether the centered norm is positive or zero.  We begin by
proving that the infimum of the centered $L^2$ scalar-curvature norm in
$c_1(A^e)$ is zero.
Proposition~\ref{prop:approximate-csck} gives K\"ahler metrics
$\omega_q\in c_1(A)$ such that
\[
 \left\lVert\Scal(\omega_q)-\overline S\right\rVert_{L^2}\to0.
\]
The metrics $e\omega_q$ belong to $c_1(A^e)$.  Since $X$ has complex
dimension five, we have
\[
 \Scal(e\omega_q)=e^{-1}\Scal(\omega_q),
 \qquad
 (e\omega_q)^5=e^5\omega_q^5,
 \qquad
 \widehat S_{A^e}=e^{-1}\overline S.
\]
Consequently,
\begin{equation}\label{eq:scaled-calabi-error}
 \left\lVert\Scal(e\omega_q)-\widehat S_{A^e}\right\rVert_{L^2}^2
 =e^3\left\lVert\Scal(\omega_q)-\overline S\right\rVert_{L^2}^2
 \to0.
\end{equation}

We now exclude a negative Donaldson--Futaki invariant for an arbitrary
normal ample test configuration of $(X,A^e)$.
Let $\mathcal T$ be such a test configuration, regarded as an exponent-one
test configuration of the polarized pair $(X,A^e)$.  If
$\lVert\mathcal T\rVert_{2,\mathrm{BHJ}}>0$, then
\eqref{eq:donaldson-bhj-normalization},
\eqref{eq:donaldson-calabi-bound}, and
\eqref{eq:scaled-calabi-error} exclude $\DF(\mathcal T)<0$.  If
$\lVert\mathcal T\rVert_{2,\mathrm{BHJ}}=0$, then
Theorem~\ref{thm:zero-norm} gives
$\DF(\mathcal T)=0$.  Thus $\DF(\mathcal T)\geq0$ in both cases.
\end{proof}

\subsection{Common rational affine data}

We next identify the two affine profiles from
Proposition~\ref{prop:affine-initials} and use algebraicity of the source
test configuration to prove that their common coefficients are rational.

\begin{lem}[Moment rigidity]\label{lem:affine-pushforward-rigidity}
Notation and conditions are as in Set-up~\ref{setup:initial-filtrations}.  If
$Z$ is distributed according to $\rho$, then
\begin{equation}\label{eq:rho-moment-data}
 \mathbb E[Z]=\frac{8807}{19450},
 \qquad
 \Var(Z)=\frac{153270007}{2648117500},
 \qquad
 \mathbb E\!\left[(Z-\mathbb E[Z])^3\right]
 =\frac{41764918113}{12876471343750}.
\end{equation}
The last two numbers in \eqref{eq:rho-moment-data} are positive.  If
$a,b,a',b'\in\mathbb R$ satisfy
\begin{equation}\label{eq:equal-affine-pushforwards}
 (a+bx)_*\rho=(a'+b'x)_*\rho,
\end{equation}
then $a=a'$ and $b=b'$.
\end{lem}

\begin{proof}
The polynomial $p$ in \eqref{eq:rho-definition} expands as
\begin{align*}
 p(x)={}&1791155412951840+48540311690994864x
 -96622883665346480x^2\\
 &+59625573524485696x^3-11543001550134080x^4.
\end{align*}
Termwise integration gives
\begin{align*}
 \int_0^1p(x)\,dx&=\frac{19354429323285160}{3},
 &\int_0^1xp(x)\,dx&=\frac{43818627005185708}{15},\\
 \int_0^1x^2p(x)\,dx&=\frac{35619115559228396}{21},
 &\int_0^1x^3p(x)\,dx&=\frac{16906516925584312}{15}.
\end{align*}
Dividing by $\int_0^1p(x)\,dx$ and taking the corresponding central
moments gives \eqref{eq:rho-moment-data}.

Assume \eqref{eq:equal-affine-pushforwards}.  Equality of variances gives
$b^2=(b')^2$.  If $b'=b$, equality of means gives $a'=a$.  Suppose
that $b'=-b$.  Equality of means gives
$a'=a+2b\mathbb E[Z]$, whereas equality of the third central moments
gives
\[
 b^3\mathbb E\!\left[(Z-\mathbb E[Z])^3\right]
 =-b^3\mathbb E\!\left[(Z-\mathbb E[Z])^3\right].
\]
The last number in \eqref{eq:rho-moment-data} is nonzero, so $b=0$.
Thus $b'=0$, and equality of means gives $a'=a$.
\end{proof}

\begin{cor}\label{cor:common-affine-data}
Notation and conditions are as in Set-up~\ref{setup:initial-filtrations}.  There
exist unique real numbers $a,b$ such that, after division by the actual
degree, both opposite initial filtrations have convex transform
\begin{equation}\label{eq:common-affine-transform}
 G^+(x,y)=G^-(x,y)=a+bx
\end{equation}
at every interior point of $\Omega$.  The entry probability laws of
$\mathcal T$, $\chi^+$, and
$\chi^-$ are all
\begin{equation}\label{eq:common-affine-law}
 (a+bx)_*\rho.
\end{equation}
In the exponent-one grading of $(X,A^e)$, the common law is
\begin{equation}\label{eq:common-affine-law-regraded}
 (e(a+bx))_*\rho.
\end{equation}
\end{cor}

\begin{proof}
Proposition~\ref{prop:affine-initials} gives affine transforms
$a_++b_+x$ and $a_-+b_-x$.  It also identifies both pushforwards of
$\rho$ with the entry law of $\mathcal T$ after division by the actual degree.
Lemma~\ref{lem:affine-pushforward-rigidity} gives
\[
 a_+=a_-=:a,
 \qquad
 b_+=b_-=:b.
\]
The formulas after division by the actual degree and in exponent-one degree
are the two conclusions of
Proposition~\ref{prop:affine-initials}.  Uniqueness follows from
Lemma~\ref{lem:affine-pushforward-rigidity}.
\end{proof}

\begin{lem}[Rational moments of algebraic Duistermaat--Heckman measures]
\label{lem:dh-rational-moments}
Let $(Y,H)$ be an $n$-dimensional polarized projective scheme, and let
$(\mathcal Y,\mathcal H)$ be an ample exponent-one algebraic test
configuration.  For all sufficiently large $k$, write
\begin{equation}\label{eq:dh-weight-decomposition}
 H^0(\mathcal Y_0,\mathcal H_0^k)
 =\bigoplus_{\lambda\in\mathbb Z}E_{k,\lambda}
\end{equation}
for the decomposition into weight spaces on the central fiber, and set
\begin{equation}\label{eq:unnormalized-dh-measures}
 \mu_k=k^{-n}\sum_{\lambda\in\mathbb Z}
 \dim E_{k,\lambda}\,\delta_{\lambda/k}.
\end{equation}
If $\mu_k$ converges weakly to $\mu$, then
\begin{equation}\label{eq:rational-dh-moments}
 \int_{\mathbb R}x^q\,d\mu(x)\in\mathbb Q
\end{equation}
for every nonnegative integer $q$.
\end{lem}

\begin{proof}
We first express all weight moments by rational generating functions.  We
then recover the limiting moments from the rational leading coefficients of
the resulting quasipolynomials.

\medskip

\noindent\textbf{Step 1.} We construct a rational generating function for
each unnormalized weight moment.
Put
\[
 R=\bigoplus_{k\geq0}H^0(\mathcal Y_0,\mathcal H_0^k),
 \qquad
 D=R_0=H^0(\mathcal Y_0,\mathcal O_{\mathcal Y_0}).
\]
Properness of $\mathcal Y_0$ makes $D$ finite-dimensional over
$\mathbb C$.  Choose a homogeneous basis of the finite-dimensional
$\mathbb C^*$-module $D$, and choose homogeneous generators $x_i$ of
$R$ as a $D$-algebra.  Write the bidegree of $x_i$ as
$(d_i,w_i)$, where $d_i>0$.  Give the variables of
$S=\mathbb C[X_i]$ the corresponding bidegrees.  The chosen basis of
$D$ makes $R$ a finite multigraded
$S$-module.  A finite multigraded free resolution over $S$ gives
\begin{equation}\label{eq:bigraded-hilbert-series}
 \sum_{k,\lambda}\dim E_{k,\lambda}t^kz^\lambda
 =\frac{P(t,z)}{\prod_i(1-t^{d_i}z^{w_i})},
\end{equation}
where $P(t,z)$ is a Laurent polynomial with integer coefficients.

Fix $q\geq0$.  Acting by $(z\,d/dz)^q$ on
\eqref{eq:bigraded-hilbert-series} and setting $z=1$, we obtain a rational
one-variable series
\[
 \sum_{k\geq0}s_q(k)t^k,
 \qquad
 s_q(k)=\sum_\lambda\lambda^q\dim E_{k,\lambda},
\]
whose denominator is a product of powers of $1-t^{d_i}$.  Let $r$ be
a common multiple of the integers $d_i$.  Partial fractions after
splitting the denominator over the $r$-th roots of unity show that
$s_q(k)$ is, for all sufficiently large $k$, a quasipolynomial with
rational coefficients and period dividing $r$.

\medskip

\noindent\textbf{Step 2.} We identify the limiting moment with a rational
leading coefficient.
Algebraicity gives a uniform linear bound on the central weights.  Hence the
supports of the measures in \eqref{eq:unnormalized-dh-measures} lie in one
compact interval.  Weak convergence implies convergence of every fixed
moment, and
\[
 \int_{\mathbb R}x^q\,d\mu_k(x)=k^{-(n+q)}s_q(k).
\]
On each residue class modulo $r$, the limit is the coefficient of
$k^{n+q}$ in a polynomial with rational coefficients, with value zero if
the degree is smaller.  The full sequence has one weak limit, so the
residue-class limits agree.  Their common value is the moment in
\eqref{eq:rational-dh-moments}.
\end{proof}

\begin{cor}\label{cor:rational-affine-data}
The coefficients $a$ and $b$ in
Corollary~\ref{cor:common-affine-data} belong to $\mathbb Q$.
\end{cor}

\begin{proof}
Lemma~\ref{lem:dh-rational-moments}, applied to $\mathcal T$ regarded as an
exponent-one test configuration of $(X,A^e)$, shows that the zeroth moment is the
positive rational leading Hilbert coefficient.  Dividing by this mass shows
that every moment of the associated central-weight probability law is
rational.  The
convention for increasing entries in Set-up~\ref{setup:initial-filtrations}
provides $\sigma\in\{1,-1\}$ such that this law is the law of
\begin{equation}\label{eq:algebraic-affine-random-variable}
 Y=\sigma e(a+bZ),
 \qquad
 Z\sim\rho.
\end{equation}

Suppose that $b\neq0$.  Rationality of the second and third central
moments of $Y$, together with \eqref{eq:rho-moment-data}, gives
\begin{equation}\label{eq:rational-affine-powers}
 e^2b^2\Var(Z)\in\mathbb Q,
 \qquad
 \sigma e^3b^3\mathbb E\!\left[(Z-\mathbb E[Z])^3\right]\in\mathbb Q.
\end{equation}
Both coefficients multiplying $b^2$ and $b^3$ in
\eqref{eq:rational-affine-powers} are nonzero rational numbers.  Thus
$b^2,b^3\in\mathbb Q$, and
\[
 b=\frac{b^3}{b^2}\in\mathbb Q.
\]
The mean of $Y$ is rational.  Together, \eqref{eq:rho-moment-data} and
\eqref{eq:algebraic-affine-random-variable} then give $a\in\mathbb Q$.
If $b=0$, rationality of the mean gives $a\in\mathbb Q$ directly.
\end{proof}

\subsection{Normalization and the Donaldson--Futaki invariant}

We record the normalization formula, including the codimension and
Duistermaat--Heckman conclusions that will be used after base change.

\begin{lem}[Normalization comparison]
\label{lem:normalization-comparison}
Let $(Y,H)$ be a polarized normal projective complex variety of dimension
$n$.  Let $(\mathcal Y,\mathcal H)$ be an integral ample algebraic test
configuration whose generic polarized fiber is $(Y,H^e)$, and let
\begin{equation}\label{eq:normalization-morphism}
 \nu\colon(\mathcal Y^\nu,\mathcal H^\nu)
 \to(\mathcal Y,\mathcal H)
\end{equation}
be its normalization, where $\mathcal H^\nu=\nu^*\mathcal H$.  Write
\begin{equation}\label{eq:normalization-structure-morphisms}
 p\colon\mathcal Y\to\mathbb A^1,
 \qquad
 p^\nu=p\circ\nu\colon\mathcal Y^\nu\to\mathbb A^1,
 \qquad
 R=\mathbb C[t].
\end{equation}
Write
\begin{align}
 h(m)&=A_0m^n+A_1m^{n-1}+O(m^{n-2}),
 \label{eq:normalization-hilbert-expansion}\\
 w_{\mathcal Y}(m)&=B_0m^{n+1}+B_1m^n+O(m^{n-1}),
 \label{eq:normalization-source-weight-expansion}\\
 w_{\mathcal Y^\nu}(m)&=B_0m^{n+1}+B_1^\nu m^n+O(m^{n-1}),
 \label{eq:normalization-target-weight-expansion}
\end{align}
using increasing-entry totals and $A_0>0$.  Put
\begin{equation}\label{eq:normalization-quotient}
 \mathcal Q=\nu_*\mathcal O_{\mathcal Y^\nu}/\mathcal O_{\mathcal Y}.
\end{equation}
For all sufficiently large $m$, put
\begin{equation}\label{eq:normalization-relative-pushforwards}
 \mathcal E_m=p_*\mathcal H^m,
 \qquad
 \mathcal E_m^\nu=p^\nu_*(\mathcal H^\nu)^m,
 \qquad
 \mathcal T_m=p_*(\mathcal Q\otimes\mathcal H^m),
\end{equation}
and let $\ell(m)$ be the $R$-length of
$\Gamma(\mathbb A^1,\mathcal T_m)$.  Then there exists
$c\geq0$ such that
\begin{equation}\label{eq:normalization-length}
 \ell(m)=cm^n+O(m^{n-1}),
\end{equation}
and
\begin{equation}\label{eq:normalization-df-comparison}
 \DF(\mathcal Y,\mathcal H)
 -\DF(\mathcal Y^\nu,\mathcal H^\nu)=\frac{2c}{A_0}.
\end{equation}
If the two Donaldson--Futaki invariants are equal, then
\begin{equation}\label{eq:normalization-codimension}
 \operatorname{codim}_{\mathcal Y}\operatorname{Supp}(\mathcal Q)\geq2.
\end{equation}
If the two Donaldson--Futaki invariants are equal, the total-variation
distance between the normalized probability measures of the increasing
entries in degree $m$ is $O(m^{-1})$; equivalently,
the reflected central-weight measures have this property.  The two test
configurations have equal Duistermaat--Heckman probability measures.
\end{lem}

\begin{proof}
We compare the two extension lattices, read their determinant-weight
difference from Smith normal form, and then interpret the equality case by
the Hilbert polynomial of the normalization quotient.

\medskip

\noindent\textbf{Step 1.} We express the weight difference as the length of
the normalization quotient.
The normalization is an isomorphism away from the central fiber, so
$\mathcal Q$ is a coherent sheaf supported on that fiber.  Tensor the
defining sequence of $\mathcal Q$ by $\mathcal H^m$.  Relative Serre
vanishing \cite[Tag~02O1]{Sta26} and the projection formula give, for all sufficiently large $m$,
the exact sequence of coherent $\mathcal O_{\mathbb A^1}$-modules
\begin{equation}\label{eq:normalization-relative-lattice-sequence}
 0\to\mathcal E_m\to\mathcal E_m^\nu\to\mathcal T_m\to0.
\end{equation}
The morphism $p$ is flat.  The normal integral total space
$\mathcal Y^\nu$ has no $R$-torsion, and torsion-free modules over the
principal ideal domain $R$ are flat; hence $p^\nu$ is flat as well.
Cohomology and base change therefore make $\mathcal E_m$ and
$\mathcal E_m^\nu$ locally free of equal rank $h(m)$ for large $m$.
The module $\mathcal T_m$ is finite and supported at $t=0$, so its global
sections have finite $R$-length.  Since $\mathbb A^1$ is affine, taking
global sections in \eqref{eq:normalization-relative-lattice-sequence} gives
\begin{equation}\label{eq:normalization-lattice-sequence}
 0\to\Gamma(\mathbb A^1,\mathcal E_m)
 \to\Gamma(\mathbb A^1,\mathcal E_m^\nu)
 \to\Gamma(\mathbb A^1,\mathcal T_m)
 \to0.
\end{equation}
The first two terms are free $R$-modules of rank $h(m)$, and
the last term is $t$-torsion.  Put the first map in Smith normal form,
with diagonal factors
\[
 t^{r_{m,1}},\ldots,t^{r_{m,h(m)}},
 \qquad
 r_{m,k}\geq0.
\]
The increasing-entry convention gives
\begin{equation}\label{eq:normalization-smith-length}
 \ell(m)=\sum_{k=1}^{h(m)}r_{m,k},
 \qquad
 w_{\mathcal Y}(m)=w_{\mathcal Y^\nu}(m)+\ell(m).
\end{equation}
The determinant of the lattice inclusion is a unit times
$t^{\ell(m)}$.  The inclusion is equivariant, so comparison of the two
equivariant determinant lines, with $t$ of increasing-entry weight one,
gives the second equality in \eqref{eq:normalization-smith-length}.  This
argument uses Smith normal form only for the underlying lattices and does not
require a Smith basis compatible with the weight decompositions.

\medskip

\noindent\textbf{Step 2.} We obtain the normalization formula from the
leading coefficient of the length polynomial.
Relative Serre vanishing identifies $\ell(m)$ with the Euler
characteristic, and hence with the Hilbert polynomial, of $\mathcal Q$ with
respect to $\mathcal H$.  Its degree is at most $n$, because $\mathcal Q$ is supported on the
$n$-dimensional central fiber.  Its coefficient at $m^n$ is
nonnegative and is positive exactly when $\mathcal Q$ has an
$n$-dimensional support component.  This proves
\eqref{eq:normalization-length}.  Comparing coefficients in
\eqref{eq:normalization-smith-length} gives $B_1=B_1^\nu+c$.  Substitution
into
\[
 \DF=\frac{2(B_1A_0-A_1B_0)}{A_0^2}
\]
gives \eqref{eq:normalization-df-comparison}.

\medskip

\noindent\textbf{Step 3.} We prove the codimension and measure statements
when the two Donaldson--Futaki invariants are equal.
By \eqref{eq:normalization-df-comparison}, we have $c=0$.  Thus
$\mathcal Q$ has support dimension at most $n-1$, which proves
\eqref{eq:normalization-codimension}, and $\ell(m)=O(m^{n-1})$.

Put $k_0=R/(t)$ and, in degree $m$,
\[
 M_m=\Gamma(\mathbb A^1,\mathcal E_m),\qquad
 M_m^\nu=\Gamma(\mathbb A^1,\mathcal E_m^\nu),\qquad
 Q_m=\Gamma(\mathbb A^1,\mathcal T_m).
\]
Tensoring \eqref{eq:normalization-lattice-sequence} with $k_0$ gives the
exact sequence
\begin{equation}\label{eq:normalization-central-tor-sequence}
 0\to \operatorname{Tor}^{R}_{1}(Q_m,k_0)
 \to M_m/tM_m
 \to M_m^\nu/tM_m^\nu
 \to Q_m/tQ_m\to0.
\end{equation}
The sequence is $\mathbb C^*$-equivariant, because
\eqref{eq:normalization-lattice-sequence} is equivariant and the ideal $(t)$
is invariant.  Cohomology and base change identify the two middle terms with
the degree-$m$ central-fiber section representations.  Both have dimension
$h(m)$.  Since $Q_m$ is supported at $t=0$, its $R$-length is its complex
vector-space dimension.  Hence
\[
 \dim_{\mathbb C}(Q_m/tQ_m)\leq\ell(m).
\]
The equality of the dimensions of the two middle terms in
\eqref{eq:normalization-central-tor-sequence} then gives
\[
 \dim_{\mathbb C}\operatorname{Tor}^{R}_{1}(Q_m,k_0)
 =\dim_{\mathbb C}(Q_m/tQ_m)\leq\ell(m).
\]
Let $I_m$ be the image of the middle arrow in
\eqref{eq:normalization-central-tor-sequence}.  It is a common
$\mathbb C^*$-representation: it is a quotient of $M_m/tM_m$ and a
subrepresentation of $M_m^\nu/tM_m^\nu$.  Finite-dimensional
$\mathbb C^*$-representations are semisimple, so the two weight multisets
have the weights of $I_m$ in common and have at most $\ell(m)$ unmatched
weights on either side.  Let $d_{\mathrm{TV}}$ denote total-variation
distance.  If $\mu_m$ and $\mu_m^\nu$ denote the empirical
probability measures of the central weights divided by $m$, then
\[
 d_{\mathrm{TV}}(\mu_m,\mu_m^\nu)
 \leq
 \frac{\dim\operatorname{Tor}^{R}_{1}(Q_m,k_0)
       +\dim(Q_m/tQ_m)}{h(m)}
 \leq\frac{2\ell(m)}{h(m)}=O(m^{-1}).
\]
Reflection under $x\mapsto-x$ preserves total-variation distance.  Thus the
increasing-entry measures also have distance $O(m^{-1})$, and the two
Duistermaat--Heckman probability measures coincide.
\end{proof}

\subsection{The integral product comparator}

We finally clear the rational affine coefficients by base change and use
the normalization comparison to retain both the zero invariant and the
Duistermaat--Heckman law.

\begin{prop}[Integral product comparator]
\label{prop:integral-product-comparator}
Notation and conditions are as in Set-up~\ref{setup:initial-filtrations}, and
let $a,b\in\mathbb Q$ be the coefficients in
Corollary~\ref{cor:common-affine-data}.  There exists a positive integer
$N$ with the following properties.  Let $\mathcal T'_N$ be the ordinary
base change by $t=(t')^N$, and let $\mathcal T_N$ be its normalization.
Then $\mathcal T'_N$ is integral, while $\mathcal T_N$ is a normal ample
test configuration for $(X,A^e)$ with
\begin{equation}\label{eq:normalized-base-change-df}
 \DF(\mathcal T_N)=0.
\end{equation}
Its central-weight Duistermaat--Heckman probability law is the $N$-dilation
of the central-weight law of $\mathcal T$.

The integers
\begin{equation}\label{eq:integral-affine-coefficients}
 \alpha_N=Nea,
 \qquad
 \beta_N=Nb
\end{equation}
define a polarized product test configuration $\mathcal P_N$.  Its
increasing jump on $V_{em,j}$ is
\begin{equation}\label{eq:product-comparator-jump}
 w_N(m,j)=\alpha_Nm+\beta_Nj=N(eam+bj).
\end{equation}
After division by the actual degree, both opposite initial filtrations of
$\mathcal T_N$, as well as the product filtration of $\mathcal P_N$, have
convex transform
\begin{equation}\label{eq:base-changed-affine-transform}
 Na+Nbx.
\end{equation}
\end{prop}

\begin{proof}
We clear the denominators of the affine coefficients, prove integrality of
the ordinary base change, compare it with its normalization, and then
construct the product configuration from the resulting integral affine
profile.

\medskip

\noindent\textbf{Step 1.} We choose the base-change order and prove that the
ordinary base change is integral.
Choose $N$ such that the two numbers in
\eqref{eq:integral-affine-coefficients} are integers, and form
\[
 \mathcal T'_N
 =\mathcal T\times_{\mathbb A^1_t}\mathbb A^1_{t'},
 \qquad
 t=(t')^N.
\]
The base-changed family is flat over $\mathbb C[t']$, so multiplication
by $t'$ is injective on every affine coordinate ring.  After inverting
$t'$, the marked family is the integral product
$X\times\mathbb G_m$.  Suppose that a product of two elements in an
affine coordinate ring vanishes.  One factor vanishes after localization,
so a power of $t'$ annihilates that factor.  Torsion-freeness over
$\mathbb C[t']$ makes the factor zero.  Thus every affine coordinate ring
is a domain, and $\mathcal T'_N$ is integral.

Let
\[
 \nu_N\colon\mathcal T_N\to\mathcal T'_N
\]
be the normalization.  Normalization is finite for a finite-type complex
scheme.  Finite pullback preserves relative ampleness and projectivity, and
the action and marking lift to $\mathcal T_N$.  Its normal integral total
space is torsion-free, hence flat, over the principal ideal domain
$\mathbb C[t']$.  Therefore $\mathcal T_N$ is a normal ample test
configuration for $(X,A^e)$.

\medskip

\noindent\textbf{Step 2.} We prove that normalization preserves the zero
invariant and the dilated Duistermaat--Heckman law in this situation.
Ordinary $N$-fold base change multiplies every central weight by $N$,
and hence
\begin{equation}\label{eq:base-change-df}
 \DF(\mathcal T'_N)=N\DF(\mathcal T)=0.
\end{equation}
Lemma~\ref{lem:normalization-comparison}, applied to $\nu_N$, gives the
following comparison.  If $c$ is
the coefficient in \eqref{eq:normalization-length} and $A_0>0$ is the
leading Hilbert coefficient, then
\[
 0-\DF(\mathcal T_N)=\frac{2c}{A_0}\geq0.
\]
Proposition~\ref{prop:semistability-all-exponents} gives
$\DF(\mathcal T_N)\geq0$.  Thus
\[
 \DF(\mathcal T_N)=0,
 \qquad
 c=0.
\]
The equality clause of Lemma~\ref{lem:normalization-comparison} identifies
the Duistermaat--Heckman laws of $\mathcal T'_N$ and $\mathcal T_N$.
Ordinary base change dilates every weight by $N$, so this common law is
the $N$-dilation of the law of $\mathcal T$.

\medskip

\noindent\textbf{Step 3.} We identify the common affine profile after base
change.
Corollary~\ref{cor:common-affine-data} applied to $\mathcal T_N$ gives a
common affine profile for its two opposite initial filtrations.  Reflecting
the dilated central-weight law to the increasing-entry coordinate, their
entry law after division by the actual degree is
\[
 (Na+Nbx)_*\rho.
\]
Lemma~\ref{lem:affine-pushforward-rigidity} identifies their common profile
with \eqref{eq:base-changed-affine-transform}.

\medskip

\noindent\textbf{Step 4.} We construct the integral product comparator.
In the convention for increasing entries in
Definition~\ref{defn:test-configuration}, the integer $\alpha_N$ is the
label of the scalar character and $\beta_N$ is the label of fiber scaling.
Let $\mathcal P_N$ be the resulting polarized product test
configuration.  Its increasing jump on $V_{em,j}$ is
$\alpha_Nm+\beta_Nj$, which is \eqref{eq:product-comparator-jump}.
Dividing by the actual degree $em$ gives
\[
 \frac{w_N(m,j)}{em}=Na+Nb\frac{j}{em}.
\]
Thus the product filtration has the transform in
\eqref{eq:base-changed-affine-transform}.
No degreewise equality between $\mathcal T_N$ and $\mathcal P_N$ is asserted
at this stage.  That equality is the conclusion of the marked Smith argument
in Section~\ref{sec:classification}.
\end{proof}
 
\section{Classification of zero-invariant test configurations}
\label{sec:classification}

In this section, we prove Theorem~\ref{thm:equality-classification}.

We use the increasing filtration convention fixed in
Set-up~\ref{setup:initial-filtrations}.  In particular, the section ring of
$A^e$ is the Veronese ring in \eqref{eq:initial-veronese-ring}, and its
blocks for fiber scaling are those in \eqref{eq:affine-initials-blocks}.

\subsection{Normal section algebras and marked rigidity}
\label{subsec:normal-section-algebras}

In this subsection, we prove that a sufficiently small marked relative Smith
spectrum determines a normal ample test configuration.

\begin{deflem}[Relative Smith spectrum]
\label{deflem:relative-smith-spectrum}
Let $\mathcal O$ be a discrete valuation ring with uniformizer $t$ and
fraction field $K$, and let $\Lambda_1,\Lambda_2$ be full $\mathcal O$-lattices
in an $N$-dimensional $K$-vector space.  There exist a $K$-basis
$e_1,\ldots,e_N$ and uniquely determined integers $r_1,\ldots,r_N$, up to
permutation, such that
\begin{equation}\label{eq:relative-smith-spectrum-definition}
 \Lambda_1=\bigoplus_{i=1}^N\mathcal Oe_i,
 \qquad
 \Lambda_2=\bigoplus_{i=1}^N\mathcal O t^{r_i}e_i.
\end{equation}
We call the multiset $\{r_1,\ldots,r_N\}$ the relative Smith spectrum of
$\Lambda_1$ and $\Lambda_2$.
\end{deflem}

\begin{proof}
We choose $\mathcal O$-bases of the two lattices and let
$A\in\operatorname{GL}_N(K)$
be the matrix whose columns are the coordinates of the second basis in the
first.  Choose an integer $M$ such that $t^M A$ has entries in $\mathcal O$.
For a nonzero matrix over $\mathcal O$, move an entry of minimum valuation to
the upper-left corner.  This entry divides every other entry because
$\mathcal O$ is a discrete valuation ring.  Elementary row and column
operations over $\mathcal O$ therefore clear its column and row.  Induction
on $N$, followed by absorbing units into the bases, gives
\begin{equation}\label{eq:relative-smith-diagonalization}
 U(t^M A)V=\operatorname{diag}(t^{d_1},\ldots,t^{d_N}),
 \qquad
 U,V\in\operatorname{GL}_N(\mathcal O),
 \qquad
 0\leq d_1\leq\cdots\leq d_N.
\end{equation}
If the original bases are written as row vectors $x$ and $y=xA$, then
$xU^{-1}$ is a basis of $\Lambda_1$ and $yV$ is a basis of $\Lambda_2$.
Thus \eqref{eq:relative-smith-spectrum-definition} holds with
$r_i=d_i-M$.  For uniqueness, let $I_q(A)$ be the fractional ideal generated
by the $q$-by-$q$ minors of $A$.  Left or right multiplication by an element
of $\operatorname{GL}_N(\mathcal O)$ does not change this ideal.  If the
integers are ordered increasingly, the diagonal form gives
\begin{equation}\label{eq:relative-smith-minor-valuations}
 \operatorname{ord}_t I_q(A)=r_1+\cdots+r_q,
 \qquad 1\leq q\leq N.
\end{equation}
The successive differences of these intrinsic valuations determine the
ordered list $r_1,\ldots,r_N$, and hence determine the multiset.
\end{proof}

\begin{lem}[Normal Veronese section algebras]
\label{lem:normal-veronese-section-algebra}
Let
\begin{equation}
 \pi\colon \mathcal Y\to \mathbb A^1
 \label{eq:normal-section-family}
\end{equation}
be a flat projective morphism such that $\mathcal Y$ is normal and integral
and $\pi_*\mathcal O_{\mathcal Y}=\mathcal O_{\mathbb A^1}$.  Let
$\mathcal H$ be a $\pi$-ample $\mathbb C^*$-linearized line bundle, where the
$\mathbb C^*$-action on $\mathcal Y$ covers the standard scaling action on
$\mathbb A^1$, and put
\begin{equation}
 S_m=H^0(\mathcal Y,\mathcal H^m),
 \qquad
 S=\bigoplus_{m\geq0}S_m.
 \label{eq:relative-section-algebra}
\end{equation}
Then $S$ is a normal finitely generated domain over $\mathbb C[t]$.  Moreover,
there exists a positive integer $r$ such that
\begin{equation}
 S^{(r)}=\bigoplus_{k\geq0}S_{rk}
 \label{eq:normal-section-veronese}
\end{equation}
is a normal finitely generated domain generated over $\mathbb C[t]$ by
$S_r$, and the same holds for every positive multiple of $r$.  In
particular, $\mathcal H^r$ defines a relatively projectively
normal equivariant embedding.

Let $(Y,H)$ be the generic polarized fiber, and fix an equivariant product
identification over $\mathbb G_m$.  Define the decreasing extension-order filtration
\begin{equation}
 F^\lambda H^0(Y,H^{rk})
 =\left\{s\mid t^{-\lambda}s\text{ extends to a section of }
 \mathcal H^{rk}\text{ on }\mathcal Y\right\}.
 \label{eq:extension-order-filtration}
\end{equation}
Then this filtration has Rees algebra
\begin{equation}
 S^{(r)}
 =\bigoplus_{k\geq0}\bigoplus_{\lambda\in\mathbb Z}
 F^\lambda H^0(Y,H^{rk})t^{-\lambda},
 \label{eq:extension-order-rees-algebra}
\end{equation}
with its section degree and $\mathbb C^*$-weight retained.
If
\begin{equation}
 E_\ell H^0(Y,H^{rk})=F^{-\ell}H^0(Y,H^{rk})
 \label{eq:increasing-decreasing-filtration-reindexing}
\end{equation}
denotes the corresponding increasing filtration, then the marked extension
lattice in \eqref{eq:extension-order-rees-algebra} can be written as
\begin{equation}
 S^{(r)}
 =\bigoplus_{k\geq0}\sum_{\ell\in\mathbb Z}
 t^\ell E_\ell H^0(Y,H^{rk}).
 \label{eq:increasing-extension-lattice}
\end{equation}
\end{lem}

\begin{proof}
We first prove finite generation of the full section algebra, then choose a
Veronese generated in degree one.  We next prove normality by divisorial
valuations.  Finally, we identify the equivariant section algebra with the
Rees algebra of the extension-order filtration.

\medskip
\noindent\textbf{Step 1.} In this step, we prove that $S$ is finitely
generated over $\mathbb C[t]$.

Choose $q>0$ such that $\mathcal H^q$ is relatively very ample and gives a
closed immersion
\[
 \mathcal Y\hookrightarrow\mathbb P^N_{\mathbb C[t]}.
\]
For $0\leq a<q$, the module
\[
 M_a=\bigoplus_{k\geq0}H^0(\mathcal Y,\mathcal H^{qk+a})
\]
is the section module of the coherent sheaf $\mathcal H^a$ for this
embedding.  For a finite graded presentation of its pushforward to
$\mathbb P^N_{\mathbb C[t]}$, the higher cohomology of the finitely many
coherent relation sheaves vanishes after sufficiently large twists.
Consequently, finitely many homogeneous sections generate all sufficiently
large components over $\mathbb C[t][x_0,\ldots,x_N]$.  After adjoining the finitely
many lower components, each $M_a$ is finite over this polynomial ring.  The
finite direct sum of the modules $M_a$ is $S$, so $S$ is a finitely generated
$\mathbb C[t]$-algebra.

\medskip
\noindent\textbf{Step 2.} In this step, we choose a Veronese of $S$ generated
in degree one.

Choose homogeneous algebra generators $x_1,\ldots,x_h$ of positive degrees
$d_1,\ldots,d_h$, and let $d$ be a common multiple of the integers $d_i$.
Inside $S^{(d)}$, with its rescaled grading, put
\[
 B=\mathbb C[t]
 [x_1^{d/d_1},\ldots,x_h^{d/d_h}].
\]
Each $x_i$ is integral over $B$.  Hence $S$, and therefore
$C:=S^{(d)}$, is finite over $B$.  Choose homogeneous $B$-module generators
$z_1,\ldots,z_s$ of $C$, of rescaled degrees $e_1,\ldots,e_s$, and choose
$n\geq\max_i e_i$.  We prove that $C^{(n)}$ is generated by $C_n$.

Let $k\geq2$ and $c\in C_{kn}$.  Write
\[
 c=\sum_i b_i z_i,
 \qquad
 b_i\in B_{kn-e_i}.
\]
Since $B$ is standard graded, the multiplication map
\[
 B_{n-e_i}\otimes B_{(k-1)n}\to B_{kn-e_i}
\]
is surjective.  Every term $b_i z_i$ is consequently a sum of products of
an element of $C_n$ and an element of $C_{(k-1)n}$.  Induction on $k$ proves
that $C^{(n)}$ is generated by $C_n$.  Replace $dn$ by a positive multiple
$r$ for which $\mathcal H^r$ is relatively very ample.  The $r$-th Veronese
remains generated in degree one, which proves the generation assertion in
\eqref{eq:normal-section-veronese}.

\medskip
\noindent\textbf{Step 3.} In this step, we prove that $S$ and $S^{(r)}$ are
normal domains.

Choose a nonzero rational section $u$ of $\mathcal H$, and let $D$ be its
Cartier divisor.  If $K$ is the function field of $\mathcal Y$, then $S$
identifies with the subring of $K[U]$ formed by the terms $fU^m$ such that
\[
 m\geq0,
 \qquad
 \operatorname{div}(f)+mD\geq0.
\]
For every prime divisor $E$ of $\mathcal Y$, this condition is
\[
 \operatorname{ord}_E(f)+m\operatorname{coeff}_E(D)\geq0.
\]
For every prime divisor $E$, let $R_E\subset K[U]$ be the homogeneous
valuation subring defined by this inequality.  Thus
$S=K[U]\cap\bigcap_E R_E$.  Each ring in this intersection is integrally closed, and
therefore $S$ is a normal domain.

The Veronese $S^{(r)}$ is the invariant subring for the finite cyclic action
that multiplies $S_m$ by the $m$-th power of a primitive $r$-th root of
unity.  If an element of $\operatorname{Frac}\left(S^{(r)}\right)$ is
integral over $S^{(r)}$, then it is integral over $S$, belongs to $S$, and is
fixed by the cyclic action.  It belongs to $S^{(r)}$, which proves the
normality of the Veronese.

\medskip
\noindent\textbf{Step 4.} In this step, we identify the Rees algebra after
fixing the product marking over $\mathbb G_m$.

The filtration in \eqref{eq:extension-order-filtration} is multiplicative
because products of extensions extend.  The $\mathbb C^*$-action decomposes
every $S_{rk}$ into weight spaces.  A section homogeneous for this weight restricts to
a Laurent monomial $t^{-\lambda}s$, which extends precisely when
$s\in F^\lambda H^0(Y,H^{rk})$.  Taking weight components preserves the
module of global sections.  Hence extension holds coefficient by coefficient,
and \eqref{eq:extension-order-rees-algebra} follows.
\end{proof}

\begin{lem}[Marked affine module gluing]
\label{lem:marked-affine-module-gluing}
Let $A=\mathbb C[t]$, let $K=\mathbb C(t)$, and let $W$ be a
finite-dimensional $K$-vector space.  For $i=1,2$, let
$M_i\subset W$ be a finite torsion-free $A$-module spanning $W$.  Suppose
that $f\in A\setminus(t)$ and that
\begin{equation}\label{eq:marked-affine-local-equalities}
 (M_1)_f=(M_2)_f,
 \qquad
 (M_1)_t=(M_2)_t
\end{equation}
as submodules of $W$.  Then $M_1=M_2$ inside $W$.
\end{lem}

\begin{proof}
The condition $f(0)\ne0$ gives $(f,t)=A$, so $D(f)$ and $D(t)$ cover
$\operatorname{Spec}A$.  For either module, the affine \v{C}ech sequence for
this cover is exact.  Torsion freeness embeds all of its terms in $W$, and
therefore
\begin{equation}\label{eq:marked-affine-intersection}
 M_i=(M_i)_f\cap(M_i)_t
 \quad\text{inside }W.
\end{equation}
The overlap maps are the restrictions of the identity of $W$.  Substitution
of \eqref{eq:marked-affine-local-equalities} in
\eqref{eq:marked-affine-intersection} proves the assertion.
\end{proof}

\begin{prop}[Rigidity from the marked Smith spectrum]
\label{prop:marked-smith-rigidity}
Let $(Y,H)$ be a connected normal $n$-dimensional polarized complex
projective variety.  For $i=1,2$, let
\begin{equation}
 (\mathcal Y_i,\mathcal H_i)\to\mathbb A^1
 \label{eq:marked-test-families}
\end{equation}
be a normal ample algebraic test configuration for $(Y,H)$.  Equip the two
configurations with a common $\mathbb C^*$-equivariant product trivialization
over $\mathbb G_m$, including a common marking of the generic line bundle and
a common linearization convention.  Their marked relative section algebras
\begin{equation}
 S_i=\bigoplus_{d\geq0}H^0(\mathcal Y_i,\mathcal H_i^d)
 \label{eq:marked-section-algebras}
\end{equation}
lie in the common generic graded algebra
\begin{equation}
 \bigoplus_{d\geq0}H^0(Y,H^d)
 \otimes_{\mathbb C}\mathbb C[t,t^{-1}].
 \label{eq:marked-generic-algebra}
\end{equation}

Choose a common positive integer $r$ for which the $r$-th Veronese of each
$S_i$ is normal and generated in rescaled degree one, as in
Lemma~\ref{lem:normal-veronese-section-algebra}.  Put
\begin{align}
 \mathcal O&=\mathbb C[t]_{(t)},
 &K&=\mathbb C(t),
 &V_m&=H^0(Y,H^{rm})\otimes_{\mathbb C}K,
 \notag\\
 \Lambda_m^i
 &=H^0(\mathcal Y_i,\mathcal H_i^{rm})
 \otimes_{\mathbb C[t]}\mathcal O
 \subset V_m.
 \label{eq:marked-extension-lattices}
\end{align}
Put $N_m=\dim_KV_m$.
Let $r'_{m,1},\ldots,r'_{m,N_m}$ be the relative Smith elementary divisors
of $\Lambda_m^1$ and $\Lambda_m^2$ in the convention of
Definition-Lemma~\ref{deflem:relative-smith-spectrum}.  Assume that
\begin{equation}
 \sum_{j=1}^{N_m}(r'_{m,j})^2=o(m^{n+2}).
 \label{eq:smith-rigidity-hypothesis}
\end{equation}
Then $S_1=S_2$ inside \eqref{eq:marked-generic-algebra}.  Consequently, the
two test configurations are isomorphic as marked polarized
$\mathbb C^*$-equivariant test configurations.
\end{prop}

\begin{proof}
We prove Proposition~\ref{prop:marked-smith-rigidity} in three steps.  First,
we record a min--max
criterion for the relative Smith integers.  Second, a divisorial valuation
turns any difference between the two localized Veronese algebras into a
quadratic lower bound that contradicts
\eqref{eq:smith-rigidity-hypothesis}.  Third, normality removes the common
Veronese and relative $\operatorname{Proj}$ recovers the marked test
configuration.

Put
\[
 \widetilde R_i=S_i^{(r)},
 \qquad
 R_i=(\widetilde R_i)_{(t)}
     =\bigoplus_{m\geq0}\Lambda_m^i.
\]
By Lemma~\ref{lem:normal-veronese-section-algebra}, each $\widetilde R_i$ is
a normal domain generated in rescaled degree one.  Each localization $R_i$
is also a normal domain generated in rescaled degree one, and
\[
 R_1[t^{-1}]=R_2[t^{-1}]
\]
inside the common marked generic algebra.

\medskip
\noindent\textbf{Step 1.} In this step, we relate lattice-order separation
on a subspace to the relative Smith spectrum.

For an $\mathcal O$-lattice $\Lambda$ in a finite-dimensional $K$-vector
space and a nonzero element $z$, put
\[
 \ell_\Lambda(z)=\max\{q\in\mathbb Z\mid z\in t^q\Lambda\}.
\]
Order the Smith integers decreasingly.  Let $\delta>0$, and let
$W\subset V_m$ be a
$D$-dimensional $K$-subspace such that
\[
 \ell_{\Lambda_m^1}(z)-\ell_{\Lambda_m^2}(z)\geq\delta
\]
for every nonzero $z\in W$.  Choose a Smith basis $e_1,\ldots,e_{N_m}$ with
\[
 \Lambda_m^1=\bigoplus_j\mathcal O\,e_j,
 \qquad
 \Lambda_m^2=\bigoplus_j\mathcal O t^{r'_{m,j}}e_j.
\]
For $z=\sum_j c_j e_j$, we have
\[
 \ell_{\Lambda_m^1}(z)=\min_j\operatorname{ord}_t(c_j),
 \qquad
 \ell_{\Lambda_m^2}(z)
 =\min_j\left(\operatorname{ord}_t(c_j)-r'_{m,j}\right).
\]
If fewer than $D$ Smith integers were at least $\delta$, then $W$ would
meet the span of the remaining basis vectors nontrivially, contradicting the
order separation.  Thus at least $D$ Smith integers are at least $\delta$.
If instead
$\ell_{\Lambda_m^1}(z)-\ell_{\Lambda_m^2}(z)\leq-\delta$ on a
$D$-dimensional subspace, the argument with the two lattices interchanged
shows that at least $D$ Smith integers are at most $-\delta$.

\medskip
\noindent\textbf{Step 2.} In this step, we prove the quantitative
contrapositive to \eqref{eq:smith-rigidity-hypothesis}.

Assume that $R_1\neq R_2$.  Since both algebras are generated in degree one,
their degree-one lattices differ.  After interchanging the algebras if
necessary, choose
\[
 x\in(R_1)_1\setminus(R_2)_1.
\]
Since $R_2$ is a normal Noetherian domain, it is the intersection of its
localizations at the height-one prime ideals inside its fraction field by
\cite[Tag~031T]{Sta26}.
The element $x\notin R_2$ therefore has negative order along a height-one
prime $P$; write its normalized valuation as
\[
 v_P(x)=-a<0.
\]
We may choose $P$ homogeneous.  The finite set of negative prime components
of the divisor of the homogeneous element $x$ is preserved by the connected
section-grading torus, so each component is fixed.  Moreover, $t\in P$.  If
$t\notin P$, then $t$ is a unit in $(R_2)_P$ and
$x\in R_2[t^{-1}]\subset(R_2)_P$, contradicting $v_P(x)<0$.  Put
$b=v_P(t)>0$.

Since $t\in P$, we have $P\cap\mathcal O=(t)$.  The discrete valuation ring
$\mathcal O$ is Cohen--Macaulay and hence universally catenary by
\cite[Tag~00NM]{Sta26}.  Its finite-type algebra $R_2$ is therefore
catenary by the definition of universal catenarity.  Let
\[
 \mathfrak m=(t)+(R_2)_{>0}.
\]
The generic fiber of $R_2$ is the section ring of the $n$-fold
$(Y,H^r)$, so its fraction field has transcendence degree $n+1$ over
$K$.  The dimension formula
\cite[Tag~02IJ]{Sta26} therefore gives
$\operatorname{ht}(\mathfrak m)=1+(n+1)=n+2$.  Catenarity of
the chain $(0)\subset P\subset\mathfrak m$ gives
$\operatorname{ht}(\mathfrak m/P)=n+1$.  Since $R_2/P$ is a standard
graded complex domain and $\mathfrak m/P$ is its irrelevant ideal,
\cite[Tag~00P6]{Sta26} gives
\[
 \dim(R_2/P)=n+1.
\]
The ideal $P$ is homogeneous and the degree-zero part of the quotient is
$\mathcal O/(t)=\mathbb C$.  Thus $R_2/P$ is a standard graded complex
domain of dimension $n+1$.  Its Hilbert function agrees in large degree
with a polynomial of degree $n$ and positive leading coefficient.  Hence
\[
 D_l=\dim_{\mathbb C}(R_2/P)_l
 =c_P l^n+O(l^{n-1})
\]
for a positive constant $c_P$.  Choose lifts
\[
 s_{l,1},\ldots,s_{l,D_l}\in(R_2)_l
\]
of a complex basis modulo $P$.  If $c_j(t)\in K$ are not all zero and
$h=\min_j\operatorname{ord}_t(c_j(t))$, put
$d_j=t^{-h}c_j\in\mathcal O$.  At least one $d_j$ has nonzero image in
$\mathcal O/(t)$.  The images of the $s_{l,j}$ form a complex basis modulo
$P$, so $\sum_jd_js_{l,j}$ has nonzero image modulo $P$.  Factoring out
$t^h$ therefore gives
\[
 v_P\left(\sum_j c_j(t)s_{l,j}\right)=bh.
\]
In particular, the chosen lifts are $K$-linearly independent.

The degree-one lattices are commensurable.  Choose $C\geq0$ such that
\[
 t^C(R_2)_1\subset(R_1)_1.
\]
Generation in degree one gives
\[
 t^{Cl}(R_2)_l\subset(R_1)_l
 \qquad(l\geq1).
\]
For positive integers $k,l$, put $m=k+l$, and let $W_{k,l}$ be the
$K$-subspace spanned by
\[
 x^k t^{Cl}s_{l,j},
 \qquad
 1\leq j\leq D_l.
\]
It has dimension $D_l$.  For a nonzero element
\[
 z=x^k t^{Cl}\sum_j c_j(t)s_{l,j},
\]
with $h=\min_j\operatorname{ord}_t(c_j(t))$, the generators
$x^k t^{Cl}s_{l,j}$ lie in $\Lambda_m^1$, and hence
\[
 \ell_{\Lambda_m^1}(z)\geq h.
\]
On the other hand,
\[
 v_P(z)=-ak+Clb+bh.
\]
If $z\in t^q\Lambda_m^2$, then $v_P(z)\geq qb$.  Therefore
\[
 \ell_{\Lambda_m^2}(z)
 \leq h+Cl-\frac{a}{b}k,
\]
and
\begin{equation}
 \ell_{\Lambda_m^1}(z)-\ell_{\Lambda_m^2}(z)
 \geq\frac{a}{b}k-Cl.
 \label{eq:lattice-order-separation}
\end{equation}

Choose a positive rational number $\epsilon$ such that
$C\epsilon<a/(2b)$, with any positive $\epsilon$ allowed when $C=0$.
Take infinitely many positive pairs $(k,l)$ with
$l/k\to\epsilon$.  There exists $\gamma>0$ such that the right-hand side
of \eqref{eq:lattice-order-separation} is at least $\gamma m$.  The integer
$l$ is a fixed positive proportion of $m$, so $D_l\geq c_1m^n$ for a
positive constant $c_1$.  The $D_l$-dimensional subspace used to obtain
\eqref{eq:lattice-order-separation} therefore forces at least $D_l$ relative
Smith integers to have absolute value at least $\gamma m$.  Hence
\[
 \sum_j(r'_{m,j})^2\geq c_1\gamma^2m^{n+2}
\]
along an infinite sequence.  This contradicts
\eqref{eq:smith-rigidity-hypothesis}, and we conclude that $R_1=R_2$.

\medskip
\noindent\textbf{Step 3.} In this step, we remove the common Veronese and
recover the marked test configurations.

We first recover the two Veronese algebras over $\mathbb C[t]$.  Equality
$R_1=R_2$ and finite generation allow us to clear the finitely many
denominators in homogeneous generating sets.  Hence there exists
$f\in\mathbb C[t]\setminus(t)$ such that
\[
 (\widetilde R_1)_f=(\widetilde R_2)_f.
\]
Over $D(t)$, the fixed punctured marking gives
$(\widetilde R_1)_t=(\widetilde R_2)_t$.  Since $f(0)\neq0$, the open sets
$D(f)$ and $D(t)$ cover $\Spec\mathbb C[t]$.  The two
identifications are restrictions of the identity inside
\eqref{eq:marked-generic-algebra}.  If
$M_{i,d}=(S_i^{(r)})_d$ and
\begin{equation}\label{eq:smith-common-degree-ambient}
 W_d=H^0(Y,H^{rd})\otimes_{\mathbb C}\mathbb C(t),
\end{equation}
then $M_{1,d}$ and $M_{2,d}$ are finite torsion-free
$\mathbb C[t]$-submodules of the common space $W_d$.  Applying
Lemma~\ref{lem:marked-affine-module-gluing} degree by degree gives
\begin{equation}\label{eq:smith-degreewise-gluing}
 M_{i,d}=(M_{i,d})_f\cap(M_{i,d})_t
 \quad\text{inside }W_d
\end{equation}
and shows that the two sides agree for $i=1,2$.  Hence
\[
 S_1^{(r)}=S_2^{(r)}.
\]

Lemma~\ref{lem:normal-veronese-section-algebra} states that each full section
algebra $S_i$ is normal.  If $y$ is homogeneous in $S_1$, then
$y^r\in S_2^{(r)}$.  Thus $y$ is integral over $S_2$.  It also belongs to
$S_2[t^{-1}]$, and hence to $\operatorname{Frac}(S_2)$.  Normality gives
$y\in S_2$.  Symmetry gives $S_1=S_2$ inside the marked generic algebra.
For each $i$, the canonical morphism
\begin{equation}\label{eq:smith-relative-proj-reconstruction}
 \mathcal Y_i\to
 \operatorname{Proj}_{\mathbb A^1}S_i
\end{equation}
is an isomorphism by \cite[Tag~0C6J]{Sta26}.  Under this isomorphism,
the canonical evaluation map identifies
$\widetilde{S_i(1)}$ with $\mathcal H_i$ by
\cite[Tag~01QI]{Sta26}.  Equality of the full graded algebras therefore
recovers the original polarization, not only its $r$th power.  The grading,
common generic marking, and linearization are preserved, so the resulting
isomorphism is an isomorphism of marked polarized equivariant test
configurations.
\end{proof}

\subsection{Residual filtrations and the oriented Smith spectrum}
\label{subsec:residual-smith-spectrum}

In this subsection, we identify the relative Smith spectrum with the jumps of
the oriented initial and prove quadratic decay after subtracting the integral
product profile.

\begin{thm}[Filtered jump measures, {\cite[Theorem~1.11 and
Remark~1.12(i)]{BC11}}]
\label{thm:filtered-jump-measures}
Let $Y$ be a projective variety, let $H$ be an ample line bundle on $Y$, and let
$R(H)=\bigoplus_{m\geq0}H^0(Y,H^m)$ carry a multiplicative filtration.
Assume that, in the decreasing convention, the filtration is pointwise left
bounded and linearly right bounded.  Fix a full-rank valuation with
one-dimensional leaves, and let $G$ be the associated concave transform on
the Newton--Okounkov body.  If $N_m=h^0(Y,H^m)$ and
$\lambda_{m,1},\ldots,\lambda_{m,N_m}$ are the filtration jumps in degree
$m$, counted with multiplicity, then
\begin{equation}
 \frac{1}{N_m}\sum_{k=1}^{N_m}
 \delta_{\lambda_{m,k}/m}
 \to
 G_*\left(\frac{\operatorname{Leb}}
 {\vol(\Delta(H))}\right)
 \label{eq:filtered-jump-measure-limit}
\end{equation}
weakly, where $\Delta(H)$ is the Newton--Okounkov body of $H$.
\end{thm}

\begin{lem}[Quadratic decay of the residual jumps]
\label{lem:residual-quadratic-decay}
Notation and conditions are as in
Proposition~\ref{prop:integral-product-comparator}.  Let $\chi$ be an
integral increasing filtration of $R^{[e]}$ that preserves every block
$V_{em,j}$.  Assume that $\chi$ is multiplicative and two-sided linearly
bounded, and that its convex transform after division by the actual
$A$-degree $d=em$ is
\eqref{eq:base-changed-affine-transform}.  If
$i^\chi_{em,j,\gamma}$ are its jumps, where
$1\leq\gamma\leq\dim V_{em,j}$, define the residual filtration $\eta$ by
\begin{equation}\label{eq:residual-filtered-pieces}
 \eta_{\leq\ell}\cap V_{em,j}
 =\chi_{\leq\ell+w_N(m,j)}\cap V_{em,j}.
\end{equation}
Its jumps satisfy
\begin{equation}
 i^\eta_{em,j,\gamma}
 =i^\chi_{em,j,\gamma}-w_N(m,j),
 \label{eq:residual-jumps}
\end{equation}
where $w_N(m,j)$ is the integral product jump in
\eqref{eq:product-comparator-jump}.  Then $\eta$ is an integral,
multiplicative, two-sided linearly bounded filtration of the full ring
$R^{[e]}$, and its convex transform is zero.  If
\begin{equation}
 h_m=h^0(X,A^{em}),
 \label{eq:residual-hilbert-function}
\end{equation}
then
\begin{equation}
 \frac{1}{h_m}\sum_{j,\gamma}
 \delta_{i^\eta_{em,j,\gamma}/m}
 \to\delta_0
 \label{eq:residual-measure-delta}
\end{equation}
weakly, and
\begin{equation}
 \sum_{j,\gamma}\left(i^\eta_{em,j,\gamma}\right)^2=o(m^7).
 \label{eq:residual-square-little-o}
\end{equation}
\end{lem}

\begin{proof}
We prove Lemma~\ref{lem:residual-quadratic-decay} in three steps.  First, the
inverse product profile gives
an integral multiplicative filtration, and its blockwise sum with $\chi$ is
$\eta$.  Second, the convex transforms cancel, so the filtered jump measures
converge to $\delta_0$.  Third, the common linear bound converts this weak
convergence into the quadratic estimate.

\medskip
\noindent\textbf{Step 1.} In this step, we prove the algebraic properties of
the residual filtration.

The negative product profile assigns the scalar jump $-w_N(m,j)$ to
$V_{em,j}$.  By \eqref{eq:product-comparator-jump}, it is additive under
multiplication:
\[
 -w_N(m+m',j+j')=-w_N(m,j)-w_N(m',j').
\]
Moreover, $0\leq j\leq em$ and
\eqref{eq:integral-affine-coefficients} give
\[
 \lvert w_N(m,j)\rvert
 \leq(\lvert\alpha_N\rvert+e\lvert\beta_N\rvert)m.
\]

Choose in each block a basis whose vectors of jump at most $\ell$ span the
$\ell$-th filtered piece of $\chi$.  Formula~\eqref{eq:residual-filtered-pieces}
shifts every basis jump in $V_{em,j}$ by $-w_N(m,j)$.  If
$u\in V_{em,j}$ and $v\in V_{em',j'}$, multiplicativity of $\chi$ and
additivity of the scalar profile give
\[
 i^\eta(uv)\leq i^\eta(u)+i^\eta(v).
\]
Thus $\eta$ is multiplicative.  The filtration $\chi$ is integral and
two-sided linearly bounded, while $w_N(m,j)$ is integral and satisfies
$\lvert w_N(m,j)\rvert
\leq(\lvert\alpha_N\rvert+e\lvert\beta_N\rvert)m$.  Hence $\eta$ is
integral and two-sided linearly bounded.

\medskip
\noindent\textbf{Step 2.} In this step, we determine the weak limit of the
residual jump measures.

The convex transform of $\chi$ is \eqref{eq:base-changed-affine-transform},
and \eqref{eq:product-comparator-jump} has precisely that affine transform.
Their difference is the zero transform.  After passing from increasing jumps
to the decreasing real extension, $\eta$ satisfies the hypotheses of
Theorem~\ref{thm:filtered-jump-measures} for the complete series of the ample
line bundle $A^e$.  The zero transform gives
\eqref{eq:residual-measure-delta}.

\medskip
\noindent\textbf{Step 3.} In this step, we deduce the quadratic decay from
\eqref{eq:residual-measure-delta}.

The two-sided linear bound places the supports of all measures in
\eqref{eq:residual-measure-delta} in one compact interval.  Choose a bounded
continuous function that agrees with $x^2$ on this interval.  Weak
convergence gives
\[
 \frac{1}{h_m}\sum_{j,\gamma}
 \left(\frac{i^\eta_{em,j,\gamma}}{m}\right)^2
 \to0.
\]
Since $X$ has dimension five, there exists a constant $c_e>0$ such that
\[
 h_m=c_e m^5+O(m^4).
\]
Multiplying the normalized quadratic sum by $m^2h_m$ proves
\eqref{eq:residual-square-little-o}.
\end{proof}

\begin{defn}[Initial filtrations oriented by the weights]
\label{defn:oriented-weight-initials}
Let $V=\bigoplus_jV_j$ carry an increasing filtration $F$.  Put
\begin{equation}\label{eq:highest-weight-initial}
 W_j^{\leq}=\bigoplus_{k\leq j}V_k,
 \qquad
 (F^{\mathrm{hi}})_\ell
 =\bigoplus_j\operatorname{pr}_j(F_\ell\cap W_j^{\leq}),
\end{equation}
and
\begin{equation}\label{eq:lowest-weight-initial}
 W_j^{\geq}=\bigoplus_{k\geq j}V_k,
 \qquad
 (F^{\mathrm{lo}})_\ell
 =\bigoplus_j\operatorname{pr}_j(F_\ell\cap W_j^{\geq}).
\end{equation}
We call these the initials obtained by taking the highest and lowest weight
components, respectively.  For an integer $q$, the $q$-oriented initial is
$F^{\mathrm{hi}}$ if $q>0$, is $F^{\mathrm{lo}}$ if $q<0$, and may be either
one if $q=0$.
\end{defn}

\begin{lem}[Smith spectrum of an oriented initial]
\label{lem:oriented-relative-smith}
Let
\begin{equation}
 V=\bigoplus_{j\in\mathbb Z}V_j
 \label{eq:oriented-weight-decomposition}
\end{equation}
be a finite-dimensional complex vector space with finitely many nonzero
summands, and let $F$ be an exhaustive separated increasing integer
filtration of $V$.  Fix integers $c,q$, and let $P$ be the split filtration
whose jump on $V_j\setminus\{0\}$ is
\begin{equation}
 p_j=c+qj.
 \label{eq:oriented-product-jump}
\end{equation}
Put $\mathcal O=\mathbb C[t]_{(t)}$ and $K=\mathbb C(t)$.  Associate to an
increasing filtration $E$ the lattice
\begin{equation}
 \Lambda_E=\sum_{\ell\in\mathbb Z}t^\ell E_\ell
 \subset V\otimes_{\mathbb C}K.
 \label{eq:oriented-filtration-lattice}
\end{equation}
Put $N=\dim_{\mathbb C}V$.  Let $r_1,\ldots,r_N$ be the relative Smith
elementary divisors of $\Lambda_P$ and $\Lambda_F$ in the convention of
Definition-Lemma~\ref{deflem:relative-smith-spectrum}, and let
$F_{\mathrm{or}}$ be the $q$-oriented initial in
Definition~\ref{defn:oriented-weight-initials}.  There exists a homogeneous basis
$v_1,\ldots,v_N$, with $v_i\in V_{j_i}$, and integers $a_i$ such that
$F_{\mathrm{or},\ell}$ is spanned by the vectors $v_i$ with
$a_i\leq\ell$.  For this basis,
\begin{equation}
 \{r_1,\ldots,r_N\}
 =\{a_i-(c+qj_i)\mid 1\leq i\leq N\}
 \label{eq:oriented-smith-multiset}
\end{equation}
as multisets.  In particular,
\begin{equation}
 \sum_{i=1}^N r_i^2
 =\sum_{i=1}^N(a_i-c-qj_i)^2.
 \label{eq:oriented-smith-square}
\end{equation}
\end{lem}

\begin{proof}
We first construct a basis adapted to the filtration and to the weight flag.
We then gauge the product lattice to the standard lattice and read the Smith
integers from the gauged filtered basis.  It is enough to prove the case
$q>0$.  The case $q<0$ follows after replacing every weight $j$ by $-j$,
and the case $q=0$ follows from
$\Lambda_P=t^c(V\otimes\mathcal O)$.

\medskip
\noindent\textbf{Step 1.} In this step, we construct a homogeneous basis
for the initial obtained by taking the highest weight components.

Put
\[
 W_j=\bigoplus_{k\leq j}V_k.
\]
For each pair $(\ell,j)$, choose a complement $C_{\ell,j}$ of
\[
 (F_{\ell-1}\cap W_j)+(F_\ell\cap W_{j-1})
\]
inside $F_\ell\cap W_j$, omitting repeated filtration levels.  Choose a
basis of each complement.  Induction on the finite ordered grid of pairs
$(\ell,j)$ shows that the vectors chosen up to $(\ell,j)$ span
$F_\ell\cap W_j$.  Their union is therefore a basis
$f_1,\ldots,f_N$ of $V$.

For each $i$, let $a_i$ and $j_i$ be the indices of the complement
containing $f_i$.  Then $a_i$ is the $F$-jump of $f_i$, the vector $f_i$
belongs to $W_{j_i}$, and its image
\[
 v_i=\operatorname{pr}_{j_i}(f_i)\in V_{j_i}
\]
is nonzero.  For fixed $\ell,j$, the vectors $v_i$ with
$a_i\leq\ell$ and $j_i=j$ form a basis of
\[
 \operatorname{pr}_j(F_\ell\cap W_j)\subset V_j.
\]
Thus $v_1,\ldots,v_N$ is a homogeneous basis for $F^{\mathrm{hi}}$, and the
jump of $v_i$ is $a_i$.

\medskip
\noindent\textbf{Step 2.} In this step, we diagonalize the relative lattice
after gauging the product filtration.

Define a $K$-linear automorphism $D$ by
\[
 D\vert_{V_j}=t^{-(c+qj)}\operatorname{id}_{V_j}.
\]
Then $D\Lambda_P=V\otimes\mathcal O$.  Write
\[
 f_i=v_i+\sum_{k<j_i}f_{i,k},
 \qquad
 f_{i,k}\in V_k.
\]
We have
\[
 D(t^{a_i}f_i)
 =t^{a_i-c-qj_i}
 \left(v_i+\sum_{k<j_i}t^{q(j_i-k)}f_{i,k}\right).
\]
Put
\[
 u_i(t)=v_i+\sum_{k<j_i}t^{q(j_i-k)}f_{i,k}.
\]
Relative to the complex basis $v_1,\ldots,v_N$, the matrix with columns
$u_i(t)$ has constant term equal to the identity.  It belongs to
$\operatorname{GL}_N(\mathcal O)$.  Since the basis
$f_1,\ldots,f_N$ is adapted to $F$, we obtain
\[
 D\Lambda_F
 =\bigoplus_i\mathcal O\,t^{a_i-c-qj_i}u_i(t).
\]
The relative Smith spectrum is invariant under the common gauge
transformation $D$.  In the $\mathcal O$-basis
$u_1(t),\ldots,u_N(t)$, the two lattices are diagonal with exponents
$a_i-c-qj_i$.  This proves \eqref{eq:oriented-smith-multiset} and
\eqref{eq:oriented-smith-square}.
\end{proof}

\begin{rem}[Dependence on the orientation]
\label{rem:smith-orientation-warning}
Let $V=\mathbb C e_0\oplus\mathbb C e_1$, with comparator jumps $0$ on
$\mathbb C e_0$ and $1$ on $\mathbb C e_1$.  Let $F$ be given by
\begin{equation}
 F_{-1}=0,
 \qquad
 F_0=\mathbb C(e_0+e_1),
 \qquad
 F_1=V.
 \label{eq:orientation-example-filtration}
\end{equation}
The filtration $F^{\mathrm{lo}}$ in \eqref{eq:lowest-weight-initial} has
residual jumps $\{0,0\}$.  The relative
lattices are
\begin{equation}
 \Lambda_F=\mathcal O(e_0+e_1)+\mathcal O\,te_1,
 \qquad
 \Lambda_P=\mathcal O\,e_0+\mathcal O\,te_1.
 \label{eq:orientation-example-lattices}
\end{equation}
In the $\mathcal O$-basis $(e_0,te_1)$ of $\Lambda_P$, the transition matrix
is
\begin{equation}
 \begin{pmatrix}1&0\\ t^{-1}&1\end{pmatrix}.
 \label{eq:orientation-example-matrix}
\end{equation}
Its minimum entry valuation is $-1$ and its determinant has valuation $0$,
so the relative Smith multiset is $\{-1,1\}$ and its square sum is $2$.
Thus the opposite orientation does not compute the relative Smith spectrum
in this example, which explains the sign convention in
Lemma~\ref{lem:oriented-relative-smith}.
\end{rem}

\subsection{Descent from the integral comparator}
\label{subsec:descent-integral-comparator}

In this subsection, we eliminate the rounding introduced by deck descent and
prove Theorem~\ref{thm:equality-classification}.

\begin{lem}[Elimination of the ceiling error]
\label{lem:eliminate-ceiling-errors}
Let $e$ be a positive integer and let $a,b\in\mathbb Q$.  For every positive
integer $m$ and $0\leq j\leq em$, put
\begin{equation}
 g_{em,j}=\lceil aem+bj\rceil-(aem+bj).
 \label{eq:ceiling-errors}
\end{equation}
Assume that there exists a constant $C$ such that
\begin{equation}
 \sum_{j=0}^{em}g_{em,j}\leq C
 \label{eq:ceiling-row-bound}
\end{equation}
for all sufficiently large $m$.  Then
\begin{equation}
 b\in\mathbb Z,
 \qquad
 ea\in\mathbb Z.
 \label{eq:ceiling-integrality}
\end{equation}
Consequently,
\begin{equation}
 \lceil aem+bj\rceil=aem+bj
 \label{eq:ceiling-errors-vanish}
\end{equation}
for every positive integer $m$ and every $0\leq j\leq em$.
\end{lem}

\begin{proof}
Write $b=u/q$ in lowest terms, with $q>0$.  If $q>1$, then on any $q$
consecutive values of $j$, the fractional parts of $aem+bj$ form a translate
of the $q$-grid.  The sum of their ceiling errors is at least
\[
 \frac{q-1}{2}.
\]
The left-hand side of \eqref{eq:ceiling-row-bound} would grow linearly with
$m$, which is a contradiction.  Thus $q=1$, and $b\in\mathbb Z$.

The ceiling error is now independent of $j$.  Set
\[
 \theta_m=\lceil eam\rceil-eam.
\]
The bound in \eqref{eq:ceiling-row-bound} gives
\[
 (em+1)\theta_m\leq C,
\]
so $\theta_m\to0$.  Since $ea\in\mathbb Q$, the sequence $\theta_m$ takes
values in a finite set.  It follows that $\theta_m=0$ for all sufficiently
large $m$.  Two consecutive such integers $m$ show that $ea\in\mathbb Z$.
Substitution proves \eqref{eq:ceiling-errors-vanish}.
\end{proof}

\begin{proof}[Proof of Theorem~\ref{thm:equality-classification}]
The nonnegativity assertion is
Proposition~\ref{prop:semistability-all-exponents}.  Assume that
$\DF(\mathcal T)=0$.  Corollary~\ref{cor:rational-affine-data} gives the
rational coefficients $a,b$ of the common affine transform.
Proposition~\ref{prop:integral-product-comparator} gives a positive integer
$N$, the normalized base change $\mathcal T_N$, and the integral product
comparator $\mathcal P_N$.  They satisfy
\eqref{eq:integral-affine-coefficients},
\eqref{eq:product-comparator-jump}, and
\eqref{eq:base-changed-affine-transform}, and
$\DF(\mathcal T_N)=0$.

We prove the equality assertion in four steps.  First, the oriented initial
of Definition~\ref{defn:oriented-weight-initials}
identifies the relative Smith spectrum of $\mathcal T_N$ and $\mathcal P_N$
with the residual jumps.  Second, quadratic decay and marked rigidity identify
the normalized base change with the product comparator.  Third, deck
invariants descend the equality of the marked extension lattices to a
filtration that is scalar on each block and has ceiling jumps.  Finally, the uniform bound in
\eqref{eq:affine-initials-row-budget} removes the ceiling error.

\medskip
\noindent\textbf{Step 1.} In this step, we prove that the relative Smith
spectrum of $\mathcal T_N$ and $\mathcal P_N$ has quadratic sum $o(m^7)$.

Choose $F^{\mathrm{hi}}$ from \eqref{eq:highest-weight-initial} if
$\beta_N>0$, choose $F^{\mathrm{lo}}$ from
\eqref{eq:lowest-weight-initial} if $\beta_N<0$, and choose either filtration
if $\beta_N=0$.  Denote its jumps on $V_{em,j}$ by
\[
 i^s_{em,j,\gamma}.
\]
Let $r_{m,\gamma}$ be the relative Smith integers of the marked extension
lattices of $\mathcal P_N$ and $\mathcal T_N$.  In section degree $m$,
Lemma~\ref{lem:oriented-relative-smith}, with $c=\alpha_N m$ and
$q=\beta_N$, gives
\begin{equation}
 \sum_\gamma r_{m,\gamma}^2
 =\sum_{j,\gamma}
 \left(i^s_{em,j,\gamma}-w_N(m,j)\right)^2.
 \label{eq:actual-smith-residual}
\end{equation}

The projected initial just chosen is the Grassmannian initial used in
Set-up~\ref{setup:initial-filtrations}.  To prove this, fix $m$, let $F$ be the
degree-$m$ filtration of $\mathcal T_N$, and let $g_u$ act on $V_{em,j}$
by $u^j\operatorname{id}$.  For the highest-weight flag, choose by
successive complements on the finite grid of pairs $(\ell,j)$ a basis
$f_i$ simultaneously adapted to $F$ and to the flag $W_j^{\leq}$.  Write
$f_i=v_i+\sum_{k<j_i}f_{i,k}$, where
$0\neq v_i\in V_{em,j_i}$ and $f_{i,k}\in V_{em,k}$.  If $a_i$ is the
$F$-jump of $f_i$, then the $f_i$ with $a_i\leq\ell$ span $F_\ell$,
whereas the corresponding $v_i$ span $(F^{\mathrm{hi}})_\ell$.  The
normalized frames
$u^{-j_i}g_u(f_i)=v_i+\sum_{k<j_i}u^{k-j_i}f_{i,k}$ converge to the
linearly independent vectors $v_i$ as $u\to\infty$.  Hence
$g_u(F_\ell)\to(F^{\mathrm{hi}})_\ell$ in the Grassmannian.  Replacing
every weight $j$ by $-j$ gives
$g_u(F_\ell)\to(F^{\mathrm{lo}})_\ell$ as $u\to0$.  Since
$\lambda_-(z)=g_{z^{-1}}$ and $\lambda_+(z)=g_z$ in
\eqref{eq:opposite-fiber-scaling-subgroups}, this proves
$\chi^-=F^{\mathrm{hi}}$ and $\chi^+=F^{\mathrm{lo}}$.  Thus the selected
initial is $\chi^-$ when $\beta_N>0$, is $\chi^+$ when $\beta_N<0$, and
may be either one when $\beta_N=0$.

Lemma~\ref{lem:affine-initials-veronese}, applied to $\mathcal T_N$, now
shows that the selected initial is block preserving, integral,
multiplicative, and two-sided linearly bounded
on $R^{[e]}$.  By Proposition~\ref{prop:integral-product-comparator}, its
convex transform is \eqref{eq:base-changed-affine-transform}.  Hence
Lemma~\ref{lem:residual-quadratic-decay} applies and gives
\[
 \sum_{j,\gamma}
 \left(i^s_{em,j,\gamma}-w_N(m,j)\right)^2=o(m^7).
\]
Together with \eqref{eq:actual-smith-residual}, this proves
\begin{equation}
 \sum_\gamma r_{m,\gamma}^2=o(m^7).
 \label{eq:actual-smith-little-o}
\end{equation}

\medskip
\noindent\textbf{Step 2.} In this step, we identify the normalized base
change with the product comparator.

The test configurations $\mathcal T_N$ and $\mathcal P_N$ carry a common
generic marking.  On $\mathcal T_N$, this is the base change of the original
marking under $t=(t')^N$.  The pulled-back polarization has its canonical
deck linearization: the deck group acts trivially on the marked generic
line-bundle factor and sends $t'$ to $\zeta t'$.  This action lifts
functorially through normalization.  The product comparator is placed in
the common marked Laurent algebra with this linearization.  Twisting by an
additional $\mu_N$-character would define a different descent datum and is
not part of either marked base change.

Choose a common Veronese in section degree as in
Lemma~\ref{lem:normal-veronese-section-algebra}.  The estimate
\eqref{eq:actual-smith-little-o} remains valid after this fixed rescaling of
the section degree.  Since $\dim X=5$, the exponent in
\eqref{eq:smith-rigidity-hypothesis} is $5+2=7$.  Therefore
Proposition~\ref{prop:marked-smith-rigidity} gives a marked polarized
equivariant isomorphism
\[
 \mathcal T_N\simeq\mathcal P_N.
\]
Equivalently, if $\Lambda_{\mathcal T_N,m}$ is the marked degree-$m$
extension lattice, then
\begin{equation}
 \Lambda_{\mathcal T_N,m}
 =\bigoplus_{j=0}^{em}
 (t')^{\alpha_N m+\beta_N j}
 V_{em,j}\otimes_{\mathbb C}\mathbb C[t']
 \label{eq:base-changed-product-lattice}
\end{equation}
inside the marked Laurent extension.
The equality in \eqref{eq:base-changed-product-lattice} is equality of
submodules of that fixed ambient algebra, not an abstract isomorphism chosen
afterward.  The deck action on its right-hand side is consequently the
restriction of the canonical ambient action.

\medskip
\noindent\textbf{Step 3.} In this step, we descend
\eqref{eq:base-changed-product-lattice} by taking deck invariants.

The deck group $\mu_N$ acts on the normalized base change.  On an affine
normal chart $\Spec R$ of $\mathcal T$, the ordinary
base-change ring is
\[
 S=R[t']/\left((t')^N-t\right).
\]
If $\overline S$ is its normalization, then
\begin{equation}
 \overline S^{\mu_N}=R.
 \label{eq:deck-invariants-normalization}
\end{equation}
Proposition~\ref{prop:integral-product-comparator} shows that $S$ is a
domain.  Localizing this domain at $R\setminus\{0\}$ gives
\[
 \operatorname{Frac}(R)[X]/(X^N-t),
\]
which is again a domain; hence $X^N-t$ is irreducible over
$\operatorname{Frac}(R)$.  Monicity gives the free $R$-basis
$1,t',\ldots,(t')^{N-1}$.  Thus the extension in
\eqref{eq:deck-fixed-field} has degree $N$, and its displayed $\mu_N$-action
is the full Galois group.
The fraction fields satisfy
\begin{equation}\label{eq:deck-fixed-field}
 \operatorname{Frac}(\overline S)=\operatorname{Frac}(R)(t'),
 \qquad (t')^N=t,
 \qquad
 \operatorname{Frac}(\overline S)^{\mu_N}=\operatorname{Frac}(R).
\end{equation}
An invariant element therefore belongs to $\operatorname{Frac}(R)$ and is integral over
$R$.  Normality of $R$ places it in $R$.  Every element of $R$ maps to a
$\mu_N$-invariant element of $\overline S$, which gives the reverse inclusion.
Thus $\mathcal T_N/\mu_N=\mathcal T$.

Write $\mathcal T=(\mathcal Y,\mathcal H)$ and
$\mathcal T_N=(\mathcal Y_N,\mathcal H_N)$, and let
$q\colon\mathcal Y_N\to\mathcal Y$ be the quotient morphism.  The canonical
base-change linearization is carried by
$\mathcal H_N=q^*\mathcal H$.  The projection formula and
\eqref{eq:deck-invariants-normalization} give
\begin{equation}\label{eq:deck-polarization-descent}
 \left(q_*\mathcal H_N^m\right)^{\mu_N}
 =\mathcal H^m\otimes
   \left(q_*\mathcal O_{\mathcal Y_N}\right)^{\mu_N}
 =\mathcal H^m.
\end{equation}
Taking global sections in \eqref{eq:deck-polarization-descent} identifies
the downstairs degree-$m$ section lattice with the invariant part of the
upstairs degree-$m$ lattice.

Under the pulled-back generic marking, $\mu_N$ acts trivially on every
$V_{em,j}$ and sends $t'$ to $\zeta t'$ for $\zeta\in\mu_N$.  This marking is
the pullback of the original marking under $t=(t')^N$, so the deck group acts
only on the parameter $t'$; explicitly,
\begin{equation}\label{eq:deck-action-on-marked-ambient}
 s\otimes(t')^k\longmapsto
 s\otimes\zeta^k(t')^k.
\end{equation}
For every integer $w$, the invariant block is
\begin{equation}\label{eq:deck-invariant-block}
 \left((t')^wV_{em,j}\otimes_{\mathbb C}\mathbb C[t']\right)^{\mu_N}
 =t^{\lceil w/N\rceil}V_{em,j}\otimes_{\mathbb C}\mathbb C[t].
\end{equation}
A section downstairs pulls back to an invariant section upstairs, and
$\mathcal T$-lattice is the invariant lattice by
\eqref{eq:deck-polarization-descent}.  Deck invariance requires the
$t'$-exponent to be divisible by $N$.  On $V_{em,j}$, the least admissible
invariant exponent is
\[
 \min\{\ell\in N\mathbb Z\mid
       \ell\geq\alpha_Nm+\beta_Nj\}
 =
 N\left\lceil\frac{\alpha_N m+\beta_N j}{N}\right\rceil
 =N\lceil eam+bj\rceil.
\]
This identity is valid without a sign restriction on
$\alpha_Nm+\beta_Nj$.
Consequently, the increasing jump of the original filtration on
$V_{em,j}$ is
\begin{equation}
 \lceil eam+bj\rceil.
 \label{eq:descended-ceiling-jump}
\end{equation}

\medskip
\noindent\textbf{Step 4.} In this step, we remove the ceiling error and
identify the original test configuration with a product.

The filtration in \eqref{eq:descended-ceiling-jump} is scalar on every block
$V_{em,j}$.  Fiber scaling acts by a scalar on each block and
therefore preserves every filtered subspace.  Both opposite initials equal
the descended filtration.  In the actual degree $d=em$, its normalized jump
is
\[
 \frac{\lceil ad+bj\rceil}{d}.
\]
The difference from $a+b(j/d)$ is less than $1/d$, so the convex profile is
$a+bx$.  By \eqref{eq:affine-initials-row-gap}, the block-average gap for
the descended filtration is exactly
\begin{equation}
 g_{d,j}=\lceil ad+bj\rceil-(ad+bj).
 \label{eq:ceiling-row-gap}
\end{equation}
Lemma~\ref{lem:affine-initials-row-budget}, applied to this zero-invariant
initial, and \eqref{eq:affine-initials-row-budget} give a constant
$C_{\mathcal T}$ such that
\[
 \sum_{j=0}^d g_{d,j}\leq C_{\mathcal T}
\]
for all sufficiently large supported degrees $d=em$.

Lemma~\ref{lem:eliminate-ceiling-errors} now gives
\[
 \alpha:=ea\in\mathbb Z,
 \qquad
 \beta:=b\in\mathbb Z.
\]
Substitution in \eqref{eq:descended-ceiling-jump} gives the jump
$\alpha m+\beta j$ in \eqref{eq:final-affine-jump} on every block
$V_{em,j}$.  The marked section algebra is the product Rees algebra in which
fiber scaling contributes $\beta j$ and the scalar character contributes
$\alpha m$ to the increasing entry on $V_{em,j}$, in the convention of
Definition~\ref{defn:test-configuration}.  Relative $\operatorname{Proj}$ recovers $\mathcal T$ as the
polarized product test configuration described in
Theorem~\ref{thm:equality-classification}.
\end{proof}
 
\section{Proof of Theorem A}\label{sec:main-proof}

In this section, we prove Theorem~\ref{thm:main} and then derive
Corollaries~\ref{cor:scheme-strengthening} and~\ref{cor:literal-donaldson}.

\begin{proof}[Proof of Theorem~\ref{thm:main}]
We verify the three assertions in Theorem~\ref{thm:main} in the order in which
they are stated.  We retain the notation of
Construction~\ref{cons:explicit-fivefold} throughout the proof.
Lemma~\ref{lem:orthogonal-curves} gives curves with the genera in
\eqref{eq:intro-genera} and the vanishing in
\eqref{eq:intro-hom-orthogonality}.  With these curves,
Construction~\ref{cons:explicit-fivefold} defines $(X,A)$ from the degrees
in \eqref{eq:intro-degrees}.  Proposition~\ref{prop:basic-geometry} shows that
$X$ is a smooth projective fivefold, that $A$ is ample, and that
$\Aut^0(X)=\mathbb C^*$, where the action is fiber scaling.

We next prove K-polystability at every positive exponent.
Fix a positive integer $e$, and let $\mathcal T$ be a normal ample algebraic
test configuration whose generic polarized fiber is $(X,A^e)$.
Proposition~\ref{prop:semistability-all-exponents} gives
\begin{equation}\label{eq:main-proof-semistability}
 \DF(\mathcal T)\geq 0.
\end{equation}
If equality holds, Theorem~\ref{thm:equality-classification} identifies
$\mathcal T$ with the polarized product induced by an integral
one-parameter subgroup of fiber scaling together with a scalar character on
the polarization.  Since $e$ and $\mathcal T$ were arbitrary, this proves
the second assertion of Theorem~\ref{thm:main}.

Finally, we prove the analytic nonexistence assertion.
Proposition~\ref{prop:no-extremal-metric} shows that $c_1(A)$ contains no
extremal K\"ahler metric.  In particular, it contains no cscK metric.  Thus
$(X,A)$ satisfies the hypothesis of
Conjecture~\ref{conj:ytd-sufficiency} but not its conclusion.  This completes
the proof of Theorem~\ref{thm:main}.
\end{proof}

\subsection{Scheme-theoretic consequences}
\label{subsec:scheme-theoretic-consequences}

This subsection proves Corollary~\ref{cor:scheme-strengthening} for arbitrary
ample scheme-theoretic test configurations and then proves
Corollary~\ref{cor:literal-donaldson} for the unrepaired equality convention.

\begin{proof}[Proof of Corollary~\ref{cor:scheme-strengthening}]
Let $\nu\colon\mathcal Y^\nu\to\mathcal Y$ be the normalization, and put
$\mathcal H^\nu=\nu^*\mathcal H$.  We first verify that this is a normal
ample test configuration, then use the normalization comparison, and
finally determine the equality case.

\medskip
\noindent\textbf{Step 1.} In this step, we show that
$(\mathcal Y^\nu,\mathcal H^\nu)$ is a normal ample test configuration for
$(X,A^e)$.
The generic polarized fiber is the marked product with the smooth connected
variety $X$, and is therefore integral.  If
$\Spec R\subseteq\mathcal Y$ is
affine, flatness over $\mathbb C[t]$ makes the homomorphism
\begin{equation}\label{eq:scheme-localization-map}
 R\to R[t^{-1}]
\end{equation}
injective.  The target is the coordinate ring of an affine open subset of
$X\times\mathbb G_m$ and is a domain.  Hence $R$ is a domain, so
$\mathcal Y$ is integral.

The normalization is finite and inherits the $\mathbb C^*$-action and the
marking over $\mathbb G_m$.  Its coordinate rings are torsion free over the
principal ideal domain $\mathbb C[t]$, and hence the normalization remains
flat over $\mathbb A^1$.  Finite pullback preserves relative ampleness.
Consequently, $(\mathcal Y^\nu,\mathcal H^\nu)$ is a normal ample algebraic
test configuration for $(X,A^e)$.

\medskip
\noindent\textbf{Step 2.} In this step, we prove the nonnegativity of the
Donaldson--Futaki invariant.
Put
\begin{equation}\label{eq:scheme-normalization-quotient}
 \mathcal Q=
 \nu_*\mathcal O_{\mathcal Y^\nu}/\mathcal O_{\mathcal Y}.
\end{equation}
Lemma~\ref{lem:normalization-comparison} gives a number $c\geq0$ and the
identity
\begin{equation}\label{eq:scheme-normalization-comparison}
 \DF(\mathcal Y,\mathcal H)
 =\DF(\mathcal Y^\nu,\mathcal H^\nu)+\frac{2c}{A_0},
 \qquad A_0>0.
\end{equation}
Theorem~\ref{thm:equality-classification} gives
\begin{equation}\label{eq:scheme-normalized-nonnegative}
 \DF(\mathcal Y^\nu,\mathcal H^\nu)\geq0.
\end{equation}
Combining \eqref{eq:scheme-normalization-comparison} and
\eqref{eq:scheme-normalized-nonnegative} proves
\eqref{eq:intro-scheme-semistability}.

\medskip
\noindent\textbf{Step 3.} We determine the equality case in this step.
If $\DF(\mathcal Y,\mathcal H)=0$, then both terms on the right-hand side
of \eqref{eq:scheme-normalization-comparison} vanish.
Theorem~\ref{thm:equality-classification} identifies the normalization with
the polarized product induced by an integral one-parameter subgroup of
fiber scaling together with a scalar character.  Since $c=0$,
Lemma~\ref{lem:normalization-comparison} shows that $\mathcal Q$ is supported
in total-space codimension at least two.  The finite morphism $\nu$ is an
isomorphism away from this support.

Conversely, suppose that the normalization is the asserted polarized product
and that $\nu$ is an isomorphism away from a closed subset of total-space
codimension at least two.  The product has zero Donaldson--Futaki invariant: inverting both the
one-parameter subgroup and the scalar character yields another normal
ample product configuration whose invariant is the negative of the
first, and Proposition~\ref{prop:semistability-all-exponents} makes
both nonnegative.
The support condition gives $c=0$ in
\eqref{eq:scheme-normalization-comparison}.  Hence
$\DF(\mathcal Y,\mathcal H)=0$.  This completes the proof of
Corollary~\ref{cor:scheme-strengthening}.
\end{proof}

\begin{proof}[Proof of Corollary~\ref{cor:literal-donaldson}]
We construct an explicit exponent-six test configuration of
$(\mathbb P^1,\mathcal O_{\mathbb P^1}(1))$, compute its
Donaldson--Futaki invariant, and compare it with the Fubini--Study metric.

\medskip
\noindent\textbf{Step 1.} In this step, we construct a nonnormal
scheme-theoretic test configuration with trivial normalization.
Inside $\mathsf S=\mathbb C[t,x,y]$, let
\begin{equation}\label{eq:literal-defect-algebra}
 \mathsf A
 =\mathbb C[t,y,tx,x^2,x^3]
 =\mathbb C[t,y,u,a,b],
 \qquad
 u=tx,\quad a=x^2,\quad b=x^3.
\end{equation}
Give these generators the grading
\begin{equation}\label{eq:literal-defect-grading}
 \deg(t)=0,
 \qquad
 \deg(y)=\deg(u)=1,
 \qquad
 \deg(a)=2,
 \qquad
 \deg(b)=3.
\end{equation}
The algebra $\mathsf A$ is a finitely generated domain, and its degree-$m$
piece is the free $\mathbb C[t]$-module
\begin{equation}\label{eq:literal-defect-free-basis}
 \mathsf A_m
 =\mathbb C[t]y^m
 \oplus t\mathbb C[t]xy^{m-1}
 \oplus\bigoplus_{j=2}^m\mathbb C[t]x^jy^{m-j}.
\end{equation}
Thus
\begin{equation}\label{eq:literal-relative-proj}
 \mathcal Z=\operatorname{Proj}_{\mathbb C[t]}(\mathsf A),
 \qquad
 \mathcal L=\mathcal O_{\mathcal Z}(6)
\end{equation}
is flat over $\mathbb A^1$.  The standard presentation of the sixth
Veronese is generated in degree one, so $\mathcal L$ is a relatively ample
line bundle.  It follows that $(\mathcal Z,\mathcal L)$ is an algebraic test
configuration for
$(\mathbb P^1,\mathcal O_{\mathbb P^1}(6))$, and hence is an exponent-six
test configuration for $(\mathbb P^1,\mathcal O_{\mathbb P^1}(1))$.

The fraction fields of $\mathsf A$ and $\mathsf S$ agree because
$x=b/a$, and $x$ is integral over $\mathsf A$ because it satisfies
$X^2-a=0$.  Since $\mathsf S$ is integrally closed, it is the normalization
of $\mathsf A$.  Therefore the normalization of $\mathcal Z$ is the product
$\mathbb P^1\times\mathbb A^1$.  After inverting $t$, the identity
$x=u/t$ gives $\mathsf A[t^{-1}]=\mathsf S[t^{-1}]$, while after inverting
$a$, the identity $x=b/a$ gives
$\mathsf A[a^{-1}]=\mathsf S[a^{-1}]$.  The normalization quotient is
therefore supported on $V(t,a)$.  The relations $b^2=a^3$ and
$u^2=t^2a$ show that its support is the codimension-two locus
$V(t,a,b,u)$.  In relative Proj this is the single point defined by
\begin{equation}\label{eq:literal-defect-point}
 (t,a,b,u).
\end{equation}
Thus the normalization is an isomorphism away from that point.

The central fiber has the presentation
\begin{equation}\label{eq:literal-central-fiber}
 \mathsf A/t\mathsf A
 =\mathbb C[y,u,a,b]/(u^2,ua,ub,b^2-a^3).
\end{equation}
The basis \eqref{eq:literal-defect-free-basis} shows that the class of $u$
is nonzero.  Its annihilator in \eqref{eq:literal-central-fiber} is
$(u,a,b)$, whereas the minimal prime is $(u)$.  Hence $(u,a,b)$ is an
embedded associated prime.  In particular, the central fiber is nonreduced
and $\mathcal Z$ is not the product test configuration.

\medskip
\noindent\textbf{Step 2.} In this step, we compute the
Donaldson--Futaki invariant of $\mathcal Z$.
Let $\mathbb C^*$ act with weight one on $t$ and $u$, and with weight zero
on $x,y,a,b$.  This action covers the standard action on $\mathbb A^1$ and
becomes the trivial product action on the normalization.  In polarization
degree $k$, set $m=6k$.  The basis in
\eqref{eq:literal-defect-free-basis} gives
\begin{equation}\label{eq:literal-hilbert-weight}
 h(k)=6k+1,
 \qquad
 w(k)=1.
\end{equation}
Thus both functions in \eqref{eq:literal-hilbert-weight} are computed from
the free pieces of the sixth Veronese, which is the section algebra of the
chosen polarization $\mathcal L$.
Thus the leading and subleading coefficients of the total weight are both
zero, and
\begin{equation}\label{eq:literal-zero-df}
 \DF(\mathcal Z,\mathcal L)=0.
\end{equation}

\medskip
\noindent\textbf{Step 3.} We compare the test configuration with the cscK
metric and conclude the proof in this step.
The Fubini--Study metric is a cscK metric in
$c_1(\mathcal O_{\mathbb P^1}(1))$.  On the other hand,
\eqref{eq:literal-zero-df} holds and $\mathcal Z$ is not a product because
its central fiber in \eqref{eq:literal-central-fiber} is nonreduced.  These
facts violate the unmodified equality condition requiring an actual product.
This is the codimension-two phenomenon isolated by Stoppa as being trivial
in codimension two; see
\cite[p.~1 and Definition~1]{Sto11}.  Hence the literal
scheme-theoretic equality convention is not necessary for cscK existence.
This completes the proof of Corollary~\ref{cor:literal-donaldson}.
\end{proof}

\section{Vanishing of the reduced Donaldson--Futaki quotient}
\label{sec:uniform-threshold}

The goal of this section is to construct the test configurations in
Theorem~\ref{thm:uniform-threshold} and compute their reduced
Donaldson--Futaki quotient.

\begin{lem}[Sufficiently divisible Veronese gradings]
\label{lem:sufficiently-divisible-veronese}
Let $S=\bigoplus_{m\geq0}S_m$ be a finitely generated graded algebra over
the Noetherian ring $S_0$.  There exists a positive integer $r_0$ such that,
for every positive multiple $r$ of $r_0$, the Veronese algebra
\begin{equation}\label{eq:sufficiently-divisible-veronese}
 S^{(r)}=\bigoplus_{k\geq0}S_{rk}
\end{equation}
is generated over $S_0$ by $S_r$.
\end{lem}

\begin{proof}
Choose homogeneous algebra generators $x_1,\ldots,x_h$ of positive degrees
$d_1,\ldots,d_h$, and let $d$ be a common multiple of the $d_i$.  The
algebra $C=S^{(d)}$, with its rescaled grading, is finite over the algebra
\begin{equation}\label{eq:veronese-standard-subalgebra}
 B=S_0[x_1^{d/d_1},\ldots,x_h^{d/d_h}],
\end{equation}
which is generated in degree one.  Choose homogeneous $B$-module generators
$z_1,\ldots,z_s$ of $C$, of degrees $e_1,\ldots,e_s$, and choose
$n\geq\max_i e_i$.

Let $k\geq2$ and write an element of $C_{kn}$ as
\begin{equation}\label{eq:veronese-module-expansion}
 c=\sum_i b_i z_i,
 \qquad b_i\in B_{kn-e_i}.
\end{equation}
Since $B$ is generated in degree one, multiplication from
$B_{n-e_i}\otimes B_{(k-1)n}$ onto $B_{kn-e_i}$ is surjective.  Each term
in \eqref{eq:veronese-module-expansion} is therefore a sum of products of
an element of $C_n$ and an element of $C_{(k-1)n}$.  Induction on $k$ shows
that $C^{(n)}$ is generated in degree one.  Put $r_0=dn$.  If $r$ is a
positive multiple of $r_0$, then $S^{(r)}$ is a Veronese of the
degree-one-generated algebra $S^{(r_0)}$ and is again generated in degree
one.
\end{proof}

\begin{thm}[Openness of relative ampleness,
{\cite[Tag~0D2N]{Sta26}}]
\label{thm:relative-ampleness-near-fiber}
Let $f\colon Y\to S$ be a proper morphism of schemes with $S$ Noetherian,
let $H$ be a line bundle on $Y$, and let $s\in S$.  If $H\vert_{Y_s}$ is
ample, then there exists an open neighborhood $U$ of $s$ such that
\begin{equation}\label{eq:relative-ampleness-near-fiber}
 H\vert_{f^{-1}(U)}\text{ is ample relative to }U.
\end{equation}
\end{thm}

\begin{proof}
This is \cite[Tag~0D2N]{Sta26}.
\end{proof}

\begin{thm}[Relative ampleness and projectivity,
{\cite[Tags~01VJ and~0B45]{Sta26}}]
\label{thm:relative-ampleness-and-projectivity}
Let $f\colon Y\to S$ be a quasi-compact morphism, and let $H$ be a line
bundle on $Y$.  If $S$ has an affine open covering $S=\bigcup_iU_i$ such
that $H\vert_{f^{-1}(U_i)}$ is ample for every $i$, then $H$ is
$f$-ample.  If, in addition, $S$ has an ample line bundle and $f$ is proper,
then $f$ is projective.
\end{thm}

\begin{proof}
The first assertion is \cite[Tag~01VJ]{Sta26}, and the second is the
implication from condition (6) to condition (1) in
\cite[Tag~0B45]{Sta26}.
\end{proof}

\begin{prop}[Rational single creases]\label{prop:rational-single-crease}
Let $\mu=c/q\in\mathbb Q\cap(0,1)$, where $c$ and $q$ are coprime positive
integers, and let $d$ be a positive integer such that
$d\mu\in\mathbb Z$.  Put
\begin{equation}\label{eq:rational-crease-function}
 f_\mu(x)=(x-\mu)_+.
\end{equation}
For every sufficiently divisible positive integer $r$, the saturated
epigraph of $df_\mu$, with its $r$-th Veronese polarization in section degree,
defines a normal nonproduct $T$-equivariant algebraic test configuration
$\mathcal T(d,\mu;r)$ of exponent $r$ for $(X,A)$.  For every $k\geq1$,
its increasing filtration entry on
$V_{rk,j}=H^0(B,M^{rk}\otimes L^j)$ is
\begin{equation}\label{eq:rational-crease-weight}
 drk f_\mu(j/(rk))=\max\{0,dj-d\mu rk\},
 \qquad 0\leq j\leq rk.
\end{equation}
Its Donaldson--Futaki invariant is
\begin{equation}\label{eq:rational-crease-df}
 \DF\bigl(\mathcal T(d,\mu;r)\bigr)=\frac{dF(\mu)}{a_0}.
\end{equation}
The limiting distribution of the normalized fiber coordinate $j/(rk)$ is the
probability measure
\begin{equation}\label{eq:rational-crease-density}
 d\nu(x)=\frac{p(x)}{a_0}\,dx
 \qquad\text{on }[0,1].
\end{equation}
The product test configuration generated by fiber scaling satisfies
\begin{equation}\label{eq:fiber-product-df}
 \DF_{\mathrm{fib}}=-\frac{F(1)}{a_0}=0.
\end{equation}
\end{prop}

\begin{proof}
We first construct the test configuration from its saturated epigraph and
then compute its invariant from the exact dimensions of the blocks in large
degree.  We
use the notation of Set-up~\ref{setup:boundary-data}.

\medskip
\noindent\textbf{Step 1.} In this step, we construct
$\mathcal T(d,\mu;r)$ and prove its normality and nonproductness.
The function $df_\mu$ has integral slopes and intercepts.  Put
\begin{equation}\label{eq:rational-crease-homogeneous-profile}
 g(m,j)=dmf_\mu(j/m)
 \qquad(m\geq1,\ 0\leq j\leq m),
\end{equation}
and put $g(0,0)=0$.  This is the positively homogeneous extension of
$df_\mu$.  Its epigraph semigroup is
\begin{equation}\label{eq:rational-crease-semigroup}
 \Sigma=\{(m,j,\ell)\in\mathbb Z^3\mid
 m\geq0,\ 0\leq j\leq m,\ \ell\geq g(m,j)\}.
\end{equation}
Triangulate the rational polyhedral cone given by the epigraph into rational simplicial
cones and choose an integral generator on each ray.  Every lattice point in
one of these cones is the sum of a lattice point in the bounded fundamental
parallelepiped and nonnegative integral multiples of its ray generators.
Only finitely many lattice points lie in these parallelepipeds, so $\Sigma$
is finitely generated.  It is saturated by its definition as the
full set of lattice points in the cone.  Hence the semigroup algebra
$\mathbb C[\Sigma]$, graded by $m$ and with degree-zero part $\mathbb C[t]$,
is a finitely generated normal graded $\mathbb C[t]$-algebra; here $t$
corresponds to $(0,0,1)$.

Enlarge the divisibility condition on $r$ so that $q\mid r$.  By
Lemma~\ref{lem:sufficiently-divisible-veronese}, the algebra
$\mathbb C[\Sigma]^{(r)}$ is generated over $\mathbb C[t]$ by its degree-one
part.  Its semigroup is saturated in its group, so the Veronese algebra is
normal.  Each homogeneous piece is a finite free $\mathbb C[t]$-module, and
the coordinate ring of every standard affine chart of its relative
projective spectrum is torsion free over $\mathbb C[t]$ and hence flat.
After $t$ is inverted, localization removes the lower bound on the third
semigroup coordinate and identifies the relative projective spectrum with
$\mathbb P^1\times\mathbb G_m$.  Thus
\begin{equation}\label{eq:rational-crease-local-toric-family}
 T_\Sigma=\operatorname{Proj}_{\mathbb A^1}
 \mathbb C[\Sigma]^{(r)}
\end{equation}
is a normal flat toric degeneration of $\mathbb P^1$ with an invertible
relatively ample polarization $\mathcal O_{T_\Sigma}(1)$.

We now globalize this family over $B$.  Define a graded
$\mathcal O_B[t]$-algebra by
\begin{equation}\label{eq:rational-crease-global-algebra}
 \mathcal S_k
 =\bigoplus_{j=0}^{rk}
  (M^{rk}\otimes L^j)t^{g(rk,j)}\mathcal O_B[t],
 \qquad
 \mathcal S=\bigoplus_{k\geq0}\mathcal S_k.
\end{equation}
To make the gluing explicit, choose an affine cover $\{U_\alpha\}$ that
trivializes $M$ and $L$, and denote their transition functions by
$m_{\alpha\beta}$ and $l_{\alpha\beta}$.  On the $(k,j)$-summand of
\eqref{eq:rational-crease-global-algebra}, the transition multiplier is
\begin{equation}\label{eq:rational-crease-transition-multiplier}
 m_{\alpha\beta}^{rk}l_{\alpha\beta}^{j},
 \qquad t\longmapsto t.
\end{equation}
The exponents $(rk,j)$ are additive, so these multipliers preserve
multiplication and every binomial relation of the semigroup algebra.  The
cocycle law is inherited from those of $M$ and $L$.  Moreover,
$l_{\alpha\beta}^{j}$ is the fiber-torus transition and
$m_{\alpha\beta}^{rk}=(m_{\alpha\beta}^r)^k$ glues the degree-one
polarization.  Hence
\begin{equation}\label{eq:rational-crease-relative-proj}
 \mathcal X_\mu=\operatorname{Proj}_{B\times\mathbb A^1}\mathcal S,
 \qquad
 \mathcal B_\mu=\mathcal O_{\mathcal X_\mu}(1)
\end{equation}
is obtained on every $U_\alpha$ from $U_\alpha\times T_\Sigma$.  Thus the
local families, the test action, the fiber action, and their polarizations
glue to a normal flat $T$-equivariant family with a proper morphism
\begin{equation}\label{eq:rational-crease-projective-map}
 \rho\colon\mathcal X_\mu\to B\times\mathbb A^1.
\end{equation}
On each trivializing open subset $U\subseteq B$,
\eqref{eq:rational-crease-local-toric-family} shows that the restriction of
$\rho$ over $U\times\mathbb A^1$ is projective and that $\mathcal B_\mu$ is
ample relative to this restriction.  Properness and relative ampleness are
local on the target, so $\rho$ is proper and $\mathcal B_\mu$ is
$\rho$-ample.  Since $B$ is projective, the composite
$\mathcal X_\mu\to\mathbb A^1$ is proper.  After inverting $t$, its polarized
fiber is $(X,A^r)$.

We verify ampleness relative to this composite.  Put
$R_\Sigma=\mathbb C[\Sigma]^{(r)}$.  The quotient $R_\Sigma/(t)$ has the
boundary-monomial basis
\begin{equation}\label{eq:rational-crease-central-boundary-basis}
 \chi^{(rk,j,g(rk,j))},
 \qquad k\geq0,
 \qquad 0\leq j\leq rk,
\end{equation}
because a monomial whose third coordinate is strictly larger than
$g(rk,j)$ is divisible by $t$.  The product of two boundary monomials
survives modulo $t$ exactly when their exponent vectors lie in a common
linearity cone of $g$.  For vectors in the interiors of the two different
cones, the strict inequality
\[
 g(u+v)<g(u)+g(v)
\]
makes their product divisible by $t$.

Let $\Sigma_0$ and $\Sigma_1$ be the two lower-face semigroups.  Retaining
the monomials on $\Sigma_a$ and killing the other boundary monomials defines
a homomorphism to $\mathbb C[\Sigma_a]$.  The two homomorphisms give an
homomorphism
\begin{equation}\label{eq:rational-crease-central-face-injection}
 R_\Sigma/(t)\to
 \mathbb C[\Sigma_0]\times\mathbb C[\Sigma_1].
\end{equation}
Within either face, the retained boundary monomials map to distinct monomials
and are therefore linearly independent.  Since every boundary monomial belongs
to at least one face, the two kernels have zero intersection.
Thus the homomorphism in
\eqref{eq:rational-crease-central-face-injection} is injective.  The two
kernels are incomparable prime ideals.  Thus
the local central fiber is reduced, has no embedded component, and has
exactly the two components determined by the linearity intervals.  The
transition multipliers in
\eqref{eq:rational-crease-transition-multiplier} preserve both face
semigroups, so this description glues over $B$.

Put $j_0=0$, $j_1=r\mu$, and $j_2=r$.  The two irreducible components of the
reduced central fiber correspond to the two linearity intervals
$[j_0/r,j_1/r]$ and $[j_1/r,j_2/r]$.  On the component corresponding to
$[j_a/r,j_{a+1}/r]$, the associated fiber over $B$ is a projective line
bundle whose two fixed sections correspond to the fiber weights $j_a$ and
$j_{a+1}$.  If this component is denoted by $Y_a$ and its projection by
$\pi_a$, put $s_a=j_{a+1}-j_a$.  After removing the affine $t$-character,
the graded face algebra of this component is
\begin{equation}\label{eq:rational-crease-face-algebra}
 \bigoplus_{k\geq0}\ \bigoplus_{h=0}^{s_ak}
 M^{rk}\otimes L^{kj_a+h}.
\end{equation}
This is the $s_a$-th Veronese of the symmetric algebra of
$\mathcal O_B\oplus L$, twisted in degree $k$ by
$M^{rk}\otimes L^{kj_a}$.  Consequently,
\begin{equation}\label{eq:rational-crease-central-component}
 Y_a\simeq\mathbb P_B(\mathcal O_B\oplus L),
 \qquad
 \mathcal B_\mu\vert_{Y_a}
 \simeq\mathcal O_{Y_a}(j_{a+1}-j_a)
 \otimes\pi_a^*(M^r\otimes L^{j_a}).
\end{equation}
The restrictions of $\mathcal B_\mu$ to the two fixed sections of $Y_a$
are
\begin{equation}\label{eq:rational-crease-central-endpoints}
 M^r\otimes L^{j_a}
 \quad\text{and}\quad
 M^r\otimes L^{j_{a+1}}.
\end{equation}
For $a=0$ and $a=1$, the common face has algebra
$\bigoplus_{k\geq0}M^{rk}\otimes L^{kj_1}$.  Hence the two components meet
along their common fixed section $B$, and both restrictions of
$\mathcal B_\mu$ to this intersection equal $M^r\otimes L^{j_1}$.
For every $0\leq j\leq r$ and every curve factor $C_i$, the degree of
$M_i^r\otimes L_i^j$ is
\begin{equation}\label{eq:rational-crease-central-degrees}
 r\alpha_i+j\delta_i=r\ell_i(j/r)>0
\end{equation}
by \eqref{eq:four-affine-factors}.
Lemma~\ref{lem:split-projective-bundle-ampleness}, applied with
$N=M^r\otimes L^{j_a}$ and $s=j_{a+1}-j_a$, shows that
$\mathcal B_\mu$ is ample on each component of the reduced central fiber.
Lemma~\ref{lem:componentwise-ampleness} then shows that it is ample on the central
fiber.  Theorem~\ref{thm:relative-ampleness-near-fiber}
gives a neighborhood of $0$ over which $\mathcal B_\mu$ is relatively
ample.  Over $\mathbb G_m$, the polarized family is the product with
$(X,A^r)$ and is relatively ample there.  These two open subsets cover
$\mathbb A^1$, so Theorem~\ref{thm:relative-ampleness-and-projectivity} shows
that $\mathcal B_\mu$ is ample relative to $\mathbb A^1$ and that the proper
morphism to $\mathbb A^1$ is projective.  This polarized family is
$\mathcal T(d,\mu;r)$.

It remains to identify the induced filtration.  Put
$I=(R_\Sigma)_{>0}$.  Since $R_\Sigma$ is generated in degree one,
$\mathcal O_{T_\Sigma}(1)$ is invertible and the punctured cone
\begin{equation}\label{eq:rational-crease-punctured-cone}
 U_\Sigma=\operatorname{Spec}R_\Sigma\setminus V(I)
 \to
 T_\Sigma=\operatorname{Proj}_{\mathbb C[t]}R_\Sigma
\end{equation}
is the associated $\mathbb G_m$-torsor.  The group generated by the
semigroup of $R_\Sigma$ has rank three, whereas
$R_\Sigma/I=\mathbb C[t]$.  Thus $V(I)$ has codimension two.  Normality and
the codimension statement give
\begin{equation}\label{eq:rational-crease-punctured-sections}
 \Gamma(U_\Sigma,\mathcal O_{U_\Sigma})=R_\Sigma.
\end{equation}
A regular function on $U_\Sigma$ defines an element of
$\operatorname{Frac}(R_\Sigma)$ lying in $(R_\Sigma)_{\mathfrak p}$ for
every height-one prime $\mathfrak p$, since no such prime belongs to
$V(I)$.  A normal noetherian domain is the intersection of its height-one
localizations inside its fraction field by \cite[Tag~031T]{Sta26}, so
this element belongs to
$R_\Sigma$; the reverse inclusion is immediate.  Taking the weight-$k$
part on the torsor in \eqref{eq:rational-crease-punctured-cone} gives, for
every $k\geq0$,
\begin{equation}\label{eq:rational-crease-exact-proj-sections}
 H^0(T_\Sigma,\mathcal O_{T_\Sigma}(k))=(R_\Sigma)_k.
\end{equation}
This identity is compatible with the transition multipliers in
\eqref{eq:rational-crease-transition-multiplier}; it therefore glues to
\begin{equation}\label{eq:rational-crease-relative-section-pushforward}
 \rho_*\mathcal B_\mu^k=\mathcal S_k.
\end{equation}
Taking sections over $B\times\mathbb A^1$ gives the exact section lattice
\begin{equation}\label{eq:rational-crease-section-lattice}
 \bigoplus_{j=0}^{rk}
 t^{drkf_\mu(j/(rk))}V_{rk,j}\otimes_{\mathbb C}\mathbb C[t]
\end{equation}
inside the marked generic algebra.  This proves
\eqref{eq:rational-crease-weight}.  Finally, the crease is interior, so the
entries are not affine in $(rk,j)$.  Since every one-parameter subgroup of
$\Aut(X)$ lies in $\Aut^0(X)=\mathbb C^*$, the resulting test configuration
is not a product.

\medskip
\noindent\textbf{Step 2.} In this step, we compute the Hilbert and weight
coefficients and obtain the formula given by point evaluation.
For all sufficiently large $k$, every line bundle appearing in the summands
$V_{rk,j}$ on the four curve factors has degree greater than $2g_i-2$ and
therefore has dimension equal to its degree plus $1-g_i$.  The tensor-product
decomposition gives,
uniformly for $0\leq j\leq rk$,
\begin{equation}\label{eq:rational-crease-block-rr}
 \dim V_{rk,j}
 =\prod_{i=0}^3(rk\alpha_i+j\delta_i+\beta_i)
 =(rk)^4p(j/(rk))+(rk)^3q(j/(rk))+O(k^2).
\end{equation}
For a polynomial $\psi$, expansion into monomials and the exact power-sum
identities give
\begin{equation}\label{eq:rational-crease-polynomial-sum}
 \sum_{j=0}^{N}\psi(j/N)
 =N\int_0^1\psi(x)\,dx
  +\frac{\psi(0)+\psi(1)}2+O(N^{-1}).
\end{equation}
Applying \eqref{eq:rational-crease-polynomial-sum} to the two polynomials in
\eqref{eq:rational-crease-block-rr} gives
\begin{equation}\label{eq:rational-crease-hilbert-expansion}
 h_r(k)=r^5a_0k^5+r^4a_1k^4+O(k^3).
\end{equation}
For a rational convex piecewise-linear function $f$, define
\begin{equation}\label{eq:rational-crease-weight-coefficients}
 b_0(f)=\int_0^1f(x)p(x)\,dx,
 \qquad
 b_1(f)=\frac{f(0)p(0)+f(1)p(1)}2
       +\int_0^1f(x)q(x)\,dx.
\end{equation}
For $f=f_\mu$, the integer $rk\mu$ is a lattice point.  Apply
\eqref{eq:rational-crease-polynomial-sum} separately on the two sides of
this point to the polynomial pieces of $fp$ and $fq$.  The two endpoint
terms with coefficient $1/2$ at the crease, followed by the subtraction of
the duplicated lattice value, cancel because $f$ is continuous.  Thus only the endpoints $0$ and
$1$ contribute to the second coefficient, and the total increasing entry in
test-configuration degree $k$ is
\begin{equation}\label{eq:rational-crease-weight-expansion}
 w_r(k)=dr^6b_0(f)k^6+dr^5b_1(f)k^5+o(k^5).
\end{equation}
The exponent factors cancel in the Donaldson--Futaki normalization:
\begin{equation}\label{eq:rational-crease-veronese-df}
 \DF(f)
 =\frac{2\bigl((dr^5b_1)(r^5a_0)-(r^4a_1)(dr^6b_0)\bigr)}
 {(r^5a_0)^2}
 =\frac{2d(b_1a_0-a_1b_0)}{a_0^2}.
\end{equation}
Since $\overline S=2a_1/a_0$ and $F''=2q-\overline S p$,
\eqref{eq:rational-crease-veronese-df} gives
\begin{equation}\label{eq:rational-crease-integration-by-parts}
 \frac{1}{d}\DF(f)
 =\frac1{a_0}\left(
 f(0)p(0)+f(1)p(1)+\int_0^1F''(x)f(x)\,dx\right)
 =\frac1{a_0}\int_0^1F(x)\,d(f'(x)).
\end{equation}
Here we used the four boundary identities
$F(0)=F(1)=0$, $F'(0)=p(0)$, and $F'(1)=-p(1)$ from
Lemma~\ref{lem:boundary-polynomial}.  For $f=f_\mu$, the measure
$d(f')$ is the unit point mass at $\mu$.  This proves
\eqref{eq:rational-crease-df}.

The expansion \eqref{eq:rational-crease-block-rr} also shows that the normalized
block-counting measures converge to the measure in
\eqref{eq:rational-crease-density}.  Taking $f(x)=x$ in the first expression
of \eqref{eq:rational-crease-integration-by-parts} gives
\begin{equation}\label{eq:fiber-product-df-calculation}
 \DF_{\mathrm{fib}}
 =\frac{p(1)+\int_0^1xF''(x)\,dx}{a_0}
 =\frac{p(1)+F'(1)-F(1)+F(0)}{a_0}
 =0,
\end{equation}
where the last equality uses the boundary identities in
Lemma~\ref{lem:boundary-polynomial}.  This proves
Proposition~\ref{prop:rational-single-crease}.
\end{proof}

\begin{proof}[Proof of Theorem~\ref{thm:uniform-threshold}]
We argue in four steps.  First, we construct the Fibonacci test
configurations and prove convergence of their crease parameters.  We then
compute their Donaldson--Futaki invariants, compute their reduced
non-Archimedean $J$-functionals, and finally compare these quantities with
the relative non-Archimedean Mabuchi functional.  We write
$\DF_n=\DF(\mathcal T_n)$ and
$J^{\mathrm{NA}}_{T,n}=J_T^{\mathrm{NA}}(\mathcal T_n)$ throughout.

\medskip
\noindent\textbf{Step 1.} In this step, we construct $\mathcal T_n$ and
prove the convergence in \eqref{eq:intro-fibonacci-limit}.
The Fibonacci recurrence gives
$\gcd(p_n,q_n)=\gcd(Q_n,Q_{n+2})=\gcd(Q_n,Q_{n+1})=1$, so $\mu_n$ is in
lowest terms.  For $n\geq3$,
Proposition~\ref{prop:rational-single-crease} applied with
\begin{equation}\label{eq:uniform-choice-of-crease}
 d=q_n,
 \qquad
 \mu=\mu_n=\frac{p_n}{q_n}
\end{equation}
gives a normal nonproduct $T$-equivariant test configuration
$\mathcal T_n$ of exponent $r_n$ whose block entry in test-configuration
degree $k$ is
\begin{equation}\label{eq:uniform-fibonacci-weight}
 q_nr_nk f_{\mu_n}(j/(r_nk))
 =\max\{0,q_nj-p_nr_nk\},
 \qquad 0\leq j\leq r_nk.
\end{equation}
Since $q_n=Q_{n+1}+Q_n$, the inequalities
$Q_{n+1}>Q_n$ and $Q_{n+1}<2Q_n$ give
\begin{equation}\label{eq:uniform-fibonacci-interval}
 \frac13<\mu_n<\frac12.
\end{equation}
The recurrence gives
\begin{equation}\label{eq:uniform-cassini}
 \begin{split}
 p_n^2-3p_nq_n+q_n^2
 &=Q_{n+1}^2-Q_nQ_{n+1}-Q_n^2\\
 &=Q_{n+1}Q_{n-1}-Q_n^2=(-1)^n.
 \end{split}
\end{equation}
After division by $q_n^2$, every limit point of $\mu_n$ in
$[1/3,1/2]$ is a root of $x^2-3x+1$.  The only such root in this interval
is
\begin{equation}\label{eq:uniform-irrational-limit}
 \lambda=\frac{3-\sqrt5}{2}.
\end{equation}
This proves the asserted convergence.

\medskip
\noindent\textbf{Step 2.} In this step, we compute $\DF_n$ exactly.
Write
\begin{equation}\label{eq:uniform-boundary-constants}
 C_{\mathrm{bd}}=461999\cdot2327\cdot356\cdot52,
 \qquad
 K=90C_{\mathrm{bd}}.
\end{equation}
Lemma~\ref{lem:boundary-polynomial} and
\eqref{eq:uniform-cassini} give
\begin{equation}\label{eq:uniform-boundary-value}
 F(\mu_n)
 =K\mu_n(1-\mu_n)(\mu_n^2-3\mu_n+1)^2
 =\frac{K\mu_n(1-\mu_n)}{q_n^4}.
\end{equation}
By Proposition~\ref{prop:rational-single-crease},
\begin{equation}\label{eq:uniform-df-rate}
 \DF_n
 =\frac{q_nF(\mu_n)}{a_0}
 =\frac{K\mu_n(1-\mu_n)}{a_0q_n^3}>0.
\end{equation}

\medskip
\noindent\textbf{Step 3.} In this step, we compute the reduced
non-Archimedean $J$-functional and obtain its uniform lower bound.
We use the conventions of Nitta--Saito
\cite[Section~3.2 and Definitions~3.6.1 and~3.6.3]{NS21}.  By
\eqref{eq:rational-crease-density}, the limiting fiber-coordinate measure is
\begin{equation}\label{eq:uniform-coordinate-measure}
 d\nu(x)=\frac{p(x)}{a_0}\,dx.
\end{equation}
In test-configuration degree $k$, the Duistermaat--Heckman normalization
divides the entries by the exponent times the degree, namely $rk$.  Formula
\eqref{eq:rational-crease-weight} therefore gives the normalized entry
$df_\mu(x)$, where $x=j/(rk)$.  A real twist by the fiber-scaling torus adds
a real multiple of $x$; after division by $d$, we denote that multiple by
$s$.  The Duistermaat--Heckman measure in the definition of
$J^{\mathrm{NA}}$ is the central-weight measure.  Since central weights are
the negatives of increasing entries by Definition~\ref{defn:test-configuration},
$J^{\mathrm{NA}}$ is the entry mean minus the minimum entry.  Therefore
\begin{equation}\label{eq:uniform-reduced-j-minimization}
 \frac{J_T^{\mathrm{NA}}(\mathcal T(d,\mu;r))}{d}
 =\inf_{s\in\mathbb R}
 \left(
 \int_0^1(f_\mu(x)+sx)\,d\nu(x)
 -\min_{0\leq x\leq1}\{f_\mu(x)+sx\}
 \right).
\end{equation}
Put
\begin{equation}\label{eq:uniform-first-moment}
 m_1=\int_0^1x\,d\nu(x),
 \qquad
 I_\mu=\int_0^1f_\mu(x)\,d\nu(x),
\end{equation}
so that $0<m_1<1$.  The convex piecewise-linear function
$f_\mu(x)+sx$ has slopes $s$ and $s+1$ on the two sides of $\mu$.
Consequently, the objective in
\eqref{eq:uniform-reduced-j-minimization} is
\begin{equation}\label{eq:uniform-piecewise-linear-objective}
 \begin{cases}
 I_\mu+s m_1,&s\geq0,\\
 I_\mu+s(m_1-\mu),&-1\leq s\leq0,\\
 I_\mu+\mu-1+s(m_1-1),&s\leq-1.
 \end{cases}
\end{equation}
The first branch is minimized at $s=0$, the third at $s=-1$, and the
middle branch is affine.  Its minimum is therefore attained at one of those
two endpoints.  Hence
\begin{equation}\label{eq:uniform-reduced-j-formula}
 \frac{J_T^{\mathrm{NA}}(\mathcal T(d,\mu;r))}{d}
 =\min\left\{
 \int_0^\mu(\mu-x)\,d\nu(x),
 \int_\mu^1(x-\mu)\,d\nu(x)
 \right\}.
\end{equation}

Write
\begin{equation}\label{eq:uniform-density-polynomial}
 P(x)=(1+29x)(6-5x)(3-2x)(5-2x),
 \qquad
 p(x)=C_{\mathrm{bd}}P(x).
\end{equation}
We have $P(x)\leq2700$ on $[0,1]$, and hence
$a_0\leq2700C_{\mathrm{bd}}$.  Suppose that
$\mu\in[1/3,1/2]$.  On $[3/4,1]$ we have
$x-\mu\geq1/4$ and $P(x)\geq273/4$, while on $[0,1/4]$ we have
$\mu-x\geq1/12$ and $P(x)\geq855/16$.  It follows that
\begin{align}
 \int_\mu^1(x-\mu)\,d\nu(x)
 &\geq\frac{(1/4)(1/4)(273/4)}{2700}
 =\frac{91}{57600},
 \label{eq:uniform-right-j-lower-bound}\\
 \int_0^\mu(\mu-x)\,d\nu(x)
 &\geq\frac{(1/4)(1/12)(855/16)}{2700}
 =\frac{19}{46080}.
 \label{eq:uniform-left-j-lower-bound}
\end{align}
Since $19/46080<91/57600$, formula
\eqref{eq:uniform-reduced-j-formula} gives
\begin{equation}\label{eq:uniform-reduced-j-lower-bound}
 \frac{J_T^{\mathrm{NA}}(\mathcal T(d,\mu;r))}{d}
 \geq\frac{19}{46080}.
\end{equation}
Taking $d=q_n$ and $\mu=\mu_n$, using
$\mu_n(1-\mu_n)\leq1/4$, and combining \eqref{eq:uniform-df-rate} with
\eqref{eq:uniform-reduced-j-lower-bound}, we obtain
\begin{equation}\label{eq:uniform-df-j-ratio}
 0<\frac{\DF_n}{J^{\mathrm{NA}}_{T,n}}
 \leq
 \frac{K}{4a_0(19/46080)q_n^4}
 \to0.
\end{equation}
This proves \eqref{eq:intro-zero-reduced-threshold}.

\medskip
\noindent\textbf{Step 4.} We compare the ratio in
\eqref{eq:uniform-df-j-ratio} with the relative non-Archimedean Mabuchi
functional and conclude the proof in this step.
The product configuration generated by fiber scaling has zero
Donaldson--Futaki invariant by \eqref{eq:fiber-product-df}.  For a product
configuration this invariant is a positive multiple of the classical Futaki
invariant of the generating vector field \cite[Section~2.2]{Don02}; since
$\operatorname{Lie}(T)$ is one-dimensional by
Proposition~\ref{prop:basic-geometry}, the Futaki character vanishes
identically.  The nondegenerate Futaki--Mabuchi pairing therefore makes the
extremal vector zero; see \cite[Definitions~2.3.2 and~2.3.4]{NS21}.  The extremal correction
in the relative non-Archimedean Mabuchi functional vanishes by
\cite[Definition~3.3.5(2)]{NS21}, so
\begin{equation}\label{eq:uniform-relative-equals-ordinary}
 M^{\mathrm{NA}}_{\mathrm{rel}}(\mathcal T_n)
 =M^{\mathrm{NA}}(\mathcal T_n).
\end{equation}
The comparison formula in \cite[Definition~3.4]{SD19} gives
\begin{equation}\label{eq:uniform-mabuchi-df-comparison}
 M^{\mathrm{NA}}_{\mathrm{rel}}(\mathcal T_n)
 =M^{\mathrm{NA}}(\mathcal T_n)
 \leq\DF_n.
\end{equation}
Together with \eqref{eq:uniform-df-j-ratio}, this yields
\begin{equation}\label{eq:uniform-relative-ratio}
 \limsup_{n\to\infty}
 \frac{M^{\mathrm{NA}}_{\mathrm{rel}}(\mathcal T_n)}
 {J^{\mathrm{NA}}_{T,n}}
 \leq0.
\end{equation}
Thus no $\delta>0$ can satisfy
\begin{equation}\label{eq:uniform-relative-inequality}
 M^{\mathrm{NA}}_{\mathrm{rel}}(\mathcal T)
 \geq\delta J_T^{\mathrm{NA}}(\mathcal T)
\end{equation}
for every normal $T$-equivariant test configuration $\mathcal T$.  By
\cite[Definitions~3.3.5(2) and~3.7.1(4)]{NS21}, this is precisely the failure of
uniform relative K-polystability asserted in
Theorem~\ref{thm:uniform-threshold}.  This completes
the proof of Theorem~\ref{thm:uniform-threshold}.
\end{proof}

\appendix

\section{Use of generative AI}\label{sec:generative-ai-use}

\begin{center}
{\scshape by Bin Dong\footnote{Beijing International Center for
Mathematical Research \& Center for Machine Learning Research, Peking
University, No.~5 Yiheyuan Road, Haidian District, Beijing 100871,
China.  Email: \texttt{dongbin@math.pku.edu.cn}}, Guoxiong
Gao\footnote{School of Mathematical Sciences, Peking University,
No.~5 Yiheyuan Road, Haidian District, Beijing 100871, China.
Email: \texttt{samggx@stu.pku.edu.cn}}}, and Jihao Liu
\end{center}

\medskip

Artificial intelligence has lately proven strikingly good at producing
counterexamples, and this paper too disproves a conjecture by exhibiting
an explicit variety.  It belongs, however, to a different genre than
most counterexamples found by AI so far.  The statement that an object
is a counterexample can be of two kinds.  It may reduce to a finite
certificate, checkable by a finite mechanical computation once the
object is written down: the recent counterexample to the Jacobian
conjecture \cite{Alp26} is typical, an explicit polynomial map whose
Jacobian determinant is a nonzero constant by direct expansion and
whose non-injectivity is witnessed by two points with the same image.
Or it may admit no finite certificate even in principle, because it is
itself a theorem, typically universally quantified: the recent
construction of a non-sofic group \cite{OAI26} is of this kind, since
non-soficity asserts the nonexistence of approximate embeddings into
finite symmetric groups of any size and can be witnessed by no
computation with the group itself.  A counterexample of the second kind
is, in substance, a proof.

The distinction is older than AI.  Proposed disproofs of the Hodge
conjecture have never been short of candidate classes; they founder on
proving that the candidate is the class of no algebraic cycle
whatsoever.  Rationally connected varieties are expected not to be
unirational in general, and plausible candidates abound---a very
general hypersurface of bidegree $(2,4)$ in
$\mathbb P^2\times\mathbb P^2$ may well be one---but no known technique
can prove any rationally connected variety non-unirational, so no
candidate can be certified.

This paper is of the second kind.  Candidate manifolds of the present
shape have been available since \cite{ACG+08}, and producing candidates
is precisely what contemporary AI does well; what had been missing, and
what constitutes the mathematical content of this paper, is the proof
that the mechanism works---that every normal ample algebraic test
configuration of the fivefold has nonnegative Donaldson--Futaki
invariant, with equality only for products, a statement quantified over
all degenerations and established by the classification of
Sections~\ref{sec:affine-initials}--\ref{sec:main-proof}.  This, we
believe, is the axis along which AI-assisted mathematics should be
judged: not whether the conclusion is a counterexample, but whether the
assertion that it is one is a finite certificate or a proof.

In the course of exploring a number of open problems with Claude Code,
the author found initial signs of a possible breakthrough on the
problem studied here, and then had Claude Code (Fable 5), Codex
(GPT-5.6-sol), and Danus \cite{Liu+26} work on it in collaboration;
the three systems together produced the counterexample and its proof.
Human input was crucial at one point: the author realized that the
example constructed in Theorem~\ref{thm:main} is either a
counterexample to the Codogni--Stoppa conjecture or a counterexample
to the cscK Yau--Tian--Donaldson conjecture---either case would be
striking---and asked the agents to go all-in on this particular
example and determine which conjecture is false.  To understand what
the AI systems had in fact contributed, the problem was then attempted
afresh by an improved version of Danus\footnote{This improved version
of Danus will be open-sourced shortly.} working alone.  Given only the
original problem and none of the earlier findings, it settled the
problem 5 hours and 29 minutes into its run, in less time than the
collaboration had taken and without the human input that had been
key to the first attempt, and produced a counterexample and a complete proof of this paper.  The text of the main part of the paper
then went through several rounds of discussion between the author and
Danus to improve its presentation, none of which altered the
mathematics, and a final round of human polishing and checking.

Danus is a mathematical agent built on top of Rethlas \cite{Ju+26} by
the same team and designed for long-horizon reasoning.  A main agent
orchestrates Rethlas agents as workers on a single problem: the
workers produce facts, which accumulate into a fact graph whose edges
record their dependencies, and the run stops only when the target
statement itself appears in the graph as a fact.  The design is
described in \cite{Liu+26} and the accompanying open-source
code.\footnote{\texttt{https://github.com/frenzymath/Danus}}  The
improved version differs from it in several respects.  The Rethlas
workers now run on GPT-5.6-sol, and the main agent moved to a
Codex-based implementation using the same model as well, so the whole
system runs on a single API.  With models of this strength, a design in which the main
agent does no mathematics itself and defers high-level planning to a
strategic consultation of GPT-5.5 Pro leaves their capability
underused; the consultation was removed, and the main agent now does
the mathematical thinking and directs the global strategy, dispatching
Codex subagents---three in the present work---to extend its
mathematical reach and to relieve the context burden of a large fact
graph and memory.  Its global strategic reflection was strengthened,
so that it does not stall in a dead end and lose sight of the need to
change direction, and fact granularity was retuned so that workers
produce longer, more complex local results.  Together these changes
let Danus make fuller use of the current generation of models.

The authors also tested how other agent systems perform on this
problem, posed to them as a request to prove or disprove the
conjecture: QED \cite{An+26}, ProofCouncil \cite{Sch+26}, MechMath
\cite{Cao+26}, Codex (GPT-5.6-sol), Claude Code (Fable 5), and
GPT-5.6-sol Pro asked directly through its web interface.  The first
three all run on GPT-5.6-sol, with the effort level of their Codex
components or GPT API set to ``xhigh''; Codex and Claude Code were set to
``max'' effort.  Given the original problem, the first three ran for
twelve hours without producing a solution; Codex produced a proof,
rejected it in its own verification, and gave no valid result within
the time limit; Claude Code reported that this is a well-known conjecture and
offered some possible approaches but no solution, and GPT-5.6-sol Pro,
after 133 minutes of thinking, reached much the same assessment and
likewise claimed no solution.  On the easier task of being given the
counterexample of this paper and asked only for a proof that it is
one, none of the six produced a complete proof under the same
twelve-hour limit and settings.  Verifying this counterexample is
therefore not trivial, as the second kind of counterexample described
above requires.

The original run, the one this paper comes from, had eight Rethlas
workers, with the Codex effort level set to ``xhigh'' for four of them
and ``high'' for the other four.  Its final fact graph contains 616
verified facts, and its global memory contains 353 conclusions, 143
identified obstacles, 96 directions, 33 proof attempts, 26 plans, 17
counterexamples, 2 recorded dead ends, and 924 verification records.
Table~\ref{tab:closure} records how the 616 facts divide: 88 lie in
the supporting closure of the final main-theorem fact, computed from
the dependencies recorded with each fact, and the remaining 528 lie
outside it, unused by the paper.

\begin{table}[ht]
\centering
\begin{tabular}{lrr}
\toprule
& Count & Share of all facts \\
\midrule
Main-theorem closure & 88 & 14\% \\
Outside the closure, unused by the paper & 528 & 86\% \\
\midrule
Total & 616 & 100\% \\
\bottomrule
\end{tabular}
\caption{Distribution of facts in the fact graph.}
\label{tab:closure}
\end{table}

Table~\ref{tab:roles} describes what the 88 in-closure facts do.  The
decomposition is approximate: a fact often serves several purposes at
once, and each is placed by the role of its final clause.  Six of the
88 are theorems from the literature, restated by a worker so that
other facts could cite them; the remaining 82 the swarm proved itself.
Only a minority appear in the paper as statements in their own right;
the rest are internal steps beneath the argument the paper presents.

\begin{table}[ht]
\small
\centering
\begin{tabular}{@{}p{3.3cm}r@{\hskip 10pt}p{9.4cm}@{}}
\toprule
Role & Count & Function in the proof \\
\midrule
Veronese gradings and rays & 20 &
Fix the gradings under which a test configuration is read as a
filtration, then work one rational ray at a time: sufficiently
divisible Veronese exponents with normal, finitely generated section
algebras, multiplication along blocks and rays, and the Chow weight
carried by each ray. \\[2pt]
The scalar profile & 14 &
Expand the block dimensions and total weights in the degree, average
them over the fiber index into a one-variable convex profile, evaluate
the Donaldson--Futaki invariant of that profile against the boundary
data, and bound the accumulated blockwise gaps by summability
estimates. \\[2pt]
Newton--Okounkov convex transforms & 10 &
Compute the product-box Newton--Okounkov body of the split bundle and
control the Boucksom--Chen transforms on it: convexity and rigidity of
the slice averages, and the theorem that both opposite initials of a
zero-invariant configuration have affine transforms. \\[2pt]
Duistermaat--Heckman measures & 10 &
Constrain the limit weight measure: invariance under specialization,
rational breakpoints, coefficients and moments, and the exclusion, by
rationality alone, of a crease at the irrational double zero. \\[2pt]
Semistability and polystability & 8 &
Signs and equality cases quantified over whole classes of degenerations
of $(X,A)$, from nonnegativity for every equivariant filtration up to
K-semistability at every exponent, the classification of the
zero-invariant case as products, and the capstone. \\[2pt]
Nonexistence of extremal metrics & 7 &
The analytic half: admissible metrics whose scalar-curvature error
tends to zero, the automorphism and maximal-compactness input that
makes the admissible ansatz exhaustive, and the conclusion that
$c_1(A)$ contains no extremal metric. \\[2pt]
Comparator and Smith spectrum & 7 &
Prepare a zero-invariant configuration for comparison against an
integral product comparator: normalization priced against the
invariant, ramified base change realized as a weight dilation, and the
relative Smith elementary divisors identified with the residual jumps.
\\[2pt]
The admissible boundary polynomial & 6 &
Fix the curve datum with its genera and degrees, produce the
admissible boundary polynomial with the required endpoint data, and
certify its nonnegativity on the interval with a single interior
double zero at an irrational point. \\[2pt]
Recorded external results & 6 &
Published theorems restated in internal notation with no derived
content, among them the admissible-extremal criterion of \cite{ACG+08}
and standard identifications of the Duistermaat--Heckman measure. \\
\bottomrule
\end{tabular}
\caption{Functional decomposition of the 88 facts in the main-theorem
closure.}
\label{tab:roles}
\end{table}

Table~\ref{tab:outside} summarizes the 528 facts outside the closure.
Unused here means unused by the final proof, not mathematically
meaningless or unhelpful to the search: some were proved for arbitrary
data rather than for the fivefold, and the proof needs only particular
instances of them; some prove the Donaldson--Futaki inequality one class of
degenerations at a time, later absorbed by the general classification;
some record that a step fails, and so kept workers from retrying it;
others rule out simpler candidates, and so fixed the shape of the
example.

\begin{table}[ht]
\small
\centering
\begin{tabular}{@{}p{3.3cm}r@{\hskip 10pt}p{9.4cm}@{}}
\toprule
Cluster & Count & Content \\
\midrule
Superseded duplicates & 146 &
Earlier or duplicate forms of conclusions the graph later carries in
better shape, including the endgame theorem itself, which several
workers wrote out independently. Only one instance of each conclusion
can enter the closure. \\[2pt]
Lemmas beyond the fivefold & 113 &
Statements proved for arbitrary graded algebras, filtrations,
Newton--Okounkov bodies and products of arbitrarily many curves, where
the paper instantiates a handful of them once each. Staying general is
what let one lemma serve every worker and every exponent. \\[2pt]
Abandoned torus-specialization route & 64 &
Structural analysis of a hypothetical zero-invariant non-product
configuration along the Codogni--Stoppa specialization strategy. The
final proof reaches the same conclusion through the affine transforms
and the oriented Smith spectrum instead. \\[2pt]
Restricted-class nonnegativity & 61 &
Sign results for the fivefold proved one named class of degenerations
at a time, blow-ups, monomial and flag filtrations among them, each a
restricted case of Theorem~\ref{thm:main}(2) that the final
classification later absorbs. \\[2pt]
Undischarged hypotheses & 58 &
Positive results resting on an antecedent the run never verified,
chiefly finite generation. The most consequential is the criterion a
boundary polynomial would have to meet, which the explicit four-factor
datum was then built to satisfy. \\[2pt]
Refuted proof steps & 56 &
Facts whose whole content is that a proposed step fails: an implication
that does not hold, an estimate that cannot exist, an ansatz that is the
flat specialization of no test configuration. Each closes a branch. \\[2pt]
Excluded candidate families & 24 &
Negative existence results over whole families of candidates, showing
that split line bundles over one, two or three curves satisfy the
conjecture. This is why the base is a product of four curves. \\[2pt]
Off-path external results & 6 &
Published results transcribed into the run's own notation so that
workers could apply them inside the graph, on paths the final proof does
not take. \\
\bottomrule
\end{tabular}
\caption{Clusters among the 528 facts outside the main-theorem closure.}
\label{tab:outside}
\end{table}

\FloatBarrier


\begin{thebibliography}{99}

\bibitem[Alp26]{Alp26}
L.~Alp\"oge,
announcement of a counterexample to the Jacobian conjecture,
post on X (formerly Twitter), July 20, 2026,
\texttt{https://x.com/i/status/2079028340955197566}.

\bibitem[An$^+$26]{An+26}
C.~An, Q.~Ye, M.~Pan, and J.~Zhang,
\textit{QED: An Open-Source Multi-Agent System for Generating
Mathematical Proofs on Open Problems},
arXiv:2604.24021.

\bibitem[ACG$^+$08]{ACG+08}
V.~Apostolov, D.~M.~J.~Calderbank, P.~Gauduchon, and
C.~W.~T\o nnesen-Friedman,
\textit{Hamiltonian $2$-forms in K\"ahler geometry. III. Extremal metrics
and stability},
Invent. Math. \textbf{173} (2008), no.~3, 547--601.

\bibitem[ACG$^+$11]{ACG+11}
V.~Apostolov, D.~M.~J.~Calderbank, P.~Gauduchon, and
C.~W.~T{\o}nnesen-Friedman,
\textit{Extremal K\"ahler metrics on projective bundles over a curve},
Adv. Math. \textbf{227} (2011), no.~6, 2385--2424.

\bibitem[AH15]{AH15}
V.~Apostolov and H.~Huang,
\textit{A splitting theorem for extremal K\"ahler metrics},
J. Geom. Anal. \textbf{25} (2015), no.~1, 149--170.

\bibitem[APS26]{APS26}
V.~Apostolov, B.~Pym, and J.~Streets,
\textit{Poisson K-stability and the semiclassical Yau--Tian--Donaldson
correspondence},
arXiv:2607.06688.

\bibitem[Ber16]{Ber16}
R.~J.~Berman,
\textit{K-polystability of $\mathbb Q$-Fano varieties admitting
K\"ahler--Einstein metrics},
Invent. Math. \textbf{203} (2016), no.~3, 973--1025.

\bibitem[BB17]{BB17}
R.~J.~Berman and B.~Berndtsson,
\textit{Convexity of the K-energy on the space of K\"ahler metrics and
uniqueness of extremal metrics},
J. Amer. Math. Soc. \textbf{30} (2017), no.~4, 1165--1196.

\bibitem[BBJ21]{BBJ21}
R.~J.~Berman, S.~Boucksom, and M.~Jonsson,
\textit{A variational approach to the Yau--Tian--Donaldson conjecture},
J. Amer. Math. Soc. \textbf{34} (2021), no.~3, 605--652.

\bibitem[BDL20]{BDL20}
R.~J.~Berman, T.~Darvas, and C.~H.~Lu,
\textit{Regularity of weak minimizers of the K-energy and applications to
properness and K-stability},
Ann. Sci. \'Ec. Norm. Sup\'er. (4) \textbf{53} (2020), no.~2, 267--289.

\bibitem[BJ20]{BJ20}
H.~Blum and M.~Jonsson,
\textit{Thresholds, valuations, and K-stability},
Adv. Math. \textbf{365} (2020), Paper No.~107062.

\bibitem[BX19]{BX19}
H.~Blum and C.~Xu,
\textit{Uniqueness of K-polystable degenerations of Fano varieties},
Ann. of Math. (2) \textbf{190} (2019), no.~2, 609--656.

\bibitem[BC11]{BC11}
S.~Boucksom and H.~Chen,
\textit{Okounkov bodies of filtered linear series},
Compos. Math. \textbf{147} (2011), no.~4, 1205--1229.

\bibitem[BHJ17]{BHJ17}
S.~Boucksom, T.~Hisamoto, and M.~Jonsson,
\textit{Uniform K-stability, Duistermaat--Heckman measures and singularities
of pairs},
Ann. Inst. Fourier (Grenoble) \textbf{67} (2017), no.~2, 743--841.

\bibitem[BHJ19]{BHJ19}
S.~Boucksom, T.~Hisamoto, and M.~Jonsson,
\textit{Uniform K-stability and asymptotics of energy functionals in K\"ahler
geometry},
J. Eur. Math. Soc. (JEMS) \textbf{21} (2019), no.~9, 2905--2944.

\bibitem[BHJ22]{BHJ22}
S.~Boucksom, T.~Hisamoto, and M.~Jonsson,
\textit{Erratum to Uniform K-stability and asymptotics of energy
functionals in K\"ahler geometry},
J. Eur. Math. Soc. (JEMS) \textbf{24} (2022), no.~2, 735--736.

\bibitem[BJ18]{BJ18}
S.~Boucksom and M.~Jonsson,
\textit{A non-Archimedean approach to K-stability},
arXiv:1805.11160.

\bibitem[BJ22]{BJ22}
S.~Boucksom and M.~Jonsson,
\textit{Global pluripotential theory over a trivially valued field},
Ann. Fac. Sci. Toulouse Math. (6) \textbf{31} (2022), no.~3, 647--836.

\bibitem[BJ23]{BJ23}
S.~Boucksom and M.~Jonsson,
\textit{A non-Archimedean approach to K-stability. II. Divisorial stability
and openness},
J. Reine Angew. Math. \textbf{805} (2023), 1--53.

\bibitem[BJ25a]{BJ25a}
S.~Boucksom and M.~Jonsson,
\textit{A non-Archimedean approach to K-stability. I. Metric geometry of
spaces of test configurations and valuations},
Ann. Inst. Fourier (Grenoble) \textbf{75} (2025), no.~2, 829--927.

\bibitem[BJ25b]{BJ25b}
S.~Boucksom and M.~Jonsson,
\textit{On the Yau--Tian--Donaldson conjecture for weighted cscK metrics},
arXiv:2509.15016.

\bibitem[Cal82]{Cal82}
E.~Calabi,
\textit{Extremal K\"ahler metrics},
in: \textit{Seminar on Differential Geometry},
Ann. of Math. Stud. \textbf{102}, Princeton Univ. Press, Princeton, NJ
(1982), 259--290.

\bibitem[Cal85]{Cal85}
E.~Calabi,
\textit{Extremal K\"ahler metrics. II},
in: \textit{Differential Geometry and Complex Analysis},
Springer, Berlin (1985), 95--114.

\bibitem[Cao$^+$26]{Cao+26}
Y.~Cao, R.~Qiu, J.~Liu, J.~Wang, D.~Guo, R.~Feng, L.~Zhi, and
X.-S.~Gao,
\textit{MechMath Agent Team: LLM Driven Agents for Mathematical
Research},
arXiv:2607.04394.

\bibitem[CC18]{CC18}
X.~Chen and J.~Cheng,
\textit{On the constant scalar curvature K\"ahler metrics, general
automorphism group},
arXiv:1801.05907.

\bibitem[CC21a]{CC21a}
X.~Chen and J.~Cheng,
\textit{On the constant scalar curvature K\"ahler metrics (I)---A priori
estimates},
J. Amer. Math. Soc. \textbf{34} (2021), no.~4, 909--936.

\bibitem[CC21b]{CC21b}
X.~Chen and J.~Cheng,
\textit{On the constant scalar curvature K\"ahler metrics (II)---Existence
results},
J. Amer. Math. Soc. \textbf{34} (2021), no.~4, 937--1009.

\bibitem[CDS15a]{CDS15a}
X.~Chen, S.~Donaldson, and S.~Sun,
\textit{K\"ahler--Einstein metrics on Fano manifolds. I: Approximation of
metrics with cone singularities},
J. Amer. Math. Soc. \textbf{28} (2015), no.~1, 183--197.

\bibitem[CDS15b]{CDS15b}
X.~Chen, S.~Donaldson, and S.~Sun,
\textit{K\"ahler--Einstein metrics on Fano manifolds. II: Limits with cone
angle less than $2\pi$},
J. Amer. Math. Soc. \textbf{28} (2015), no.~1, 199--234.

\bibitem[CDS15c]{CDS15c}
X.~Chen, S.~Donaldson, and S.~Sun,
\textit{K\"ahler--Einstein metrics on Fano manifolds. III: Limits as cone
angle approaches $2\pi$ and completion of the main proof},
J. Amer. Math. Soc. \textbf{28} (2015), no.~1, 235--278.

\bibitem[CLS14]{CLS14}
B.~Chen, A.-M.~Li, and L.~Sheng,
\textit{Uniform K-stability for extremal metrics on toric varieties},
J. Differential Equations \textbf{257} (2014), no.~5, 1487--1500.

\bibitem[CPZ15]{CPZ15}
X.~Chen, M.~P\u{a}un, and Y.~Zeng,
\textit{On deformation of extremal metrics},
arXiv:1506.01290.

\bibitem[CSW18]{CSW18}
X.~Chen, S.~Sun, and B.~Wang,
\textit{K\"ahler--Ricci flow, K\"ahler--Einstein metric, and K-stability},
Geom. Topol. \textbf{22} (2018), no.~6, 3145--3173.

\bibitem[CS19]{CS19}
G.~Codogni and J.~Stoppa,
\textit{Torus equivariant K-stability},
in: \textit{Moduli of K-stable varieties}, Springer INdAM Ser. \textbf{31},
Springer, Cham (2019), 15--35.

\bibitem[DL20]{DL20}
T.~Darvas and C.~H.~Lu,
\textit{Geodesic stability, the space of rays and uniform convexity in
Mabuchi geometry},
Geom. Topol. \textbf{24} (2020), no.~4, 1907--1967.

\bibitem[DR17a]{DR17a}
T.~Darvas and Y.~A.~Rubinstein,
\textit{Tian's properness conjectures and Finsler geometry of the space of
K\"ahler metrics},
J. Amer. Math. Soc. \textbf{30} (2017), no.~2, 347--387.

\bibitem[DZ24]{DZ24}
T.~Darvas and K.~Zhang,
\textit{Twisted K\"ahler--Einstein metrics in big classes},
Comm. Pure Appl. Math. \textbf{77} (2024), no.~12, 4289--4327.

\bibitem[DZ25]{DZ25}
T.~Darvas and K.~Zhang,
\textit{A YTD correspondence for constant scalar curvature metrics},
arXiv:2509.15173.

\bibitem[DS16]{DS16}
V.~Datar and G.~Sz\'ekelyhidi,
\textit{K\"ahler--Einstein metrics along the smooth continuity method},
Geom. Funct. Anal. \textbf{26} (2016), no.~4, 975--1010.

\bibitem[Der16]{Der16}
R.~Dervan,
\textit{Uniform stability of twisted constant scalar curvature K\"ahler
metrics},
Int. Math. Res. Not. IMRN (2016), no.~15, 4728--4783.

\bibitem[Der18]{Der18}
R.~Dervan,
\textit{Relative K-stability for K\"ahler manifolds},
Math. Ann. \textbf{372} (2018), no.~3--4, 859--889.

\bibitem[Der25]{Der25}
R.~Dervan,
\textit{The constant scalar curvature K\"ahler condition is very general},
arXiv:2504.15195.

\bibitem[DR24]{DR24}
R.~Dervan and R.~Reboulet,
\textit{Arcs, stability of pairs and the Mabuchi functional},
arXiv:2409.13617.

\bibitem[DR17b]{DR17b}
R.~Dervan and J.~Ross,
\textit{K-stability for K\"ahler manifolds},
Math. Res. Lett. \textbf{24} (2017), no.~3, 689--739.

\bibitem[DT92]{DT92}
W.~Ding and G.~Tian,
\textit{K\"ahler--Einstein metrics and the generalized Futaki invariant},
Invent. Math. \textbf{110} (1992), no.~1, 315--335.

\bibitem[Don97]{Don97}
S.~K.~Donaldson,
\textit{Remarks on gauge theory, complex geometry and $4$-manifold topology},
in: \textit{Fields Medallists' Lectures},
World Sci. Ser. 20th Century Math. \textbf{5}, World Scientific, River Edge,
NJ (1997), 384--403.

\bibitem[Don99]{Don99}
S.~K.~Donaldson,
\textit{Symmetric spaces, K\"ahler geometry and Hamiltonian dynamics},
in: \textit{Northern California Symplectic Geometry Seminar},
Amer. Math. Soc. Transl. Ser. 2 \textbf{196}, Amer. Math. Soc.,
Providence, RI (1999), 13--33.

\bibitem[Don01]{Don01}
S.~K.~Donaldson,
\textit{Scalar curvature and projective embeddings. I},
J. Differential Geom. \textbf{59} (2001), no.~3, 479--522.

\bibitem[Don02]{Don02}
S.~K.~Donaldson,
\textit{Scalar curvature and stability of toric varieties},
J. Differential Geom. \textbf{62} (2002), no.~2, 289--349.

\bibitem[Don05a]{Don05a}
S.~K.~Donaldson,
\textit{Interior estimates for solutions of Abreu's equation},
Collect. Math. \textbf{56} (2005), no.~2, 103--142.

\bibitem[Don05b]{Don05b}
S.~K.~Donaldson,
\textit{Lower bounds on the Calabi functional},
J. Differential Geom. \textbf{70} (2005), no.~3, 453--472.

\bibitem[Don08]{Don08}
S.~K.~Donaldson,
\textit{Extremal metrics on toric surfaces: a continuity method},
J. Differential Geom. \textbf{79} (2008), no.~3, 389--432.

\bibitem[Don09]{Don09}
S.~K.~Donaldson,
\textit{Constant scalar curvature metrics on toric surfaces},
Geom. Funct. Anal. \textbf{19} (2009), no.~1, 83--136.

\bibitem[ELM$^+$06]{ELM+06}
L.~Ein, R.~Lazarsfeld, M.~Musta\c{t}\u{a}, M.~Nakamaye, and M.~Popa,
\textit{Asymptotic invariants of base loci},
Ann. Inst. Fourier (Grenoble) \textbf{56} (2006), no.~6, 1701--1734.

\bibitem[ELM$^+$23]{ELM+23}
L.~Ein, R.~Lazarsfeld, M.~Musta\c{t}\u{a}, M.~Nakamaye, and M.~Popa,
\textit{Erratum to the paper: Asymptotic invariants of base loci},
arXiv:2309.16722.

\bibitem[FT14]{FT14}
M.~Franciosi and E.~Tenni,
\textit{The canonical ring of a $3$-connected curve},
Rend. Lincei Mat. Appl. \textbf{25} (2014), 37--51.

\bibitem[Fuj92]{Fuj92}
A.~Fujiki,
\textit{Moduli spaces of polarized algebraic manifolds and K\"ahler metrics},
Sugaku Expositions \textbf{5} (1992), no.~2, 173--191.

\bibitem[Fuj19]{Fuj19}
K.~Fujita,
\textit{A valuative criterion for uniform K-stability of
$\mathbb Q$-Fano varieties},
J. Reine Angew. Math. \textbf{751} (2019), 309--338.

\bibitem[FO18]{FO18}
K.~Fujita and Y.~Odaka,
\textit{On the K-stability of Fano varieties and anticanonical divisors},
T\^ohoku Math. J. (2) \textbf{70} (2018), no.~4, 511--521.

\bibitem[Fut83]{Fut83}
A.~Futaki,
\textit{An obstruction to the existence of Einstein K\"ahler metrics},
Invent. Math. \textbf{73} (1983), no.~3, 437--443.

\bibitem[Gie82]{Gie82}
D.~Gieseker,
\textit{Lectures on moduli of curves},
Tata Inst. Fund. Res. Lectures on Math. and Phys. \textbf{69},
Tata Inst. Fund. Res., Bombay; Springer-Verlag, Berlin--New York, 1982.

\bibitem[Hat26]{Hat26}
M.~Hattori,
\textit{A decomposition formula for J-stability and its applications},
Michigan Math. J. \textbf{76} (2026), no.~3,
621--659.

\bibitem[He19]{He19}
W.~He,
\textit{On Calabi's extremal metric and properness},
Trans. Amer. Math. Soc. \textbf{372} (2019), no.~8, 5595--5619.

\bibitem[His16]{His16}
T.~Hisamoto,
\textit{Stability and coercivity for toric polarizations},
arXiv:1610.07998.

\bibitem[JSS19]{JSS19}
W.~Jian, Y.~Shi, and J.~Song,
\textit{A remark on constant scalar curvature K\"ahler metrics on minimal
models},
Proc. Amer. Math. Soc. \textbf{147} (2019), no.~8, 3507--3513.

\bibitem[Jos06]{Jos06}
J.~Jost,
\textit{Compact Riemann Surfaces: An Introduction to Contemporary
Mathematics},
3rd ed., Universitext, Springer-Verlag, Berlin, 2006.

\bibitem[Ju$^+$26]{Ju+26} H.~Ju, G.~Gao, J.~Jiang, B.~Wu, Z.~Sun, S.~Liu, L.~Chen, Y.~Wang, Y.~Wang, Z.~Wang, W.~He, P.~Wu, L.~Xiao, R.~Liu, B.~Dai, and B.~Dong, \textit{Automated Conjecture Resolution with Formal Verification}, arXiv:2604.03789.

\bibitem[JY26]{JY26}
S.~Jubert and C.~Yin,
\textit{Relative uniform Yau--Tian--Donaldson correspondence for projective
bundles over a curve},
arXiv:2602.13133.

\bibitem[LS93]{LS93}
C.~LeBrun and S.~R.~Simanca,
\textit{On the K\"ahler classes of extremal metrics},
in: \textit{Geometry and Global Analysis (Sendai, 1993)}, Tohoku Univ.,
Sendai (1993), 255--271.

\bibitem[Li17]{Li17}
C.~Li,
\textit{K-semistability is equivariant volume minimization},
Duke Math. J. \textbf{166} (2017), no.~16, 3147--3218.

\bibitem[Li22a]{Li22a}
C.~Li,
\textit{$G$-uniform stability and K\"ahler--Einstein metrics on Fano
varieties},
Invent. Math. \textbf{227} (2022), no.~2, 661--744.

\bibitem[Li22b]{Li22b}
C.~Li,
\textit{Geodesic rays and stability in the cscK problem},
Ann. Sci. \'Ec. Norm. Sup\'er. (4) \textbf{55} (2022), no.~6, 1529--1574.

\bibitem[LTW21]{LTW21}
C.~Li, G.~Tian, and F.~Wang,
\textit{On the Yau--Tian--Donaldson conjecture for singular Fano varieties},
Comm. Pure Appl. Math. \textbf{74} (2021), no.~8, 1748--1800.

\bibitem[LTW22]{LTW22}
C.~Li, G.~Tian, and F.~Wang,
\textit{The uniform version of Yau--Tian--Donaldson conjecture for singular
Fano varieties},
Peking Math. J. \textbf{5} (2022), no.~2, 383--426.

\bibitem[LX14]{LX14}
C.~Li and C.~Xu,
\textit{Special test configuration and K-stability of Fano varieties},
Ann. of Math. (2) \textbf{180} (2014), no.~1, 197--232.

\bibitem[Liu$^+$26]{Liu+26} J.~Liu, G.~Gao, Z.~Sun, B.~Wu, S.~Liu, J.~Jiang, H.~Ju, L.~Chen, R.~Cheng, X.~Zhang, and B.~Dong, \textit{Danus: Orchestrating Mathematical Reasoning Agents with Fact-Graph Memory}, arXiv:2607.06447.

\bibitem[LXZ22]{LXZ22}
Y.~Liu, C.~Xu, and Z.~Zhuang,
\textit{Finite generation for valuations computing stability thresholds and
applications to K-stability},
Ann. of Math. (2) \textbf{196} (2022), no.~2, 507--566.

\bibitem[Luo98]{Luo98}
H.~Luo,
\textit{Geometric criterion for Gieseker--Mumford stability of polarized
manifolds},
J. Differential Geom. \textbf{49} (1998), no.~3, 577--599.

\bibitem[Mab86]{Mab86}
T.~Mabuchi,
\textit{K-energy maps integrating Futaki invariants},
T\^ohoku Math. J. (2) \textbf{38} (1986), 575--593.

\bibitem[Mab08]{Mab08}
T.~Mabuchi,
\textit{K-stability of constant scalar curvature polarization},
arXiv:0812.4093.

\bibitem[Mab09]{Mab09}
T.~Mabuchi,
\textit{A stronger concept of K-stability},
arXiv:0910.4617.

\bibitem[Mat57]{Mat57}
Y.~Matsushima,
\textit{Sur la structure du groupe d'hom\'eomorphismes analytiques d'une
certaine vari\'et\'e kaehl\'erienne},
Nagoya Math. J. \textbf{11} (1957), 145--150.

\bibitem[MP25]{MP25}
P.~Mesquita-Piccione,
\textit{A non-Archimedean theory of complex spaces and the cscK problem},
Adv. Math. \textbf{481} (2025), Paper No.~110543.

\bibitem[MPWN25]{MPWN25}
P.~Mesquita-Piccione and D.~Witt Nystr\"om,
\textit{A transcendental non-Archimedean Calabi--Yau Theorem with
applications to the cscK problem},
arXiv:2509.09442.

\bibitem[Mum77]{Mum77}
D.~Mumford,
\textit{Stability of projective varieties},
Enseign. Math. (2) \textbf{23} (1977), no.~1--2, 39--110.

\bibitem[NS21]{NS21}
Y.~Nitta and S.~Saito,
\textit{A uniform version of the Yau--Tian--Donaldson correspondence for
extremal K\"ahler metrics on polarized toric manifolds},
arXiv:2110.10386.

\bibitem[Oda12]{Oda12}
Y.~Odaka,
\textit{The Calabi conjecture and K-stability},
Int. Math. Res. Not. IMRN (2012), no.~10, 2272--2288.

\bibitem[Oda13]{Oda13}
Y.~Odaka,
\textit{The GIT stability of polarized varieties via discrepancy},
Ann. of Math. (2) \textbf{177} (2013), no.~2, 645--661.

\bibitem[Oda15]{Oda15}
Y.~Odaka,
\textit{On parametrization, optimization and triviality of test
configurations},
Proc. Amer. Math. Soc. \textbf{143} (2015), no.~1, 25--33.

\bibitem[OAI26]{OAI26}
OpenAI,
\textit{Ten Advances in Mathematics and Theoretical Computer Science},
2026, \texttt{https://cdn.openai.com/pdf/ten-proofs-oai.pdf}.

\bibitem[PS07]{PS07}
D.~H.~Phong and J.~Sturm,
\textit{Test configurations for K-stability and geodesic rays},
J. Symplectic Geom. \textbf{5} (2007), no.~2, 221--247.

\bibitem[PS09]{PS09}
D.~H.~Phong and J.~Sturm,
\textit{Lectures on stability and constant scalar curvature},
in: \textit{Current Developments in Mathematics 2007}, Int. Press,
Somerville, MA, 2009, 101--176.

\bibitem[RT06]{RT06}
J.~Ross and R.~Thomas,
\textit{An obstruction to the existence of constant scalar curvature
K\"ahler metrics},
J. Differential Geom. \textbf{72} (2006), no.~3, 429--466.

\bibitem[RT07]{RT07}
J.~Ross and R.~Thomas,
\textit{A study of the Hilbert--Mumford criterion for the stability of
projective varieties},
J. Algebraic Geom. \textbf{16} (2007), no.~2, 201--255.

\bibitem[Sch$^+$26]{Sch+26}
J.~Schmitt, T.~Gehrunger, J.~Dekoninck, G.~B\'erczi, U.~Kreitner,
L.~Price, and D.~Holmes,
\textit{ProofCouncil: An LLM Agent for Solving Open Mathematical
Problems},
arXiv:2607.09474.

\bibitem[SD18]{SD18}
Z.~Sj\"ostr\"om Dyrefelt,
\textit{K-semistability of cscK manifolds with transcendental cohomology
class},
J. Geom. Anal. \textbf{28} (2018), no.~4, 2927--2960.

\bibitem[SD19]{SD19}
Z.~Sj\"ostr\"om Dyrefelt,
\textit{A partial comparison of stability notions in K\"ahler geometry},
in \textit{Moduli of K-stable varieties} (G.~Codogni, R.~Dervan, and
F.~Viviani, eds.), Springer INdAM Ser., vol.~31, Springer, Cham, 2019,
pp.~103--139.

\bibitem[SD20]{SD20}
Z.~Sj\"ostr\"om Dyrefelt,
\textit{On K-polystability of cscK manifolds with transcendental cohomology
class},
Int. Math. Res. Not. IMRN (2020), no.~9, 2769--2817; with an appendix by
R.~Dervan.

\bibitem[SW08]{SW08}
J.~Song and B.~Weinkove,
\textit{On the convergence and singularities of the J-flow with
applications to the Mabuchi energy},
Comm. Pure Appl. Math. \textbf{61} (2008), no.~2, 210--229.

\bibitem[Sta26]{Sta26}
The Stacks project authors,
\textit{The Stacks project},
published electronically at \url{https://stacks.math.columbia.edu}, 2026.

\bibitem[Sto09]{Sto09}
J.~Stoppa,
\textit{K-stability of constant scalar curvature K\"ahler manifolds},
Adv. Math. \textbf{221} (2009), no.~4, 1397--1408.

\bibitem[Sto11]{Sto11}
J.~Stoppa,
\textit{A note on the definition of K-stability},
arXiv:1111.5826.

\bibitem[SS11]{SS11}
J.~Stoppa and G.~Sz\'ekelyhidi,
\textit{Relative K-stability of extremal metrics},
J. Eur. Math. Soc. (JEMS) \textbf{13} (2011), no.~4, 899--909.

\bibitem[Sze06]{Sze06}
G.~Sz\'ekelyhidi,
\textit{Extremal metrics and K-stability},
Ph.D. thesis, Imperial College London, 2006.

\bibitem[Sze07]{Sze07}
G.~Sz\'ekelyhidi,
\textit{Extremal metrics and K-stability},
Bull. Lond. Math. Soc. \textbf{39} (2007), no.~1, 76--84.

\bibitem[Sze15]{Sze15}
G.~Sz\'ekelyhidi,
\textit{Filtrations and test-configurations},
with an appendix by S.~Boucksom,
Math. Ann. \textbf{362} (2015), no.~1--2, 451--484.

\bibitem[Tia97]{Tia97}
G.~Tian,
\textit{K\"ahler--Einstein metrics with positive scalar curvature},
Invent. Math. \textbf{130} (1997), no.~1, 1--37.

\bibitem[Tia15a]{Tia15a}
G.~Tian,
\textit{Corrigendum: K-stability and K\"ahler--Einstein metrics},
Comm. Pure Appl. Math. \textbf{68} (2015), no.~11, 2082--2083.

\bibitem[Tia15b]{Tia15b}
G.~Tian,
\textit{K-stability and K\"ahler--Einstein metrics},
Comm. Pure Appl. Math. \textbf{68} (2015), no.~7, 1085--1156.

\bibitem[TF98]{TF98}
C.~W.~T{\o}nnesen-Friedman,
\textit{Extremal K\"ahler metrics on minimal ruled surfaces},
J. Reine Angew. Math. \textbf{502} (1998), 175--197.

\bibitem[Tru24]{Tru24}
A.~Trusiani,
\textit{A relative Yau--Tian--Donaldson conjecture and stability thresholds},
Adv. Math. \textbf{441} (2024), Paper No.~109537.

\bibitem[Tru26]{Tru26}
A.~Trusiani,
\textit{A solution to the Yau--Tian--Donaldson conjecture through special
Fujita approximations},
arXiv:2605.30063.

\bibitem[WZ04]{WZ04}
X.-J.~Wang and X.~Zhu,
\textit{K\"ahler--Ricci solitons on toric manifolds with positive first
Chern class},
Adv. Math. \textbf{188} (2004), no.~1, 87--103.

\bibitem[WN12]{WN12}
D.~Witt Nystr\"om,
\textit{Test configurations and Okounkov bodies},
Compos. Math. \textbf{148} (2012), no.~6, 1736--1756.

\bibitem[Xu21]{Xu21}
C.~Xu,
\textit{K-stability of Fano varieties: an algebro-geometric approach},
EMS Surv. Math. Sci. \textbf{8} (2021), no.~1--2, 265--354.

\bibitem[Xu25]{Xu25}
C.~Xu,
\textit{K-stability of Fano Varieties},
New Mathematical Monographs \textbf{50}, Cambridge Univ. Press, Cambridge,
2025.

\bibitem[Yau93]{Yau93}
S.-T.~Yau,
\textit{Open problems in geometry},
in: \textit{Differential Geometry: Partial Differential Equations on
Manifolds}, Proc. Sympos. Pure Math. \textbf{54}, Part~1, Amer. Math. Soc.,
Providence, RI (1993), 1--28.

\bibitem[Yu19]{Yu19}
C.-F.~Yu,
\textit{Chow's theorem for semi-abelian varieties and bounds for splitting
fields of algebraic tori},
Acta Math. Sin. (Engl. Ser.) \textbf{35} (2019), no.~9, 1453--1463.

\bibitem[Zar00]{Zar00}
Yu.~G.~Zarhin,
\textit{Hyperelliptic Jacobians without complex multiplication},
Math. Res. Lett. \textbf{7} (2000), no.~1, 123--132.

\bibitem[Zha96]{Zha96}
S.~Zhang,
\textit{Heights and reductions of semi-stable varieties},
Compos. Math. \textbf{104} (1996), no.~1, 77--105.

\bibitem[Zha24]{Zha24}
K.~Zhang,
\textit{A quantization proof of the uniform Yau--Tian--Donaldson conjecture},
J. Eur. Math. Soc. (JEMS) \textbf{26} (2024), no.~12, 4763--4778.

\bibitem[Zhu21]{Zhu21}
Z.~Zhuang,
\textit{Optimal destabilizing centers and equivariant K-stability},
Invent. Math. \textbf{226} (2021), no.~1, 195--223.

\end{thebibliography}
\end{document}